\documentclass[a4paper, 12pt]{amsart}
\usepackage{amsmath,mathrsfs,amsthm}
\usepackage{txfonts,rotating,tikz}
\usepackage{esvect}
\usepackage{dsfont}
\usetikzlibrary{decorations.pathmorphing,decorations.markings,decorations.pathreplacing,arrows,shapes}
\usepackage{float}
\usepackage{enumitem}

\usepackage[all]{xy}
\usepackage{nicefrac}

\usepackage[all]{xy}
\usepackage{color} 
\usepackage{hyperref}

\newtheorem{theorem}{Theorem}[section]

\newtheorem{lemma}[theorem]{Lemma}
\newtheorem{proposition}[theorem]{Proposition}
\newtheorem{corollary}[theorem]{Corollary}

\theoremstyle{definition}

\newtheorem{remark}[theorem]{Remark}
\newtheorem{definition}[theorem]{Definition}
\newtheorem{example}[theorem]{Example}
\newcommand{\mGr}{\mathscr{G}\!r}
\DeclareMathOperator{\SL}{SL}

\DeclareMathOperator{\Gr}{\mathrm{Gr}}

\newcommand{\mb}{\mathbb}
\newcommand{\mc}{\mathcal}
\newcommand{\mf}{\mathfrak}

\newcommand{\Fl}{\mathrm{Fl}}
\newcommand{\dbp}[1]{(\!(#1)\!)}
\newcommand{\dbb}[1]{[\![#1]\!]}

\newcommand{\rk}{{\mathrm{k}}}

\newcommand{\la}{\langle}
\newcommand{\ra}{\rangle}

\newcommand{\beq}{\begin{equation}}
\newcommand{\eeq}{\end{equation}}
\newcommand{\bt}{\begin{theorem}}
\newcommand{\et}{\end{theorem}}
\newcommand{\bde}{\begin{definition}}
\newcommand{\ede}{\end{definition}}
\newcommand{\bpr}{\begin{proposition}}
\newcommand{\epr}{\end{proposition}}
\newcommand{\ble}{\begin{lemma}}
\newcommand{\ele}{\end{lemma}}
\newcommand{\bco}{\begin{corollary}}
\newcommand{\eco}{\end{corollary}}
\newcommand{\bre}{\begin{remark}}
\newcommand{\ere}{\end{remark}}
\newcommand{\bpf}{\begin{proof}}
\newcommand{\epf}{\end{proof}}
\newcommand{\val}{\mathrm{val}}

\newcommand{\sx}{{\color{blue}s_{1}}}
\newcommand{\sy}{{\color{blue}s_{3}}}
\renewcommand{\sp}{{\color{red}s_{3}}}
\newcommand{\sz}{{\color{blue}s_{4}}}
\newcommand{\so}{{\color{red}s_{2}}}

\makeatletter
\newcommand{\xRightarrow}[2][]{\ext@arrow 0359\Rightarrowfill@{#1}{#2}}
\newcommand{\xLeftarrow}[2][]{\ext@arrow 0359\Leftarrowfill@{#1}{#2}}

\makeatother

\renewcommand{\doteq}{\underset{\rm val}{\sim}}

\def\varequals#1{\mathrel{\leaders\hbox to3pt{$=$\hss}\hskip#1=}}

\begin{document}

\title{Combinatorics of Mirkovi\'{c}-Vilonen cycles in twisted affine Grassmannians   }

\author{Jiuzu Hong} 
\address{
Department of Mathematics, University of North Carolina at Chapel Hill, Chapel Hill, NC 27599-3250, U.S.A.}
\email{jiuzu@email.unc.edu}

\author{Linhui Shen}
\address{
Department of Mathematics, Michigan State University, East Lansing, MI 48824, U.S.A.}
\email{shenlin1@msu.edu}

\author{Daping Weng} 
\address{Department of Mathematics, University of North Carolina at Chapel Hill, Chapel Hill, NC 27599-3250, U.S.A.}
\email{dweng@unc.edu}

\begin{abstract}
We investigate the explicit geometric and combinatorial structure of Mirkovi\'c-Vilonen (MV) cycles in twisted affine Grassmannians associated with special parahoric group schemes. Using the twist map on configuration spaces of decorated flags, we obtain an explicit parametrization of twisted MV cycles by Lusztig data. We show that the changes between reduced words are governed by the usual Lusztig transformations  for the reductive group arising from ramified geometric Satake correspondence. As a consequence, the moment polytopes of twisted MV cycles coincide with the usual MV polytopes for this group. Our construction applies to all special parahoric subgroups in the
tamely ramified setting.
   
\end{abstract}

\maketitle

\tableofcontents

\section{Introduction}

Mirkovi\'c-Vilonen cycles (MV cycles) are certain subvarieties of the affine Grassmannian of a reductive group $G$. They provide a natural basis for the representations of the Langlands dual group $G^\vee$ via the geometric Satake correspondence. Moreover, MV cycles exhibit a very rich combinatorial structure, which is elegantly described using MV polytopes, Berenstein--Zelevinsky (BZ) data, and Lusztig data, as developed extensively in the works of Anderson \cite{An} and Kamnitzer \cite{Kam}. They are also closely related to the combinatorics of canonical bases. Around the same time, an alternative approach to describe MV cycles via LS galleries on the Bruhat--Tits building was developed by Gaussent and Littelmann \cite{gaussent2005ls}.

In the equal-characteristic and tamely ramified setting, let $\mathscr{G}$ be a special parahoric group scheme and let $\mGr$ denote its associated twisted affine Grassmannian. 
When $\mathscr{G}$ is special, $\mGr$ shares many similar properties with the usual affine Grassmannian. Zhu \cite{Zh} proved the ramified geometric Satake correspondence for $\mGr$: the category of spherical perverse sheaves on $\mGr$ is equivalent to the category of representations of a reductive group $M$, which is determined by the loop group of $\mathscr{G}$. Unlike the original approach of Mirkovi\'c and Vilonen \cite{MV}, which relies essentially on MV cycles, Zhu's approach uses nearby cycles to deduce the ramified version of the geometric Satake equivalence from the usual geometric Satake equivalence of $\mGr$. More recently, the ramified geometric Satake equivalence has been extended to sheaves over fields of characteristic $p$ in \cite{ALRS} and has also been developed in the motivic setting \cite{hove2026}, in which van den Hove generalized the LS gallery approach of MV cycles to the setting of twisted affine Grassmannians.

In this work, we follow the approach of Anderson and Kamnitzer by studying MV cycles from the perspective of MV polytopes. To make everything explicit, following \cite {BH,DH}, we first realize the special parahoric group scheme $\mathscr{G}$ equivariantly, building from a split reductive group $G$ and a finite-order automorphism on $G$. The twisted affine Grassmannian $\mGr$ can then be realized as the quotient space $G(\mathcal{K})^\sigma / G(\mathcal{O})^\sigma$, where $\sigma$ is a certain special finite-order automorphism of $G$ that also acts on the field of Laurent series $\mathcal{K}$ as a field automorphism. This explicit realization allows us to utilize the Berenstein--Zelevinsky theory of generalized minors of $G$.

Let $\overline{\mGr}_\lambda$ be a spherical Schubert variety of $\mGr$, and let $S_\mu$ (resp.$T_\mu$) denote a semi-infinite cell (resp. opposite semi-infinite cell) in $\mGr$. By the ramified geometric Satake \cite{Zh}, the components $\overline{\mGr}_\lambda\cap T_\mu$ give rise to a basis of the weight space $V(\lambda)_\mu$ of the highest weight representation $V(\lambda)$ of $M$, cf.\,Theorem \ref{thm_ramified_Satake}. Furthermore, by Corollary \ref{cor_MV_Gr} the components of $\overline{S_\lambda\cap T_\mu}$ that are contained in $\overline{\mGr}_\lambda$ are exactly the components of $\overline{ \mGr_\lambda  \cap T_\mu}= \overline{  \overline{\mGr}_\lambda\cap T_\mu  } $. Following the strategy in \cite{Kam}, instead of studying the components of $\overline{\mGr}_\lambda\cap T_\mu$, we study the components of $S_\lambda\cap T_\mu$. Up to translation, it suffices to study the components of $S_{\lambda-\mu}\cap T_0$. 

Suppose $\mu$ is a non-negative linear combination of simple roots of $M$.  Fix a reduced word $\imath_\bullet$ of the longest element $w_0$ of the Weyl group $W_M$ of $M$. By Theorem \ref{thm_MV_constr}, we have the following disjoint decomposition of equi-dimensional strata
\[ S_\mu\cap T_0= \bigsqcup_{n_\bullet} A^{\imath_\bullet}(n_\bullet), \]
where the disjoint union is over the Lusztig data $n_\bullet$ and $A^{\imath_\bullet}(n_\bullet)$ is the intersection of certain semi-infinite cells defined in (\ref{eq_A_i(n)}). We prove in Theorem \ref{thm_MV_constr} that the closure of each $A^{\imath_\bullet}(n_\bullet)$ is an irreducible component of $S_\mu\cap T_0$.

In this paper, we give an explicit construction of the strata $A^{\imath_\bullet}(n_\bullet)$ from any given Lusztig data $(\imath_\bullet,n_\bullet)$, and prove the following main result.


\begin{theorem}[Theorem \ref{cor_MV_Lusztig_transf}]
For any fixed reduced word $\imath_\bullet$, every irreducible component of $\overline{S_\mu\cap T_0}$ arises as $\overline{A^{\imath_\bullet}(n_\bullet)}$ for some integer tuple $n_\bullet$. Furthermore, 
for any two reduced words $\imath_\bullet$ and $ \imath'_\bullet$, if the Lusztig data $(\imath_\bullet, n_\bullet)$ and $(\imath'_\bullet, n'_\bullet)$ are related by Lusztig transform, then $\overline{A^{\imath_\bullet}(n_\bullet)}= \overline{ A^{\imath'_\bullet}(n'_\bullet)}$. 
\end{theorem}

As a consequence, the moment polytope of $\overline{A^{\imath_\bullet}(n_\bullet)}$ exactly coincides with the usual MV polytope of $M$ arising from the affine Grassmannian $\Gr_{M^\vee}$ of $M^\vee$ via the usual geometric Satake, as studied in \cite{Kam,An}. 

Now, we explain our construction of $A^{\imath_\bullet}(n_\bullet)$. We consider a space $\mc{A}^{\imath_\bullet}(n_\bullet)$ consisting of tuples $a_\bullet$ of formal power series, which we call \emph{twisted Lusztig coordinates}, whose entries have valuations that are equal to some appropriate rescalings of $n_\bullet$ (Definition \ref{defn: A}). Let $y$ be the product of Lusztig unipotent elements with arguments $a_\bullet$ according to the reduced word $\imath_\bullet$. By applying the inverse twist map $\eta_{w_0}^{-1}$ to $y$, we show in Theorem \ref{prop: surjective} that the image of $\mc{A}^{\imath_\bullet}(n_\bullet)$ is precisely $A^{\imath_\bullet}(n_\bullet)$. Since $\mc{A}^{\imath_\bullet}(n_\bullet)$ is irreducible, its image $A^{\imath_\bullet}(n_\bullet)$ is irreducible as well, which implies that $\overline{A^{\imath_\bullet}(n_\bullet)}$ is an irreducible component of $\overline{S_\mu\cap T_0}$.

\medskip

Finally, we remark on the new subtleties and novelties in the twisted case compared to the previous work \cite{Kam} in the untwisted case:
\begin{itemize}
\item We need to fine-tune the space $\mc{A}^{\imath_\bullet}(n_\bullet)$ so that under the twist map $\eta_{w_0}^{-1}$, it lands in $A^{\imath_\bullet}(n_\bullet)$. It is quite intricate to construct $\mc{A}^{\imath_\bullet}(n_\bullet)$ for the twisted loop group of type $A_{2l}^{(2)}$, in which case there are two special parahoric subgroups, one is absolutely special, and the other one is not (in the sense of \cite{HainesRicharz2020}).

\item For the non-absolutely special parahoric subgroups of $A_{2l}^{(2)}$, we need to adjust the transposition map in the definition of the twist map to make the twist map compatible with the automorphism $\sigma$.

\item Although the defining equations on valuations for $A^{\imath_\bullet}(n_\bullet)$ are governed by the folded Cartan data $a_{\imath\jmath}$ (Eq. \eqref{eq: folded cartan}), the Lusztig transform between $A^{\imath_\bullet}(n_\bullet)$ and $A^{\imath'_\bullet}(n'_\bullet)$ between different reduced words $\imath_\bullet$ and $\imath'_\bullet$ is governed by the Cartan data $c_{\imath\jmath}$ (Eq. \eqref{eq: cartan}) of $M$; in particular, these two Cartan data are not equal in the case of $A_{2l}^{(2)}$. 

\item We give new geometric proofs of some results of Berenstein and Zelevinsky \cite{BZ} on relating Lusztig data and generalized minors in \S \ref{ss: conf space}. 

\item We work out the explicit Lusztig transformation in the $G_2$ case in Propositions \ref{prop: G2} and \ref{prop:G2'}, which was absent in \cite{Kam}. 

\item We work over an algebraically closed field $\rm{k}$ which is allowed to be of positive characteristic. Consequently, we must address certain subtleties due to modular representation theory.  These include the definition of the discrete function $D_{w\Lambda_\imath}$ on $\mGr$ in (\ref{D_def}), as well as Lemma \ref{lem: representative}, which is crucial for proving the surjectivity of the map from $\mc{A}^{\imath_\bullet}(n_\bullet)$ to $A^{\imath_\bullet}(n_\bullet)$.
\end{itemize}

\vspace{0.2cm}

{\bf Acknowledgement}.
This project was initiated in November 2022 and developed through several stages of work, with a complete draft ready by June 2026. In September 2026, during the final stage of this work, we used OpenAI's ChatGPT for Academic Researchers for proofreading. It suggested several corrections and improvements, all of which were independently verified by the authors. 

We would like to thank Tom Haines for pointing out an error in a previous version. J. Hong and L. Shen gratefully acknowledge support from their respective Simons Travel Grants (MPS-TSM-00007468, MPS-TSM-00026110).

\section{Geometry of twisted affine Grassmannians}
In Section \ref{special parahoric subgroups}, we review the special parahoric subgroups of twisted loop groups from the perspective of  Bruhat-Tits theory, and we provide their  equivariant constructions following \cite{DH,HY2024}.  
In Section \ref{Twisted MV cycles}, we review the ramified geometric Satake correspondence and explain how the study of components of $\overline{\mGr}_\lambda\cap T_\mu$ can be reduced to the study of components of $S_\lambda\cap T_\mu$.  
In Section \ref{twisted ggms strata}, we introduce the twisted GGMS strata and the 
$D$-functions on the twisted loop groups. We give a characterization of the 
strata $A^{\imath_\bullet}(n_\bullet)$ in terms of these $D$-functions. 

\smallskip

Let $\tau$ be an action of finite order on a set $X$. Throughout this paper, we denote by $X^\tau$ the set of $\tau$-fixed elements of $X$, and by $X_\tau=X/\langle\tau\rangle$ the quotient set of $\tau$-orbits in $X$.

\subsection{Special parahoric subgroups}

\label{special parahoric subgroups}


Let $G$ be a simple algebraic group  over an algebraically closed field $\rk$ of characteristic $p$, where $p$ is either zero or a prime. Let $\tau$ be a nontrivial diagram automorphism on $G$ of order $r$ that preserves a Borel subgroup $B$ and a maximal torus $H\subset B$. Note that the order $r$ can only be $2$ or $3$.  

Let $\{ \alpha_i \mid i\in I\}$ be the set of simple roots of $G$ with respect to $(B, H)$. Let $\{x_{i}, y_{i}\mid i\in I \}$ be a pinning of $G$ associated with $\{ \alpha_i \mid i\in I\}$. Let $X^*(H)$ (resp. $X_*(H)$) denote the weight (resp. coweight) lattice of $H$. The automorphism $\tau$ induces natural actions $\tau$ on $X^*(H), X_*(H), I $, and the set   $ \{\alpha_i  \mid i\in I\}  $. The action $\tau$ preserves the pinning $\{x_i, y_i\mid i\in I \}$ in the sense that 
\[\tau(x_{i} (a) )= x_{\tau(i)}(a) , \quad \tau(y_{i} (a) )= y_{\tau(i)}(a), \quad \text{for any } a\in \rk. \]

In this paper, we assume that $p \not= r$ and that $G$ is either simply-connected or of adjoint type.  Let $G^\tau$ be the subgroup of $\tau$-fixed points of $G$. Consequently, $B^\tau$ is a Borel subgroup of $G^\tau$, and $H^\tau$ is a maximal torus of $G^\tau$ contained in $B^\tau$. Since $G$ is simply-connected or adjoint, $G^\tau, B^\tau, H^\tau$ are all connected.  There are natural isomorphisms 
\[
X^*(H^\tau)\simeq X^*(H)_\tau, \quad X_*(H^\tau)\simeq X_*(H)^\tau.\]
Let $I_\tau$ be the set of $\tau$-orbits in $I$. For each $\imath \in I_\tau$, we associate a simple root $\overline{\alpha}_{\imath}:=  \alpha_i |_{H^\tau}$ of $G^\tau$, where $i\in \imath$. Therefore, the set  $I_\tau$ coincides with the set of simple roots of $G^\tau$.

\smallskip

Let $\mc{K}_z=\rk\dbp{z}$ be the field of Laurent series over $\rk$. Let $\mc{O}_z=\rk\dbb{z}$ be the ring of formal power series over $\rk$. Fix a primitive $r$-th root of unity $\epsilon_r$.  We consider the action  $\tau$ on $\mc{K}_z$ and $\mc{O}_z$ defined by
\[\tau|_\rk={\rm Id}_\rk, \quad \mbox{and~}\tau(z)=\epsilon_r^{-1}z.\]
Let $\tau$ act diagonally on $G(\mc{K}_z)$ and $G(\mc{O}_z)$. Denote by $G( \mc{K}_z )^\tau$ and $G(\mc{O}_z)^\tau$ the  subgroups of $\tau$-fixed points.  The twisted loop group $G(\mc{K}_z)^\tau$ is associated with an affine Dynkin diagram $X_N^{(r)}$, whose affine simple roots are indexed by $\hat{I}_\tau=I_\tau \cup \{o\}$. We regard $X_*(H)^\tau_\mb{R}:= X_*(H)^\tau \otimes _\mb{Z} \mb{R}$ as an apartment of the Bruhat-Tits building of $G(\mc{K}_z)^\tau$. Let $\mf{C}$ be the fundamental alcove in $X_*(H)^\tau_ \mb{R}$. The vertices of $\mf{C}$ are naturally indexed by $\hat{I}_\tau $. 

Each vertex $\kappa$ 
 of $\mf{C}$ corresponds to a maximal parahoric subgroup $\mc{P}_\kappa$ of $G(\mc{K}_z)^\tau$. Following \cite[\S 1.9]{Ti},  we say that $\kappa$ and  $\mc{P}_\kappa$ are {\it special}  if the Weyl group of the Dynkin diagram associated with the subset $\hat{I}_\tau \backslash \{\kappa\}$ is isomorphic to the Weyl group of $G^\tau$.  Let $\mc{S}$ denote the set of special vertices of $\mf{C}$.  Following the Bourbaki labeling, we get
 \[ \mc{S}=  \begin{cases}
  \{ 0,  \check{\omega}_1+ \check{\omega}_{2\ell-1} \} \quad   \text{if } X_N^{(r)}=         A_{2\ell-1}^{(2)}  \\
    \{0,   \frac{1}{2}(\check{\omega}_\ell+ \check{\omega}_{\ell+1}) \} \quad      \text{if } X_N^{(r)}= A_{2\ell}^{(2)}   \\
     \{ 0,  \check{\omega}_\ell+ \check{\omega}_{\ell+1} \} \quad   \text{if } X_N^{(r)}= D_{\ell+1}^{(2)}  \\
     \{0\} \quad \quad \text{if }   X_N^{(r)}= E_6^{(2)}, D_4^{(3)}
 \end{cases} \]
 where  $\check{\omega}_i$ is the fundamental coweight of $G$ corresponding to $i\in I$.
Define the partial affine flag variety 
\[  \Fl_\kappa= G(\mc{K}_z)^\tau/ \mc{P}_\kappa.   \] 
Note that $\Fl_\kappa$ shares similar properties with the usual affine Grassmannians, cf.\,\cite{Zh}.  

 \smallskip
 
 Let $G_{\rm ad}=G/ Z(G)$ and $H_{\rm ad}=H/Z(G)$, where $Z(G)$ is the center of $G$.  Fix a special vertex $\kappa\in \mc{S}$.  Let $m$ be the least multiple of $r$ such that 
 \begin{equation}
 \label{eq_h}
    h:= \frac{m}{r}\kappa \in X_*(H_{\rm ad})^\tau.
 \end{equation}
  Note that $h=\kappa$ except when $X_N^{(r)}=A^{(2)}_{2\ell}$. It is easy to see that 
 \[    m=\begin{cases}
 2   \quad \text{ if }    X_N^{(r)}= A_{2\ell-1}^{(2)}, D_{\ell+1}^{(2)}, E_6^{(2)},\text{ or } X_N^{(r)}= A_{2\ell}^{(2)} \text{ and } \kappa=0 \\
 4 \quad  \text{ if } X_N^{(r)}= A_{2\ell}^{(2)} \text{ and } \kappa\not=0    \\
 3  \quad  \text{ if } X_N^{(r)}=D_4^{(3)}
 \end{cases} . \] 
In particular, if $\kappa\not=0$, then $X_N^{(r)}$ can only be either $A_{2\ell-1}^{(2)} , A_{2\ell}^{(2)} $ or $D_{\ell+1}^{(2)}$, and
 \begin{equation}\label{eq: h} h=  \begin{cases}
 \check{\omega}_1+ \check{\omega}_{2\ell-1}   \quad   \text{if } X_N^{(r)}=         A_{2\ell-1}^{(2)}  \\
   \check{\omega}_\ell+ \check{\omega}_{\ell+1} \quad      \text{if } X_N^{(r)}= A_{2\ell}^{(2)}   \\
    \check{\omega}_\ell+ \check{\omega}_{\ell+1} \quad   \text{if } X_N^{(r)}= D_{\ell+1}^{(2)}  
 \end{cases} . 
 \end{equation}
Following \cite{HainesRicharz2020}, the parahoric subgrouop $\mathcal{P}_\kappa$ is {\it absolutely special} when $(X_N^{(r)},m)\not= (A_{2\ell}^{(2)}, 4)$.
 
 Let $\epsilon$ be a primitive $m$-th root of unity in $\rk$ such that $\epsilon^{m/r}=\epsilon_r$.  We regard the $h$ in \eqref{eq_h} as a cocharacter of $H_{\rm ad}$, and associate to it the element $\epsilon^{h}:=h(\epsilon)\in H_{\rm ad}$. 
  We define the {\it special} automorphism $\sigma$ of $G$ such that
 \begin{equation}
 \sigma:= \tau\circ {\rm Ad}_{ \epsilon^{h}  }.
 \end{equation}
By \cite[Lemma 2.3]{HY2024},  $\sigma$ is of order $m$.

Let $\mc{K}=\rk\dbp{t}$ and $\mc{O}=\rk\dbb{t}$ with an action of $\gamma$ such that $\gamma|_{\rk}={\rm Id}|_{\rk}$ and $\gamma(t)=\epsilon^{-1}t$. We regard $\mc{K}_z$ as a subfield of $\mc{K}$, by identifying $z=t^{m/r}$. 
Let $ \sigma\circ \gamma$ be the automorphism  acting diagonally on $G(\mc{K})$ and $G(\mc{O})$. Without ambiguity, we still denote it by $\sigma$. 
Explicitly, for each $i\in I$, we have
\begin{equation}  \sigma(x_i(f)  )= x_{\tau(i)} ( \epsilon^{\alpha_i(h)}  \gamma(f) ), \quad \sigma(y_i(f)  )= y_{\tau(i)} ( \epsilon^{-\alpha_i(h)}  \gamma(f) ) , \quad \text{ for any } f\in \mc{K} .\end{equation}

The twisted affine Grassmannian is defined as
\begin{equation}\label{eq_twisted aff_def}
   \mGr:=G(\mc{K})^\sigma/G(\mc{O})^\sigma . 
\end{equation}


\begin{lemma}
If $G$ is either simply-connected or of adjoint type, then $G^\sigma$ is connected. 
\end{lemma}
\begin{proof} 
If $G$ is simply-connected or adjoint and $\tau$ is a diagram automorphism of $G$, then it is well known that $G^\tau$ is connected.

If $\kappa=0$, then $\sigma=\tau$ is a diagram automorphism. The Lemma follows.

If $\kappa\not=0$ and $X_N^{(r)}\not= A_{2\ell}^{(2)}$, then there is an isomorphism $G^\tau\simeq G^\sigma$, given by $g\mapsto {\rm Ad}_{\epsilon^\mu}(g)$, where 
 \[ \mu=  \begin{cases}
 \check{\omega}_1 \quad   \text{if } X_N^{(r)}=         A_{2\ell-1}^{(2)}  \\
   \check{\omega}_\ell \quad   \text{if } X_N^{(r)}= D_{\ell+1}^{(2)}  
    \end{cases}  .\]
For the final case $X_N^{(r)}=A_{2\ell}^{(2)}$, we resort to the following fact: when $G$ is simply-connected or $G$ is adjoint of type $A_{2n}, D_4, E_6$, $G^\sigma$ is always connected (cf.\,\cite[\S 8.1, \S 9.8]{St}). 
\end{proof}

 By \cite[Theorem 4.1.2]{DH} and \cite[Theorem 2.4]{HY2024}, there exists an isomorphism ${\rm Ad}_{t^h}: G(\mc{K}_z)^\tau\simeq G(\mc{K})^\sigma$ such that ${\rm Ad}_{t^{h}} (\mc{P}_\kappa )= G(\mc{O})^\sigma$. It therefore induces an isomorphism of ind-varieties
 \begin{equation} \mathrm{Ad}_{t^{h}}:  \Fl_\kappa\simeq \mGr .  \end{equation}
Keeping the isomorphism in mind,  we will  work with $\mGr$ rather than $\Fl_\kappa$ in this paper.

\subsection{Twisted MV cycles}
\label{Twisted MV cycles}

 Let $\check{G}$ denote the Langlands dual group of $G$. The action of $\tau$ on $\check{G}$ as a diagram automorphism is induced by its action on $G$. Let $\check{G}^\tau$ be the corresponding $\tau$-fixed point subgroup of $\check{G}$. The dual torus $\check{H}$ is the maximal torus of $\check{G}$.  There is a natural identification $X_*(H)\simeq X^*(\check{H})$, and $X_*(H)_\tau\simeq X^*(\check{H}^\tau)$. Let $X_*(H)_\tau^+$ be the subset of $X_*(H)_\tau$ corresponding to the set $X^*(\check{H}^\tau)^+$ of dominant weights of $\check{G}^\tau$.  Note that the image of $X_*(H)^+$ in $X_*(H)_\tau$ is contained in $X_*(H)^+_\tau$.  For $\imath\in I_\tau$, let $i\in \imath$ be the smallest index with respect to the Bourbaki labelling. Let $\gamma_\imath$ be the image of $\check{\alpha}_i$ in $X_*(H)_\tau$. Then, $\{ \gamma_\imath \,|\, \imath\in I_\tau\}$ corresponds to the set of simple roots of $\check{G}^\tau$.  For $\mu, \lambda\in X_*(H)_\tau$, we have the standard dominant order on $X_*(H)^+_\tau$, i.e. $\mu\preceq \lambda$ if $\lambda-\mu$ is a non-negative summation of $\{ \gamma_\imath\mid  \imath\in I_\tau  \}$.
 \begin{remark}\label{rmk:surj}
 There is a direct description of $X_*(H)_\tau^+$ via the quasi-split group associated to the twiste loop group $G(\mathcal{K}_z)^\tau$, see \cite[Section 2.4]{ALRS}. Moreover, the map $X_*(H)^+\to X_*(H)^+_\tau$ is surjective when $G$ is adjoint, but NOT in general. See \cite[Lemma 2.6, Remark 2.7]{ALRS}. 
 \end{remark}

For $\lambda\in X_*(H)_\tau$,  pick a lift $\tilde{\lambda}\in X_*(H)$ of $\lambda$.  Set 
$n^{\tilde{\lambda}}:= \prod_{s=1}^m \sigma^s(t^{\tilde{\lambda}})$ and $e_{\lambda}:= n^{\tilde{\lambda}}\cdot e$, where $e$ is the base point in $\mGr$. Note that $e_\lambda$ does not depend on the choice of $\tilde{\lambda}$. Set $\mGr_{\lambda}:= G(\mc{O})^\sigma\cdot e_\lambda$. We have the Cartan decomposition
\[  \mGr= \bigsqcup_{\lambda\in X_*(H)_\tau^+}  \mGr_\lambda.    \]
Let $\overline{\mGr}_\lambda$ be the closure of $ \mGr_\lambda$. Note that $\mGr_\mu\subset \overline{\mGr_\lambda} $ if and only if $\mu\preceq \lambda$.  By \cite[Corollary 4.3]{BH}, $((\mGr)^{H^\sigma} )_{\rm red}=\{e_\mu \mid \mu\in X_*(H)_\tau \}$, where $()_{\rm red}$ denotes the reduced structure. 

Let $W^\tau$ be the $\tau$-fixed point subgroup of $W$. Then $W^\tau$ can be identified with the Weyl group of $G^\sigma$; that is, $W^\tau\simeq N_{G^\sigma}(H^\sigma)/ H^\sigma$. Let $N$ be the unipotent radical of $B$. Let $N^-$ be the unipotent radical of the opposite Borel subgroup $B^-$ relative to $(B,H)$. 
Set $N^-(\mc{K})^\sigma_w:= \dot{w} N^-(\mc{K})^\sigma \dot{w}^{-1}$, where $\dot{w}$ is a representative of $w$ in $N_{G^\sigma}(H^\sigma)$. 
For $\mu\in X_*(H)_\tau$ and $w\in W^\tau$,  set 
$$T^w_\mu:= N^-(\mc{K})^\sigma_w \cdot e_\mu.$$ 
For simplicity, we denote $$T_\mu=T_\mu^e,  \qquad S_\mu=T^{w_0}_\mu=N(\mc{K})^\sigma \cdot e_\mu.$$
For each $w\in W^\tau$, we have the following Iwasawa decomposition
\[ \mGr:= \bigsqcup_{\mu\in X_*(H)_\tau}  T^w_\mu. \]
Let $\overline{T^w_\mu}$ be the closure of $T^w_\mu$. Then 
\begin{proposition}[{{\cite[Prop. 5.4]{AGLTR}}}]
\label{prop_semi_closure}
$T^w_\nu\subset \overline{T^w_\mu}$ if and only if $w^{-1}\nu\succeq w^{-1} \mu$. 
\end{proposition}

We  write $\nu\succeq_w \mu$ if $w^{-1}\nu\succeq w^{-1}\mu$. By \cite[Lemma 3.7]{ALRS}, 
$\overline{\mGr}_\lambda\cap T_\mu $ is nonempty if and only if $\mu_+\preceq \lambda$, where $\mu_+$ is the unique dominant element in  the orbit $W^\tau\cdot \mu$. In this case, $\overline{\mGr}_\lambda\cap T_\mu $ is of pure dimension $\langle \lambda-\mu, \rho \rangle$, where $\rho$ is half the sum of all positive roots of $G$. In general, for $w\in W^\tau$, if $\overline{\mGr}_\lambda\cap T^w_\mu\not= \emptyset$, then $\overline{\mGr}_\lambda\cap T^w_\mu $ is of pure dimension $\la \lambda- w^{-1}(\mu), \rho \ra$.

\smallskip

Let $\mb{G}_m$ be the multiplicative group acting on $\mGr$ by loop rotation.

\begin{lemma}
\label{key_lem}
Let $Z$ be a $\mb{G}_m\times H^\sigma$-orbit on $\mGr$. Then $Z\subset \mGr_\lambda$ for some $e_\lambda\in \overline{Z}$. 
\end{lemma}

\begin{proof}
The proof of \cite[Lemma 6]{An} carries over to $\mGr$. The only difference is that, for $\eta\in X_*(H)^+_\tau$, $\mGr_\eta$ is a vector bundle over the core $G^\sigma\cdot e_\eta$ with respect to the $\mb{G}_m$-action. 
\end{proof}

\begin{proposition}
\label{T_S_prop}
The intersection $S_\lambda\cap T_\mu$ is nonempty if and only if $\mu\preceq \lambda$. In this case, it is of pure dimension $\la \lambda-\mu, \rho \ra$.  Moreover, when $\mu\preceq \lambda$ and $\lambda,\mu$ are sufficiently dominant, then $\overline{S_\lambda\cap T_\mu}= \overline{\mGr_\lambda\cap T_\mu }$.  
\end{proposition}
\begin{proof}
For any $\eta\in X_*(H)_\tau$,  we have the following translation invariance: 
\[ S_{\lambda}\simeq S_{\lambda+\eta}, \, T_{\mu}\simeq T_{\mu+\eta},  \, \text{ and } \, S_\lambda \cap T_\mu\simeq S_{\lambda+\eta} \cap T_{\mu+\eta} . \]
We may assume $\lambda, \mu$ are sufficiently dominant so that for any $\nu\in X_*(H)_\tau$, $\mu\preceq \nu \preceq \lambda$ implies that $\nu$ is dominant.  Let $X=\overline{S_\lambda}\cap \overline{T_\mu}$.  Then, $X$ is a $\mb{G}_m\times H^\sigma$-stable closed subset in $\mGr$. By the closure relation of semi-infinite orbits (cf.\,Proposition \ref{prop_semi_closure}), $X^{H^\sigma}$ is contained in $\{ e_\nu \mid  \mu\preceq \nu\preceq \lambda \}$. In particular, if $S_\lambda\cap T_\mu$ is nonempty, then $\mu\preceq \lambda$. 
By Lemma \ref{key_lem},  
$$X\subseteq \bigcup_{e_\nu\in X^{H^\sigma} }\overline{\mGr}_\nu \subseteq \bigcup_{\mu\preceq \nu \preceq \lambda }\overline{\mGr}_\nu.$$
Then, by the assumption on $\lambda$ and $\mu$, any $\nu$  above is dominant. Thus, $X\subseteq \overline{\mGr}_\lambda$, and then  $\overline{S_\lambda}\cap T_\mu\subseteq \overline{\mGr}_\lambda\cap T_\mu$.

Since $\dim (\overline{\mGr}_\lambda\cap S_\lambda)= \dim \overline{\mGr}_\lambda$, $\overline{\mGr}_\lambda\cap S_\lambda$ is dense in $\overline{\mGr}_\lambda$. Thus, $\overline{\mGr}_\lambda\subset \overline{S}_\lambda$. It follows that, $\overline{\mGr}_\lambda\cap T_\mu \subseteq \overline{S}_\lambda\cap T_\mu$. 
Therefore, if $\mu\preceq \lambda$ then
 $ \overline{S}_\lambda\cap T_\mu=\overline{\mGr}_\lambda\cap T_\mu$ is nonempty and is of pure dimension $\la \lambda-\mu, \rho \ra$. By dimension consideration, $ S_\lambda\cap T_\mu$ is open and dense in  $\overline{S}_\lambda\cap T_\mu$. It further follows that, $\overline{S_\lambda\cap T_\mu }=\overline{ \overline{S}_\lambda\cap T_\mu}$. On the other hand, ${\mGr}_\lambda\cap T_\mu$ is also open and dense in $\overline{\mGr}_\lambda\cap T_\mu$. Thus, 
\[ \overline{\mGr_\lambda\cap T_\mu}= \overline{\overline{\mGr}_\lambda\cap T_\mu}=\overline{\overline{S}_\lambda\cap T_\mu}= \overline{S_\lambda\cap T_\mu}.\]
This concludes the proposition.
\end{proof}
\begin{definition}
\label{def_MV_weight}
Irreducible components of $S_\lambda\cap T_\mu$ (or $\overline{S_\lambda\cap T_\mu}$) are called \emph{MV cycles} of coweight $(\lambda, \mu)$. 
\end{definition}

After establishing Proposition \ref{T_S_prop}, 
the proof of the following corollary is identical to the proof of \cite[Proposition 3]{An}.
\begin{corollary}
\label{cor_MV_Gr}
Let $\lambda\in X_*(H)_\tau^+$ and $\mu\preceq \lambda$. The components of
$\overline{ S_\lambda \cap T_\mu }$ contained in $\overline{\mGr}_{ \lambda}$ are exactly the components of  
$\overline{\mGr_\lambda\cap T_\mu}$. 
\end{corollary}
 
We recall the following theorem of Zhu, cf.\,\cite{Zh}.
\begin{theorem}
Assume that $\rm{char} (\rk)\not= \ell$. 
The $\ell$-adic intersection cohomology $IH^*(\overline{\mGr}_\lambda, \overline{\mathbb{Q}}_\ell)$ with coefficients in $\overline{\mathbb{Q}}_\ell$ admits an algebraic action of the reductive $\overline{\mb{Q}}_\ell$-algebraic group $\check{G}^\tau$, which is isomorphic to the highest weight module $V(\lambda)$ of $\check{G}^\tau$. 
\end{theorem}
By \cite[Lemma 3.7]{ALRS}, $\overline{\mGr}_\lambda\cap T_\mu$ is of pure dimension $\la\lambda-\mu,\rho \ra$. 
By the same arguments as in \cite[Theorem 3.6, Proposition 3.10]{MV}, we have
\begin{theorem}
\label{thm_ramified_Satake}
There exists a decomposition 
\[ IH^*(\overline{\mGr}_\lambda,\overline{\mathbb{Q}}_\ell)\simeq \bigoplus_{\mu\in X_*(H)_\tau} H_c^{2\la \lambda-\mu, \rho \ra} (\overline{\mGr}_\lambda\cap T_\mu, \overline{\mathbb{Q}}_\ell) \]
corresponding to the weight decomposition $V(\lambda)=\bigoplus_{\mu\in X_*(H)_\nu} V(\lambda)_\mu$. Thus, the weight space $V(\lambda)_\mu$ admits a basis from the components of $\overline{\mGr}_\lambda\cap T_\mu$.
\end{theorem}


 Set $X_*(H)_{\tau,\mb{R}}:=X_*(H)_\tau\otimes_\mb{Z} \mb{R}$. For any $H^\sigma$-stable closed subset $Z$ in $\mGr$, we define its associated moment polytope
\[  P( Z ): ={\rm Conv} ( \mu   \mid  e_\mu\in  Z     ) \subset X_*(H)_{\tau,\mb{R}}.    \]
 In particular,  $P\left(\overline{\mGr}_\lambda\right)$ is the Weyl polytope $P_\lambda:={\rm Conv}(w\cdot \lambda)_{w\in W^\tau}$. For any MV cycle $C$,  $P(C)$ is called the associated MV polytope. 

\begin{proposition}
An irreducible component $C$ of $\overline{ S_\lambda\cap T_\mu }$ is contained in $\overline{\mGr}_\lambda$ if and only if $P(C)$ is contained in the Weyl polytope $P_\lambda$.
\end{proposition}
\begin{proof}
Using Lemma \ref{key_lem}, the proof of this proposition is identical to the proof of \cite[Proposition 7]{An}.
\end{proof}

\subsection{Twisted GGMS strata}
\label{twisted ggms strata}
By translation invariance, the components of $S_\lambda \cap T_\mu$ depend only on $\lambda-\mu$. Hence, it suffices to consider $S_\lambda\cap T_0$, where $\lambda$ is in the root lattice of $(\check{G})^\tau$. Let $\mGr'$ denote the twisted affine Grassmannian defined as in (\ref{eq_twisted aff_def}) for the simply-connected cover $G'$ of $G$ with the lifted action of $\sigma$. The induced morphism 
\[ \mGr'\to \mGr^\circ  \]
to the neutral component $\mGr^\circ$ of $\mGr$ restricts to a finite birational universal homeomorphism on corresponding Schubert varieties; see 
\cite[Section 2, proof of Proposition 2.1]{HainesLourencoRicharz2024}. (In fact, this is an isomorphism when $p\not \mid |\pi_1(G)| $). This morphism is compatible with the corresponding semi-infinite cells and torus fixed points.  Consequently, it induces a bijection on the irreducible components of the corresponding intersections, preserving their dimensions and moment polytopes.  We may therefore assume that \(G\) is {\it simply-connected} from now on.

Let $2\check{\rho}$ be the summation of all positive coroots of $G$. Then, $2\check{\rho}$ can be regraded as a regular cocharacter $2\check{\rho}: \mb{G}_m\to H^\sigma \subset G^\sigma$. Then, 
\[ T_\mu^w= \{ L\in \mGr  \mid      \lim_{s\to \infty}w(2\check{\rho}) (s)\cdot L= e_\mu    \}  . \]
Using this description and Proposition \ref{prop_semi_closure}, the proof of the following lemma is the same as the proof of \cite[Lemma 2.1]{Kam}.
\begin{lemma}
\label{lem:semiinfinite-int}
For any $w,y\in W^\tau$, if $T_\mu^w  \cap T_\nu^y \not= \emptyset$, then $\nu\succeq_w \mu$.
\end{lemma}

For every simple reflection $s_i\in W$, we fix a representative $\overline{s}_i:=x_i(-1)y_i(1)x_i(-1)$. These representatives satisfy the braid relation and hence give rise to a unique representative $\overline{w}$ for each $w\in W$ through any of its reduced words. Note that $\tau(\overline{s}_i)=\overline{s}_{\tau(i)}$. Therefore $\tau(\overline{w})=\overline{\tau(w)}$. Meanwhile, when $w\in W^\tau$, $\overline{w}$ may be different from the $\sigma$-invariant representative $\dot{w}$ chosen before; in particular, $\overline{w}$ is $\tau$-invariant. See \cite[\S 2.3]{Hong-Shen-2015} for more discussions.

For each $i\in I$, let $V(\Lambda_i)$ be the Weyl module of $G$ with fundamental weight $\Lambda_i$ with respect to $(B, H)$. We fix a highest weight vector $v_{\Lambda_i}$ for each $V(\Lambda_i)$ and define the extremal weight vectors $v_{w\Lambda_i}:=\overline{w}\cdot v_{\Lambda_i}$.
Let $V(\Lambda_i)^*$ be the dual representation of $V(\Lambda_i)$. Then $V(\Lambda_i)^*$ is of highest weight $\Lambda_i^*:= -w_0\Lambda_i$. Let $\{\xi_{w\Lambda_i} \mid  w\in W \}$ be the
dual vectors of $\{ v_{w\Lambda_i} \mid w\in W\}$. In particular, $\xi_{w\Lambda_i}=\overline{w}\cdot\xi_{\Lambda_i}$ is of weight $-w\Lambda_i$.

Let $\Lambda_\imath=\Lambda_i|_{H^\sigma}$ for the fixed $i\in \imath$,  and let $\overline{V}(\Lambda_\imath)$ be the Weyl module of $G^\sigma$ of highest weight $\Lambda_\imath$. By a dual version of \cite[Theorem I.8.3]{FS} (note that $p\not \mid m$),  $V(\Lambda_i)|_{G^\sigma}$ admits a Weyl filtration as a representation of $G^\sigma$.  Since $\Lambda_\imath$ is the maximal weight appearing in $V(\Lambda_i)|_{G^\sigma}$, 
by \cite[\S 4.19]{Ja} there is an embedding $\overline{V}(\Lambda_{\imath})\hookrightarrow V(\Lambda_i)|_{G^\sigma}$, which is  unique up to a scalar.  By base change from $\rk$ to $\mathcal{K}$, the loop groups $G(\mathcal{K})$ and $G(\mc{K})^\sigma$ act on $V(\Lambda_i)_\mathcal{K}$ and $V(\Lambda_i)_\mc{K}^*$.

For any $w\in W^\tau$ and $g\in G(\mc{K})^\sigma$,  define
\begin{equation}
\label{D_def}
D_{w \Lambda_\imath} (g):= \min_{v \in \overline{V}(\Lambda_\imath)}  \left\{{\rm val} (\la   \xi_{w\Lambda_i}  , g\cdot  v  \ra )  \right\}. \end{equation}
where $\langle ,\rangle: V(\Lambda_i)^*_\mathcal{K}\times V(\Lambda_i)_\mathcal{K}$ is the standard pairing. 
Define the map
\begin{align*}
a: X_*(H)_\tau&\longrightarrow X_*(H)^\tau\\
\lambda & \longmapsto \sum_{s=0}^{m-1}\tau^s(\tilde{\lambda}),
\end{align*}
 where $\tilde{\lambda}$ is a representative of $\lambda$. 


\begin{lemma}
\label{lem:semi-infinite}
For any $k\in G(\mc{O})^\sigma$ and $g\in G(\mc{K})^\sigma$, 
$D_{w \Lambda_\imath}(gk)=D_{w\Lambda_\imath}(g)$. Thus, $D_{w\Lambda_\imath}$ descends to a function on $\mGr$.  
 Moreover,
 \[
 T^w_\mu=\{L\in \mGr\mid D_{w \Lambda_\imath}(L)=\la w \Lambda_\imath, a(\mu)\ra \text{ for all $\iota\in I_\tau$}\}.
 \]
\end{lemma}
\begin{proof}
By the Iwasawa decomposition, 
we can write $g= \dot{w} u \dot{w}^{-1}n^{\tilde{\mu}} k$, where $u\in N^{-}(\mc{K})^\sigma$,  $k\in G(\mc{O})^\sigma$, $\dot{w}$ is a $\sigma$-invariant representative of $w$ in $G$, and $\tilde{\mu}$ is a lifting of $\mu\in X_*(H)_\tau$ in $X_*(H)$. Observe that for any $v\in \overline{V}(\Lambda_\imath)$,  $k\cdot v =k_0\cdot v+ tV(\Lambda_i)_{\mc{O}}$, where $k_0\in G^\sigma\subset G(\mc{O})^\sigma$ is the image of $k$ via the evaluation map $G(\mc{O})^\sigma\to G^\sigma$. Then, 
\[ \la   \xi_{w \Lambda_i}  , g\cdot v  \ra= \la  \xi_{w \Lambda_i}, g\cdot v  \ra  = (\epsilon^{-1}t)^{ \la w \Lambda_i, a(\mu)  \ra } \la  \xi_{w \Lambda_i}, k \cdot v  \rangle   \]
lies in $   (\epsilon^{-1}t)^{\la w\Lambda_i, a(\mu)  \ra }  (  \la  \xi_{w\Lambda_i}, k_0 \cdot v  \ra + t \mc{O})$.
Since $ \xi_{w \Lambda_i}$ is nontrivial on $\overline{V}(\Lambda_\imath)$,
 it follows that
 \[ \min_{v \in \overline{V}(\Lambda_\imath)} \val ( \la  \xi_{w \Lambda_\imath}  , g\cdot v  \ra)= \la w \Lambda_i, a(\mu) \ra= \la w \Lambda_\imath, a(\mu) \ra .\]
 From this calculation, it is easy to see that $D_{w \Lambda_\imath}$ is invariant with respect to the right multiplication of $G(\mc{O})^\sigma$, and thus the lemma follows. 
\end{proof}

 As a consequence of Proposition \ref{prop_semi_closure}, we have 
 \begin{align*}P( \overline{T_\mu^w} )=&\{ x \in X_*(H)_{\tau,\mathbb{R}}  \mid  x\succeq_w \mu   \}\\
 =&\{  x\in X_*(H)_{\tau,\mathbb{R}} \mid  \la w\Lambda_\imath, a(x)  \ra   \geq \la  w\Lambda_\imath, a(\mu)   \ra   \text{ for all } \imath\in I_\tau  \},\end{align*}
 where $a: X_*(H)_\tau\to X_*(H)^\tau$ extends to $a: X_*(H)_{\tau,\mathbb{R}}\to X_*(H)^\tau_\mb{R}$, and the partial order $\succeq_w$ extends from $X_*(H)_\tau$ to $X_*(H)_{\tau,\mathbb{R}} $. 

For a collection $(\mu_w)_{w\in W^\tau}$ of elements in $X_*(H)_\tau$,   we define the following GGMS stratum as in \cite[2.4]{Kam}, 
 \[A(\mu_\bullet):= \bigcap_{w\in W^\tau}  T^w_{\mu_w}  . \]
 This intersection is empty unless $\mu_y \succeq_w \mu_w$ for any $y,w$. By the same proof as in \cite[Lemma 2.3]{Kam}, we have
 \begin{lemma}
 If $A(\mu_\bullet)$ is not empty, then $P( \overline{A(\mu_\bullet) })= {\rm Conv}(\mu_\bullet)$.
 \end{lemma}

We will consider irreducible components of the intersection $ T_{0}\cap S_{\mu} $, where $\mu$ is a non-negative summation of $\gamma_\imath, \imath\in I_\tau$.

\begin{definition}\label{def_exceptional} An orbit $\imath\in I_\tau$ is said to be \emph{exceptional} if $X_N^{(r)}=A_{2l}^{(2)}$ and $\imath=\{l,l+1\}$; otherwise we say the orbit $\imath$ is \emph{standard}. Note that $a_{ij}=0$ for any distinct $i,j$ in a standard orbit $\imath$.
\end{definition}

We set 
\begin{equation}\label{eq: s_i}
s_\imath=\left\{\begin{array}{ll}
    \prod_{i\in \imath} s_i & \text{if $\imath$ is standard}, \\
    s_ls_{l+1}s_l & \text{if $\imath=\{l,l+1\}$ is exceptional.}
    \end{array}\right.
\end{equation}
Then $\{ s_\imath  \mid \imath\in I_\tau \}$ is a set of simple reflections in the Coxeter group $W^\tau$. 
Let 
%
$\imath_\bullet=(\imath_1,\cdots, \imath_N)$ be a reduced word for $w_0\in W^\tau$, where $W^\tau$ is regarded as the Weyl group of  $(\check{G})^\tau$. Let $w_{ k}^{\imath_\bullet}:= s_{\imath_1}s_{\imath_2}\cdots s_{\imath_k}$.  
Set $\beta_{k}^{\imath_\bullet}=w_{ k-1}^{\imath_\bullet}(\gamma_{\imath_k})$. 
Given a tuple of non-negative integers $n_\bullet=(n_1,\cdots, n_N)$, set $\mu_k^{\imath_\bullet }=\sum_{s=1}^k  n_s \beta_{s}^{\imath_\bullet}$.
Define 
\begin{equation}
\label{eq_A_i(n)}
A^{\imath_\bullet}(n_\bullet):=\bigcap_{k=0}^N  T^{ w_{k}^{\imath_\bullet}  } _{ \mu_{k}^{\imath_\bullet}  }  \end{equation}
In Section \ref{Section_para_MV}, we show that $\overline{A^{\imath_\bullet}(n_\bullet)}$ is an irreducible component of $\overline{S_\mu\cap T_0}$ with $\mu=\sum_{s=1}^Nn_s\beta_s^{\imath_\bullet}$.

Since $A^{\imath_\bullet}(n_\bullet)$ is the intersection of $T^{w_k^{\imath_\bullet}}_{\mu_k^{\imath_\bullet}}$, we can depict $A^{\imath_\bullet}(n_\bullet)$ using the function $D_{w_k^{\imath_\bullet}\Lambda_i}$. Let us define
\begin{equation} \label{eq: folded cartan}
a_{\imath\jmath}=\left\{\begin{array}{ll}\displaystyle\frac{1}{|\jmath|}\sum_{\substack{s\in \imath \\ t\in \jmath}}a_{st} & \text{if $\jmath$ is standard},\\
\displaystyle\sum_{\substack{s\in \imath \\ t\in \jmath}}a_{st} & \text{if $\jmath$ is exceptional}.
\end{array}\right.
\end{equation}
Then we have the following proposition.

\begin{proposition}\label{prop: A(n)} $A^{\imath_\bullet}(n_\bullet)$ consists of those $L\in \mGr$ satisfying
\[
D_{\Lambda_\imath}(L)=0 \quad \text{ for all $\imath\in I_\tau$, and}
\]
\[
D_{w_{ k}^{\imath_\bullet}\Lambda_{\imath_k}}(L)=-n_km_{\imath_k}-D_{w_{k-1}^{\imath_\bullet}\Lambda_{\imath_k}}(L)-\sum_{\jmath\neq \imath_k}a_{\jmath \imath_k}D_{w_{k-1}^{\imath_\bullet}\Lambda_{\jmath}}(L) \quad \text{ for all $1\leq k\leq N$},
\]
where $m_{\imath_k}=\frac{m}{|\imath_k|}$.
\end{proposition}
\begin{proof} The condition $D_{\Lambda_\imath}(L)=0$ is determined by the fact that $L\in T_0$. We now inductively prove the remaining equalities by an induction on $k$. Note that for each $k$, by Lemma \ref{lem:semi-infinite}, $L\in T^{w^{\imath_\bullet}_k}_{\mu^{\imath_\bullet}_k}$ if and only if $D_{w^{\imath_\bullet}_k\Lambda_\jmath}(L)=\left\la w^{\imath^\bullet}_k\Lambda_\jmath,a(\mu^{\imath_\bullet}_k)\right\ra$ for all $\jmath$. Recall that 
\[
\mu^{\imath_\bullet}_k=\mu^{\imath_\bullet}_{k-1}+n_k\beta^{\imath_\bullet}_k=\mu_{k-1}^{\imath_\bullet}+n_kw^{\imath_\bullet}_{k-1}\gamma_{\imath_k}.
\]
For any $\jmath\neq \imath_k$, we have $w^{\imath_\bullet}_k\Lambda_\jmath=w^{\imath_\bullet}_{k-1}\Lambda_\jmath$, and hence
\[
\left\la w^{\imath^\bullet}_k\Lambda_\jmath,a(\mu^{\imath_\bullet}_k)\right\ra=\left\la w^{\imath^\bullet}_{k-1}\Lambda_\jmath, a(\mu^{\imath_\bullet}_{k-1})\right\ra+n_k\left\la \Lambda_\jmath,a(\gamma_{\imath_k})\right\ra=\left\la w^{\imath^\bullet}_{k-1}\Lambda_\jmath, a(\mu^{\imath_\bullet}_{k-1})\right\ra.
\]
In other words, the conditions for $L\in T^{w^{\imath_\bullet}_k}_{\mu^{\imath_\bullet}_k}$ on $\jmath\neq \imath_k$ are identical to those for $L\in T^{w^{\imath_\bullet}_{k-1}}_{\mu^{\imath_\bullet}_{k-1}}$. The remaining condition is associated with $\imath_k$. Note that $s_{\imath_k}\Lambda_{\imath_k}=-\Lambda_{\imath_k}-\sum_{\jmath\neq \imath_k}a_{\jmath\imath_k}\Lambda_\jmath$. Therefore,
\begin{align*}
\left\la w^{\imath^\bullet}_k\Lambda_{\imath_k},a(\mu^{\imath_\bullet}_k)\right\ra=&\left\la w^{\imath^\bullet}_{k-1}s_{\imath_k}\Lambda_{\imath_k},a(\mu^{\imath_\bullet}_{k-1})\right\ra+n_k\left\la s_{\imath_k}\Lambda_{\imath_k},a(\gamma_{\imath_k})\right\ra\\
=&-\left\la w^{\imath^\bullet}_{k-1}\Lambda_{\imath_k},a(\mu^{\imath_\bullet}_{k-1})\right\ra-\sum_{\jmath\neq \imath_k}a_{\jmath\imath_k}\left\la w^{\imath^\bullet}_{k-1}\Lambda_{\jmath},a(\mu^{\imath_\bullet}_{k-1})\right\ra-n_km_{\imath_k}\\
=&-n_km_{\imath_k}-D_{w_{k-1}^{\imath_\bullet}\Lambda_{\imath_k}}(L)-\sum_{\jmath\neq \imath_k}a_{\jmath \imath_k}D_{w_{k-1}^{\imath_\bullet}\Lambda_{\jmath}}(L).
\end{align*}
This proves the remaining equalities.
\end{proof}

\begin{example} Suppose $W^\tau$ is of type $\mathrm{B}_2$. Fix the reduced word $s_1s_2s_1s_2$ for $w_0$, and let $n_\bullet=(1,2,2,1)$. The coweights $\{\mu_{ k}^{\imath_\bullet}\}_{k=0}^N$ are shown below. The red part is drawn via the other reduced word $s_2s_1s_2s_1$ of $w_0$, where  $n'_\bullet=(3,1,1,4)$ from Proposition \ref{prop:B2'}.
\begin{figure}[H]
    \centering
    \begin{tikzpicture}[scale=0.7]
        \node (0) at (0,0) [] {$\bullet$};
        \node (1) at (1,0) [] {$\bullet$};
        \node (2) at (3,2) [] {$\bullet$};
        \node (3) at (3,4) [] {$\bullet$};
        \node (4) at (2,5) [] {$\bullet$};
        \foreach \i in {1,...,4}
        {
        \pgfmathtruncatemacro{\j}{\i-1}
        \draw [->] (\j) -- (\i);
        }
        \node at (0) [below] {$0$};
        \node at (1) [below right] {$\mu_{w_{ 1}^{\imath_\bullet}}$};
        \node at (2) [right] {$\mu_{w_{ 2}^{\imath_\bullet}}$};
        \node at (3) [above right] {$\mu_{w_{ 3}^{\imath_\bullet}}$};
        \node at (4) [above] {$\mu_{w_0}$};
        \node (r1) [red] at (-3,3) [] {$\bullet$};
        \node (r2) [red] at (-3,4) [] {$\bullet$};
        \node (r3) [red] at (-2,5) [] {$\bullet$};
        \draw [->,red] (0) -- (r1);
        \draw [->,red] (r1) -- (r2);
        \draw [->, red] (r2) -- (r3);
        \draw [->, red] (r3) -- (4);
    \end{tikzpicture}
    \caption{Example of an MV polytope.}
\end{figure}
\end{example}

\section{Twisted Lusztig coordinates and generalized minors}
In this section, we introduce the twisted Lusztig coordinates and study the behavior of generalized minors under the automorphism \(\sigma\). These constructions will be used, together with the twist map developed in Section \ref{ss: conf space}, to parametrize the twisted GGMS strata in Section \ref{ss:parametrization}.

In this section and later, we assume that $G$ is simply-connected.
\subsection{Twisted Lusztig coordinates}\label{subsec: reduced word of w_0}

For $\imath\in I_\tau$, set $m_\imath=\frac{m}{|\imath|}$ and $\mc{K}_\imath=\rk\dbp{t^{m_\imath}}$, where $|\imath|$ is the cardinality of $\imath$. Note that $\mc{O}_\imath=\rk\dbb{t^{m_\imath}}$. The subfield of $\gamma$-fixed points is $\mc{K}_\imath^\gamma=\rk\dbp{t^m}$.



For each standard orbit $\imath\in I_\tau$, 
let $i\in \imath$ be a fixed representative and define for $a\in \mc{K}_\imath$
\begin{equation}\label{eq: root subgroup elements} 
y_\imath(a)= \prod_{k=0}^{|\imath|-1} y_{\tau^k(i)}\left(\epsilon^{-k\alpha_{i}(h)}\gamma^k(a)\right),
\end{equation}
where $h$ is defined in (\ref{eq_h}).

For the exceptional $\imath=\{l,l+1\}$ in $A_{2l}^{(2)}$,
we define for $(b,c)\in \mc{K}_\imath\times \mc{K}_\imath^\gamma$
\begin{equation}\label{eq: a}
a=\left\{\begin{array}{ll}(b\gamma(b)+tc)/2 & \text{if $m=2$,}\\
\left(-t^2b\gamma(b)+c\right)/2 & \text{if $m=4$.}
\end{array}\right.
\end{equation}
We define
\begin{equation}\label{eq: root subgroup element for A2n}
y_\imath(b,c)=\left\{\begin{array}{ll}y_l\left(\dfrac{\gamma(a)}{b}\right)y_{l+1}(b)y_l\left(\dfrac{a}{b}\right) & \text{if $b\neq 0$ and $m=2$,}\\
y_l\left(\dfrac{-\gamma(a)}{tb}\right)y_{l+1}(tb)y_l\left(\dfrac{a}{tb}\right) & \text{if $b\neq 0$ and $m=4$,}\\
y_l(1)y_{l+1}(- a)y_l(-1)y_{l+1}(a) & \text{if $b=0$.}\end{array}\right.
\end{equation}

The following result of Lusztig will be used multiple times in this article.

\begin{proposition}[{\cite[Proposition 2.2]{Lusztig}}]\label{prop: Lusztig transform} Suppose $a_{ij}=a_{ji}=-1$. Then for any non-zero scalars $f_1$, $f_2$, and $f_3$ with $f_1+f_3\neq 0$,
\begin{equation}\label{eq: Lusztig transform}
y_i(f_1)y_j(f_2)y_i(f_3)=y_j\left(\frac{f_2f_3}{f_1+f_3}\right)y_i(f_1+f_3)y_j\left(\frac{f_1f_2}{f_1+f_3}\right)
\end{equation}
Moreover, if $f_1+f_3=0$, then such a transformation does not exist.
\end{proposition}

\begin{lemma}\label{lem: change of word in special case} Suppose $a_{ij}=a_{ji}=-1$. Then 
\begin{enumerate}
    \item $y_i(1)y_j(-f)y_i(-1)y_j(f)=y_j(1)y_i(f)y_j(-1)y_i(-f)$.
    \item If $f\neq 0$, then for any scalars $f_1$, $f_2$, and $f_3$,
    \[
    y_i(1)y_j(-f)y_i(-1)y_j(f)\neq y_i(f_1)y_j(f_2)y_i(f_3).
    \]
    \item If $f\neq 0$, $y_i(1)y_j(-f)y_i(-1)y_j(f)\in N^-\cap Bs_is_js_iB$.
\end{enumerate}
\end{lemma}
\begin{proof} (1) and (2) can be verified by a direct computation in $\SL_3$. Let $i=1$ and $j=2$, then on the one hand,
\[
y_1(1)y_2(-f)y_1(-1)y_2(f)=\begin{pmatrix}
    1 & 0 & 0 \\ 0 & 1 & 0 \\ f & 0 & 1
\end{pmatrix}=y_2(1)y_1(f)y_2(-1)y_1(-f).
\]
On the other hand,
\[
y_1(f_1)y_2(f_2)y_1(f_3)=\begin{pmatrix} 1 & 0 & 0 \\ f_1+f_3 & 1 & 0 \\ f_2f_3 & f_2 & 1
\end{pmatrix}.
\]
In order to match the $(3,2)$ entry of the two matrices, we need $f_2=0$; but then it forces $f_2f_3=0\neq f$ at the $(3,1)$ entry. Thus, no such $f_1,f_2,f_3$ are possible.

For (3), let $x(f):=x_j(1)x_i(-f)x_j(-1)x_i(f)$. Since $f\neq 0$, there exists some $b\in B$ such that
\[
y_i(1)y_j(-f)y_i(-1)y_j(f)=x(1/f)\overline{s}_i\overline{s}_j\overline{s}_ib.
\]
This proves that $y_i(1)y_j(-f)y_i(-1)y_j(f)\in N^-\cap Bs_is_js_iB$.
\end{proof}

\begin{lemma}\label{lem: limit} Let $\imath$ be an exceptional orbit. In both the $m=2$ and $m=4$ cases in Eq. \eqref{eq: root subgroup element for A2n}, $\lim_{b\rightarrow 0}y_\imath(b,c)=y_\imath(0,c)$.
\end{lemma}
\begin{proof} It suffices to verify the lemma by a computation in $\SL_3$. Let $i=1$ and $j=2$. When $m=2$, we have $a+\gamma(a)=b\gamma(b)$, and hence
\[
\lim_{b\rightarrow 0} y_1\left(\frac{\gamma(a)}{b}\right)y_2(b)y_1\left(\frac{a}{b}\right)=\lim_{b\rightarrow 0}\begin{pmatrix} 1 & 0 & 0 \\ \gamma(b) & 1 & 0 \\ a & b & 1\end{pmatrix}=\begin{pmatrix} 1 & 0 & 0 \\ 0 & 1 & 0 \\ a & 0 & 1\end{pmatrix} = y_1(1)y_2(-a)y_1(-1)y_2(a).
\]
When $m=4$, we have $a-\gamma(a)=-t^2b\gamma(b)$, and hence
\begin{align*}
\lim_{b\rightarrow 0} y_1\left(\frac{-\gamma(a)}{tb}\right)y_2(tb)y_1\left(\frac{a}{tb}\right)=\lim_{b\rightarrow 0}\begin{pmatrix} 1 & 0 & 0 \\ -t\gamma(b) & 1 & 0 \\ a & tb & 1\end{pmatrix}=&\begin{pmatrix} 1 & 0 & 0 \\ 0 & 1 & 0 \\ a & 0 & 1\end{pmatrix}\\ =& y_1(1)y_2(-a)y_1(-1)y_2(a). \qedhere
\end{align*}
\end{proof}

\begin{proposition}\label{prop: sigma invariant unipotent} The $y_\imath(a)$ and $y_\imath(b,c)$ in Eq. \eqref{eq: root subgroup elements} and Eq. \eqref{eq: root subgroup element for A2n} are $\sigma$-invariant.
\end{proposition}
\begin{proof} It is obvious that $y_\imath(a)$ in Eq. \eqref{eq: root subgroup elements} for a standard $\imath$ is $\sigma$-invariant. 

Now we consider the exceptional orbit $\imath=\{l,l+1\}$ in $A_{2l}^{(2)}$. For $m=2$ and $b\neq 0$, we have $a+\gamma(a)=b\gamma(b)$. Hence,
\begin{align*}
\sigma(y_\imath(b,c))=&y_{l+1}\left(\frac{a}{\gamma(b)}\right)y_l\left(\gamma(b)\right)y_{l+1}\left(\frac{\gamma(a)}{\gamma(b)}\right)\\
=&y_l\left(\frac{\gamma(a)}{b}\right)y_{l+1}\left(b\right)y_l\left(\frac{a}{b}\right)\\
=&y_\imath(b,c).
\end{align*}
If $b=0$, then $\gamma(a)=-a$. Hence,
\[
\sigma(y_\imath(0,c))=y_{l+1}(1)y_l(a)y_{l+1}(-1)y_l(-a),
\]
which is equal to $y_\imath(0,c)$ by Lemma \ref{lem: change of word in special case} (1).

For $m=4$ and $b\neq 0$, we have $a-\gamma(a)=-t^2b\gamma(b)$, and hence
\begin{align*}
\sigma(y_\imath(b,c))=&y_{l+1}\left(\frac{-a}{t\gamma(b)}\right)y_l\left(-t\gamma(b)\right)y_{l+1}\left(\frac{\gamma(a)}{t\gamma(b)}\right)\\
=&y_l\left(\frac{-\gamma(a)}{tb}\right)y_{l+1}\left(tb\right)y_l\left(\frac{a}{tb}\right)\\
=&y_\imath(b,c).
\end{align*}
If $b=0$, then $\gamma(a)=a$. Hence
\[
\sigma(y_\imath(0,c))=y_{l+1}(\epsilon^{-1})y_l(- \epsilon^{-1}a)y_{l+1}(-\epsilon^{-1})y_l(\epsilon^{-1}a).
\]
By a computation similar to the proof of Lemma \ref{lem: change of word in special case}, one verifies that this is  $y_\imath(0,c)$.
\end{proof}

\subsection{Generalized minors}

For any $w\Lambda_i$, 
define the generalized minor 
\begin{equation}\label{eq: generalized minor} \Delta_{w\Lambda_i}(g)=\la  \xi_{w\Lambda_i}, g\cdot v_{\Lambda_i}   \ra ,\quad \text{ for any } g\in G(\mc{K}) .\end{equation}


\begin{lemma}
    \label{lem:MC-diagram}
For any $g\in G,  w\in W$, and any dominant weight $\Lambda_i$,   we have
\[
\Delta_{\tau(w\Lambda_i)}(\tau(g))=\Delta_{w\Lambda_i}(g).
\]
\end{lemma}
\begin{proof}
Let us set 
\[
f(g)= \la \xi_{w \Lambda_i},  g \cdot v_{\Lambda_i} \ra =\Delta_{w\Lambda_i}(g) \quad \text{and} \quad   f'(g)= \la  \xi_{\tau(w) \Lambda_{\tau(i)} },  \tau(g)\cdot v_{\Lambda_{\tau(i)}} \ra=\Delta_{\tau(w\Lambda_i)}(\tau(g)).    
\]
Since $f$ and $f'$ are regular functions on $G$, it suffices to check that $f$ and $f'$ agree on the open subset $\overline{w}N^-HN\subset G$. 
Recall that $\overline{\tau(w)}=\tau(\overline{w})$ and $\xi_{w\Lambda_i}=\overline{w}\cdot \xi_{\Lambda_i}$. Thus,
for any $y\in N^-, t\in T, x\in N $, on the one hand, 
\[
f(\overline{w}ytx  ) = \la  \xi_{w\Lambda_i}, \overline{w}ytx\cdot v_{\Lambda_i} \ra =  \la  \overline{w}^{-1}\cdot\xi_{w\Lambda_i},  yt \cdot v_{\Lambda_i}   \ra   =  \Lambda_i(t)  \la  \xi_{\Lambda_i}, v_{\Lambda_i} \ra= \Lambda_i(t)  .
\]
On the other hand, 
\begin{align*} f'\left(\overline{w}ytx\right)=&\left\la\xi_{\tau(w)\Lambda_{\tau(i)}}, \tau(\overline{w})\tau(y)\tau(t)\tau(x)\cdot v_{\Lambda_{\tau(i)}}\right\ra\\
=&\left\la \xi_{\Lambda_{\tau(i)}},\tau(t)\cdot v_{\Lambda_{\tau(i)}}\right\ra\\
=&\Lambda_{\tau(i)}(\tau(t))\la\xi_{\Lambda_{\tau(i)}},v_{\Lambda_{\tau(i)}}\ra\\
=&\Lambda_i(t).
\end{align*}
Thus, $f=f'$ on $G$. 
\end{proof}

\begin{lemma}\label{lem: minor under gamma action} For $g\in G(\mc{K})$, $w\in W$, and any dominant weight $\Lambda_i$, we have
\[
\Delta_{\tau(w\Lambda_{i})}(\sigma(g))=\epsilon^{\la h, w\Lambda_i-\Lambda_i\ra}\gamma(\Delta_{w\Lambda_i}(g)).
\]
In particular, $\val(\Delta_{\tau(w\Lambda_i)}(\sigma(g)))=\val(\Delta_{w\Lambda_i}(g))$.
\end{lemma}
\begin{proof} By Lemma \ref{lem:MC-diagram}, for any $g\in G(\mc{K})$ we have
 $\la\xi_{\tau(w\Lambda_i)},\tau(g)\cdot v_{\Lambda_{\tau(i)}}\ra=\la \xi_{w\Lambda_i},g\cdot v_{\Lambda_i}\ra$, where $\tau$ acts on $G$ as a diagram automorphism and acts on $\mc{K}$ trivially. On the other hand, we also notice that $\epsilon^h$ is $\tau$-invariant, and hence for $g(f)\in G(\mc{K})$,
\begin{align*}
    \Delta_{\tau(w\Lambda_i)}(\sigma(g))=&\la \xi_{\tau(w\Lambda_i)},\epsilon^h\tau(g(\gamma(f)))\epsilon^{-h}\cdot v_{\Lambda_{\tau(i)}}\ra\\
    =&\epsilon^{\la h,\tau(w\Lambda_i-\Lambda_i)\ra}\la \xi_{\tau(w\Lambda_i)},\tau(g(\gamma(f))).v_{\Lambda_{\tau(i)}}\ra\\
    =&\epsilon^{\la h,w\Lambda_i-\Lambda_i\ra}\Delta_{w\Lambda_i}(g(\gamma(f)))\\
    =&\epsilon^{\la h, w\Lambda_i-\Lambda_i\ra}\gamma(\Delta_{w\Lambda_i}(g(f))).\qedhere
\end{align*}
\end{proof}

\begin{corollary}\label{cor: gamma invariant minor} Suppose $g\in G(\mc{K})^\sigma$, $w\in W^\tau$, and $\tau(i)=i$. Then $\Delta_{w\Lambda_i}(g)$ is $\gamma$-invariant.
\end{corollary}
\begin{proof} By Lemma \ref{lem: minor under gamma action}, it suffices to prove that $\epsilon^{\la h,w\Lambda_i-\Lambda_i\ra}=1$. This is obvious if $h=0$. Among the cases with $h\neq 0$, $X_N^{(r)}$ can only be $A_{2l-1}^{(2)}$ and $D_{l+1}^{(2)}$, and $m=2$ for both cases, which implies that $\epsilon=-1$. Note that $w\Lambda_i-\Lambda_i$ is an integer linear combination of simple roots. Since both $w$ and $i$ are $\tau$-invariant, the coefficients of $\alpha_1$ and $\alpha_{2l-1}$ in $w\Lambda_i-\Lambda_i$ must be equal in the case of $A_{2l-1}^{(2)}$, and the coefficients of $\alpha_l$ and $\alpha_{l+1}$ in $w\Lambda_i-\Lambda_i$ must be equal in the case of $D_{l+1}^{(2)}$ as well. Thus, it follows that $\la h,w\Lambda_i-\Lambda_i\ra$ must be an even integer and therefore $\epsilon^{\la h, w\Lambda_i-\Lambda_i\ra}=1$.
\end{proof}

\begin{lemma}\label{lem: pairing mod 4} In the case of $A_{2l}^{(2)}$, let $w\in W^\tau$. Then
\begin{enumerate}
    \item $\la h, w\Lambda_i-\Lambda_i\ra$ is an even integer for any $\Lambda_i$.
    \item If $i=l,l+1$ and $h=\check{\omega}_l+\check{\omega}_{l+1}$, then $\la h, w\Lambda_i-\Lambda_i\ra\equiv 2M\mod 4$, where $M$ is the number of occurrences of $s_\imath$ for the exceptional $\imath$ in a $\tau$-invariant reduced word of $w\in W^\tau$.
\end{enumerate}
\end{lemma}
\begin{proof} Claim (1) is obviously true when $h=0$. When $h=\check{\omega}_l+\check{\omega}_{l+1}$, we need to count the occurrences of $\alpha_l$ and $\alpha_{l+1}$ in $w\Lambda_i-\Lambda_i$. Note that $\alpha_l$ and $\alpha_{l+1}$ are only produced during the action of $s_\imath$ for the exceptional $\imath$ in a reduced word of $w$, and for any weight $\mu$,
\begin{align*}
s_ls_{l+1}s_l(\mu)=& s_ls_{l+1}(\mu-\la \check{\alpha}_l,\mu\ra\alpha_l)\\
=&s_l(\mu-\la \check{\alpha}_{l+1},\mu\ra\alpha_{l+1}-\la \check{\alpha}_l,\mu\ra(\alpha_l+\alpha_{l+1}))\\
=&\mu-\la \check{\alpha}_l,\mu\ra \alpha_l-\la \check{\alpha}_{l+1},\mu\ra(\alpha_l+\alpha_{l+1})-\la \check{\alpha}_l,\mu\ra\alpha_{l+1}\\
=&\mu-\la\check{\alpha}_l+\check{\alpha}_{l+1},\mu\ra(\alpha_l+\alpha_{l+1}).
\end{align*}
Since $\la\check{\alpha}_l+\check{\alpha}_{l+1},\mu\ra$ is always an integer, we can conclude that $w\Lambda_i-\Lambda_i$ contains an integer multiple of $\alpha_l+\alpha_{l+1}$ in its expansion with respect to simple roots, and hence $\la h, w\Lambda_i-\Lambda_i\ra$ is an even integer, proving claim (1).

For claim (2), we do an induction on $M$. Let $i$ be either $l$ or $l+1$. When $M=0$, neither $\alpha_l$ nor $\alpha_{l+1}$ appears in $w\Lambda_i-\Lambda_i$, and therefore $\la h, w\Lambda_i-\Lambda_i\ra=0$. Inductively, we suppose $w=us_\imath v$ for some $u,v\in W^\tau$ such that $v$ does not contain $s_\imath$ in its reduced word; then
\begin{align*}
\la h, w\Lambda_i-\Lambda_i\ra=&\la h, us_\imath\Lambda_i-\Lambda_i\ra\\
=&\la h, u(\Lambda_i-\alpha_l-\alpha_{l+1})-\Lambda_i\ra\\
\equiv &2(M-1)-\la h, u(\alpha_l+\alpha_{l+1})\ra \mod 4.
\end{align*}
Now it remains to prove that $\la h,u(\alpha_l+\alpha_{l+1})\ra\equiv 2\mod 4$. By a similar argument as above, the change in the coefficients of $\alpha_l$ and $\alpha_{l+1}$ in the expansion of $u(\alpha_l+\alpha_{l+1})$ can only occur as a result of $s_\imath$ for the exceptional $\imath$ in a reduced word of $u$, and the change is equal to $\la \check{\alpha}_l+\check{\alpha}_{l+1}, \cdot\ra(\alpha_l+\alpha_{l+1})$. Since $\la \alpha_l,\cdot\ra$ and $\la \alpha_{l+1},\cdot\ra$ are equal integers, the change $\la \check{\alpha}_l+\check{\alpha}_{l+1}, \cdot\ra(\alpha_l+\alpha_{l+1})$ paired with $h$ is congruent to $0$ mod $4$. Thus, 
\[
\la h, u(\alpha_l+\alpha_{l+1})\ra\equiv  \la h, \alpha_l+\alpha_{l+1}\ra\equiv 2\mod 4.\qedhere
\]
\end{proof}

\begin{corollary}\label{cor: A2l with m=4} In the case of $A_{2l}^{(2)}$ with $m=4$, if $g\in G(\mc{K})^\sigma$ and $w\in W^\tau$, then $\Delta_{w\Lambda_i}(g)\in k\dbp{t^2}$ for all $\Lambda_i$.
\end{corollary}
\begin{proof} Note that in this case $\tau^2(i)=i$ and $\la h,w\Lambda_i-\Lambda_i\ra=\la h, w\Lambda_{\tau(i)}-\Lambda_{\tau(i)}\ra$. Thus, by Lemmas \ref{lem: minor under gamma action} and \ref{lem: pairing mod 4} (1), we know that
\[
\Delta_{w\Lambda_i}(g)=\Delta_{w\Lambda_{\tau^2(i)}}(g)=\epsilon^{2\la h,w\Lambda_i-\Lambda_i\ra} \gamma^2(\Delta_{w\Lambda_i}(g))=\gamma^2(\Delta_{w\Lambda_i}(g)).
\]
This shows that $\Delta_{w\Lambda_i}(g)$ 
must be an element of $k\dbp{t^2}$.
\end{proof}

\begin{corollary}\label{cor: A2n m=4 middle minor} In the case of $A_{2l}^{(2)}$ with $m=4$, if $g\in G(\mc{K})^\sigma$ and $w\in W^\tau$, then $\Delta_{ws_l\Lambda_l}(g)$ and $\Delta_{ws_{l+1}\Lambda_{l+1}}(g)$ are in $tk\dbp{t^2}$.
\end{corollary}
\begin{proof} Let $w'=ws_ls_{l+1}s_l=ws_{l+1}s_ls_{l+1}\in W^\tau$. By Lemma \ref{lem: minor under gamma action} and Corollary \ref{cor: A2l with m=4}, we know that 
\begin{align*}\Delta_{ws_{l+1}\Lambda_{l+1}}(g)=&\Delta_{\tau(ws_l\Lambda_l)}(\sigma(g))=\epsilon^{\la h,ws_l\Lambda_l-\Lambda_l\ra}\gamma(\Delta_{ws_l\Lambda_l}(g)),\\
\Delta_{w\Lambda_{l+1}}(g)=&\Delta_{\tau(w\Lambda_l)}(\sigma(g))=\epsilon^{\la h, w\Lambda_l-\Lambda_l\ra}\gamma(\Delta_{w\Lambda_l}(g)),\\
\Delta_{w'\Lambda_{l+1}}(g)=&\Delta_{\tau(w'\Lambda_l)}(\sigma(g))=\epsilon^{\la h, w'\Lambda_l-\Lambda_l\ra}\gamma(\Delta_{w'\Lambda_l}(g)).
\end{align*}
We also observe that
\[
\la h,ws_l\Lambda_l-w\Lambda_l\ra=-\la h, w\alpha_l\ra \quad \text{and} \quad \la h,w'\Lambda_l-w\Lambda_l\ra=-\la h, w(\alpha_l+\alpha_{l+1})\ra=-2\la h, w\alpha_l\ra.
\]
Now let us consider the following identity of generalized minors (\cite[Proposition 4.2]{BZ})
\begin{equation}\label{eq: minor identity}
\Delta_{ws_{l+1}\Lambda_{l+1}}(g)\Delta_{ws_l\Lambda_l}(g)=\Delta_{w\Lambda_l}(g)\Delta_{w'\Lambda_{l+1}}(g)+\Delta_{w\Lambda_{l+1}}(g)\Delta_{w'\Lambda_l}(g).
\end{equation}
With appropriate substitution and cancellation, we get
\begin{align*}
\Delta_{ws_l\Lambda_l}(g)\gamma(\Delta_{ws_l\Lambda_l}(g))=&\epsilon^{-\la h, w\alpha_l\ra}\Delta_{w\Lambda_l}(g)\gamma(\Delta_{w'\Lambda_l}(g))+\epsilon^{\la h, w\alpha_l\ra}\gamma(\Delta_{w\Lambda_l}(g))\Delta_{w'\Lambda_l}(g)\\
=&\epsilon^{\la h, w\alpha_l\ra}\left(\gamma(\Delta_{w\Lambda_l}(g))\Delta_{w'\Lambda_l}(g)-\Delta_{w\Lambda_l}(g)\gamma(\Delta_{w'\Lambda_l}(g))\right).
\end{align*}
Since both $\gamma(\Delta_{w\Lambda_l}(g))\Delta_{w'\Lambda_l}(g)$ and $\Delta_{w\Lambda_l}(g)\gamma(\Delta_{w'\Lambda_l}(g))$ are in $k\dbp{t^2}$ and are $\gamma$-images of each other, we can conclude that the right hand side of the above equation are in $t^2k\dbp{t^4}$. This implies that $\Delta_{ws_l\Lambda_l}(g)$, and hence $\Delta_{ws_{l+1}\Lambda_{l+1}}(g)$, must be in $tk\dbp{t^2}$.
\end{proof}

\begin{corollary}\label{prop: quotient in K} Suppose $\imath$ is a standard orbit and $i\in \imath$. Suppose $g\in G(\mc{K})^\sigma$ and $w\in W^\tau$ such that $\Delta_{w\Lambda_i}(g)$ and $\Delta_{ws_\imath\Lambda_i}(g)$ are both non-zero. Then
\[
\frac{\prod_{k\neq i}\Delta_{w\Lambda_k}(g)^{-a_{ki}}}{\Delta_{w\Lambda_i}(g)\Delta_{ws_\imath\Lambda_i}(g)}\in \mc{K}_{\imath}.
\]
\end{corollary}
\begin{proof} If $m_{\imath}=m/|\imath|=1$, there is nothing to show. The standard orbits in the case $A_{2l}^{(2)}$ with $m=4$ follows from Corollary \ref{cor: A2l with m=4}. Thus, it remains to consider $\tau$-fixed points, i.e., $\imath=\{i\}$. Note that by Corollary \ref{cor: gamma invariant minor}, both $\Delta_{w\Lambda_i}(g)$ and $\Delta_{ws_\imath\Lambda_i}(g)$ are $\gamma$-invariant. 

On the one hand, if $h=0$, by Lemma \ref{lem: minor under gamma action}, the product $\prod_{k\neq i}\Delta_{w\Lambda_k}(g)^{-a_{ki}}$ is obviously $\gamma$-invariant as well, and therefore the whole quotient must be in the $\gamma$-invariant subfield of $\mc{K}$, which is $\mc{K}_{\imath}$. 

On the other hand, among the cases with $h\neq 0$, the only cases with $\tau$-fixed indices are $A_{2l-1}^{(2)}$ and $D_{l+1}^{(2)}$. Now for $k\neq i$, either $\tau(k)=k$ and hence $\Delta_{w\Lambda_k}(g)^{-a_{ki}}$ is $\gamma$-invariant by Corollary \ref{cor: gamma invariant minor}, or $\{k,k'\}$ is a $\tau$-orbit and both $\Delta_{w\Lambda_k}(g)$ and $\Delta_{w\Lambda_{k'}}(g)=\epsilon^{\la h,w\Lambda_k-\Lambda_k\ra}\gamma(\Delta_{w\Lambda_k}(g))$ (Lemma \ref{lem: minor under gamma action}) are present in the product with the same power $-a_{ki}$, which implies that the product is $\gamma$-invariant. Thus, the whole quotient must again be in $\mc{K}_{\imath}$.
\end{proof}

\begin{lemma}
    \label{lem: representative}
For any $g\in N^-(\mc{K})^\sigma$, there exists $h\in (N^-)^\sigma$, such that

\[ {\rm val}(  \Delta_{w\Lambda_i}(g \cdot h) )=D_{w\Lambda_\imath}(g) \quad \text{ for any $w\in W^\tau$ and any $\imath\in I_\tau$}  .  \]
\end{lemma}
 \begin{proof}
We adapt part of the proof of \cite[Theorem 4.5]{Kam} into our setting.

Let $\overline{V}(\Lambda_\imath)'$ be the span of all $n\cdot v_{\Lambda_i}, n\in (N^-)^\sigma$.  Then $\overline{V}(\Lambda_\imath)'$ is a rational $(N^-)^\sigma$-submodule of $\overline{V}(\Lambda_\imath)$. Hence, it is stable under the action of the hyperalgebra $\mathfrak{D}((N^-)^\sigma)$ of $(N^{-})^\sigma$.  Since $\overline{V}(\Lambda_\imath)$ is a Weyl module with $v_{\Lambda_i}$ as a highest weight vector,  by \cite[I.7.15, II.2.13]{Ja} we have $\mathfrak{D}((N^-)^\sigma)\cdot v_{\Lambda_i}=\overline{V}(\Lambda_\imath)$. Thus, $\overline{V}(\Lambda_\imath)'=\overline{V}(\Lambda_\imath)$.

Choose a basis $\{u_s\}$ in $\overline{V}(\Lambda_\iota)$ (as a subspace of $V(\Lambda_i)$). Let $\{u_s^*\}$ be the dual basis of $\{u_s\}$ in $\overline{V}(\Lambda_\iota)^*$ (as a quotient of $V(\Lambda_i)^*$).  
For any $h\in (N^-)^\sigma$,  
\[ h\cdot v_{\Lambda_i}=\sum_{s} c_s u_s \in  \overline{V}(\Lambda_\iota) .\]
Set $d=D_{w\Lambda_i}(g)$. Let $p_s$ be the coefficient of $t^{d}$ in $\la \xi_{w\Lambda}, g\cdot u_s      \ra$.  Since 
\[ D_{w\Lambda_i}(g)= \min_{s} \{ \val ( \la \xi_{w\Lambda_i}, g\cdot u_s   \ra  ) \} , \]
there exists $s$ such that $p_s$ is nonzero. Set $u^*= \sum_s p_s u_s^* \in \overline{V}(\Lambda_i)^*$.  
 Then, $u^*$ is nonzero, and ${\rm val}(  \Delta_{w\Lambda_i}(g \cdot h) )=d$ if and only if $\la u^*, h\cdot v_{\Lambda_i} \ra =\sum_s p_s c_s$ is nonzero.  

 For each $w\Lambda_{\iota}$ (with a choice of $i\in \iota$), we define a function $f_{w\Lambda_\iota}:  (N^-)^\sigma\to \rk$ with $f_{w\Lambda_\iota}(h)=\la u^*, h\cdot v_{\Lambda_i}  \ra$.
Then, $f_{w\Lambda_\iota}$ is nonzero regular function on $(N^-)^\sigma$.  Since $(N^-)^\sigma$ is an irreducible variety, we can find an element $h\in (N^-)^\sigma$ such that $f_{w\Lambda_{\iota}}(h)\not= 0$ for all $w\Lambda_{\iota}$. Thus, there exists $h\in (N^-)^\sigma$, such that for any $w\Lambda_\iota$,
 \[ {\rm val}(  \Delta_{w\Lambda_i}(g \cdot h) )=D_{w\Lambda_\imath}(g) . \qedhere\]
\end{proof}

\section{Configuration spaces of flags and the twist map}\label{ss: conf space} In this section, we review configuration spaces of flags and recall some of their basic properties. We define the twist map in terms of configurations of flags and establish several properties that will be useful in the next section. For convenience, we suppress the notation ``$(\mc{K})$'' for all spaces within this section. 

\subsection{Configurations of flags and angle unipotents} Fix a triple $H\subset B\subset G$. Let $N:=[B,B]$ be the maximal unipotent subgroup of $B$. Let $B^-$ be the opposite Borel subgroup, so that $B\cap B^-=H$. Let $N^-:=[B^-,B^-]$. We consider two presentations of flag varieties: $G/B$ and $B^-\backslash G$. Correspondingly, we define two versions of \emph{decorated flag varieties}: $G/N$ and $N^-\backslash G$. Note that $g\in G$ acts on $G/B$ and $G/N$ by left multiplication by $g$, and acts on $B^-\backslash G$ and $N^-\backslash G$ by right multiplication by $g^{-1}$. 
There are natural projections $G/N\rightarrow G/B$ and $N^-\backslash G\rightarrow B^-\backslash G$.

\begin{remark} For convenience, we denote flags by the symbol ``$B$'' and decorated flags by the symbol ``$A$''. We use subscripts for left cosets of $B$ and $N$, and superscripts for right cosets of $B^-$ and $N^-$.
\end{remark}

By the Bruhat decomposition $G=\bigsqcup_{w\in W} BwB=\bigsqcup_{w\in W}B^-wB^-$, we define the corresponding Tits distances as follows:
\begin{itemize}
    \item $(xB,yB)$ has Tits distance $w$ if $x^{-1}y\in BwB$; we denote this relation by $xB\xrightarrow{w}yB$.
    \item $(B^-x,B^-y)$ has Tits distance $w$ if $xy^{-1}\in B^-wB^-$; we denote this relation by $B^-x\xrightarrow{w}B^-y$.
\end{itemize}
By the Birkhoff decomposition $G=\bigsqcup_{w\in W}B^-wB$, we can define a Tits codistance:
\begin{itemize}
    \item $(B^-x,yB)$ has Tits codistance $w$ if $xy\in B^-wB$; we denote this relation by $B^-x\overset{w}{\mbox{ --- }}yB$. In particular, when two flags are of Tits codistance $e$, we omit the Weyl group element label and just denote the relation by $B^-x\mbox{ --- }yB$.
\end{itemize}

We will use the same arrow and dash notations among decorated flags to denote the Tits distances and Tits codistance relations between the underlying undecorated flags.

\smallskip



A pair of decorated flags satisfying the Tits distance relation $xN\xrightarrow{w} yN$ is said to be \emph{compatible} if $x^{-1}y\in N\overline{w}N$; we denote such a compatible pair of decorated flags by $xN\xRightarrow{w}yN$. 
A pair of decorated flags satisfying the Tits codistance relation $A^1\mbox{ --- }A_2$ is said to be a \emph{pinning} if there exists some $g$ such that $(g\cdot A^1,g\cdot A_2)=(N^-,N)$; we denote a pinning by $A^1\varequals{10pt} A_2$ (note that this is not an equal sign). 

\begin{lemma}[{\cite[Lemmas 2.5, 2.11]{ShenWeng}}]\label{lem: relative position} Suppose $w=uv$ in $W$ and $l(w)=l(u)+l(v)$. Then
\begin{enumerate}\item $B_1\xrightarrow{w}B_3$ if and only if there exists $B_2$ such that $B_1\xrightarrow{u}B_2\xrightarrow{v} B_3$. Moreover, such $B_2$ is unique if it exists.
\item $A_1\xRightarrow{w}A_3$ if and only if there exists $A_2$ such that $A_1\xRightarrow{u}A_2\xRightarrow{v} A_3$. Moreover, such $A_2$ is unique if it exists.
\item $B^1\overset{u}{\mbox{ --- }}B_3$ if and only if there exists $B_2$ such that $B^1\overset{w}{\mbox{ --- }}B_2\xrightarrow{v^{-1}}B_3$. Moreover, such $B_2$ is unique if it exists.
\item $B^1\overset{v}{\mbox{ --- }}B_3$ if and only if there exists $B^2$ such that $B^1\xrightarrow{u^{-1}}B^2\overset{w}{\mbox{ --- }}B_3$. Moreover, such $B^2$ is unique if it exists.
\end{enumerate}
\end{lemma}

\begin{lemma}\label{lem: possible Tits codistance} 
\begin{enumerate}
    \item Suppose $w\in W$ with $l(ws_i)>l(w)$. If $B^1\overset{w}{\mbox{ --- }}B_2\xrightarrow{s_i}B_3$, then the Tits codistance from $B^1$ to $B_3$ is either $w$ or $ws_i$.
    \item Suppose $w\in W$ with $l(s_iw)>l(w)$. If $B^1\overset{w}{\mbox{ --- }}B_2\xleftarrow{s_i}B_3$, then the Tits codistance from $B^1$ to $B_3$ is either $w$ or $s_iw$.
    \item If $B_1\xrightarrow{s_i}B_2\xrightarrow{s_i}B_3$, then the Tits distance from $B_1$ to $B_3$ is either $e$ or $s_i$.
\end{enumerate}
\end{lemma}
\begin{proof} The first two claims are proved in {\cite[Proposition A.8]{ShenWeng}}. The last claim follows from the fact that $Bs_iBs_iB=B\sqcup Bs_iB$.
\end{proof}

\begin{definition} A \emph{configuration} of flags is a tuple of (decorated) flags satisfying certain Tits (co)distance and compatibility conditions, considered up to a diagonal $G$-action.
\end{definition}

Elements inside the open subset $B^-B\subset G$ are said to be \emph{Gaussian decomposable}. Any Gaussian decomposable element $g$ can be factored uniquely as $[g]_-[g]_0[g]_+$ for $[g]_-\in N^-$, $[g]_0\in H$, and $[g]_+\in N$. For each fundamental weight $\Lambda_i$ we define a function $\Delta_{\Lambda_i}$ on the Gaussian decomposable elements by $\Delta_{\Lambda_i}(g):=\Lambda_i([g]_0)$. This function extends regularly to $G$ and agrees with the generalized minors defined in \eqref{eq: generalized minor}. Given a pair of decorated flags $(N^-x,yN)$, we define 
\[
\Delta_{\Lambda_i}(N^-x,yN):=\Delta_{\Lambda_i}(xy).
\]

\begin{lemma}[{\cite[\S 7.2]{GLS}}]\label{lem: Gaussian decomposable} An element $g\in G$ is Gaussian decomposable if and only if $\Delta_{\Lambda_i}(g)\neq 0$ for all $\Lambda_i$.
\end{lemma}

\begin{lemma}\label{lem: Gaussian decomposition} Suppose $g\in G$ is Gaussian decomposable and $\sigma(g)=g$. Then $\sigma([g]_-)=[g]_-$, $\sigma([g]_0)=[g]_0$, and $\sigma([g]_+)=[g]_+$.
\end{lemma}
\begin{proof} Since $\sigma$ preserves $N^-$, $H$, and $N$, $\sigma([g]_-)\sigma([g]_0)\sigma([g]_+)$ is a Gaussian decomposition of $\sigma(g)=g$. The statement then follows from the uniqueness of the Gaussian decomposition.
\end{proof}

Given a triple $B_1\mbox{ --- } A^2\mbox{ --- } B_3$, there exist unique elements $g\in G$ and $y\in N^-$ such that $(g\cdot B_1,g \cdot A^2,g\cdot B_3)=(B,N^-,yB)$. We call $y$ the \emph{angle unipotent} associated with the triple $(B_1,A^2,B_3)$ and denote it by $u(B_1,A^2,B_3)$. Note that if $B_1\xrightarrow{w}B_3$, then $u(B_1,A^2,B_3)$ is an element of $N^-(w):=N^-\cap BwB$.

\begin{lemma}\label{lem: in N^-(w_0)} An element $g\in N^-$ is in $N^-(w_0)$ if and only if $\Delta_{w_0\Lambda_i}(g)\neq 0$ for all $\Lambda_i$.
\end{lemma}
\begin{proof} Note that 
\[
\Delta_{w_0\Lambda_i}(g)=\la \xi_{w_0\Lambda_i},g\cdot v_{\Lambda_i}\ra=\la \overline{w}_0\cdot \xi_{\Lambda_i},g\cdot v_{\Lambda_i}\ra=\la \xi_{\Lambda_i}, \overline{w}_0^{-1}g\cdot v_{\Lambda_i}\ra=\Delta_i(\overline{w}_0^{-1}g).
\]
Thus, by Lemma \ref{lem: Gaussian decomposable}, $\Delta_{w_0\Lambda_i}(g)\neq 0$ for all $\Lambda_i$ implies that $\overline{w}_0^{-1}g\in B^{-1}B$, which is equivalent to $g\in (w_0B^-w_0)w_0B=Bw_0B$.
\end{proof}





\begin{lemma}[{\cite[Corollary 3.22]{ShenWeng}}]\label{lem: Lusztig variables in terms of A variables} Suppose $\begin{tikzpicture}[scale=0.6,baseline=5] \node (1) at (0,0) [] {$B_1$};
\node (2) at (1,2) [] {$A^2$};
\node (3) at (2,0) [] {$B_3$};
\draw (1) -- (2);
\draw (2) -- (3);
\draw [->] (1) -- node [below] {$s_i$} (3);
\end{tikzpicture}$. Let $A_1\xRightarrow{s_i}A_3$ be compatible lifts of $B_1\xrightarrow{s_i}B_3$. Let $\Delta_j:=\Delta_{\Lambda_j}(A^2,A_1)$ and $\Delta'_j:=\Delta_{\Lambda_j}(A^2,A_3)$. Then
\begin{enumerate}
    \item $\Delta_j=\Delta'_j$ for $j\neq i$.
    \item $u(B_1,A^2,B_3) = y_i(f)$ for
    \[
    f=\frac{\prod_{j\neq i} \Delta_j^{-a_{ji}}}{\Delta_i\Delta'_i}
    \]
    where $a_{ji}$ is the Cartan matrix of $G$.
\end{enumerate}
\end{lemma}

\begin{proposition}\label{prop: vanishing minor implies special angle unipotent} Suppose $a_{ij}=a_{ji}=-1$ in the Cartan matrix and suppose $\begin{tikzpicture}[scale=0.6,baseline=5] \node (1) at (0,0) [] {$B_1$};
\node (2) at (1.5,2) [] {$A^2$};
\node (3) at (3,0) [] {$B_3$};
\draw (1) -- (2);
\draw (2) -- (3);
\draw [->] (1) -- node [below] {$s_is_js_i$} (3);
\end{tikzpicture}$. Let $A_1\xRightarrow{s_is_js_i}A_3$ be compatible lifts of $B_1\xrightarrow{s_is_js_i}B_3$, and let $A_4$, $A_5$, $A_6$, and $A_7$ be decorated flags that fit into the diagram below (the existence and uniqueness of these flags follow from Lemma \ref{lem: relative position}).
\[
\begin{tikzpicture}
    \node (1) at (0,0) [] {$A_1$};
    \node (4) at (1.5,0.5) [] {$A_4$};
    \node (5) at (3,0.5) [] {$A_5$};
    \node (3) at (4.5,0) [] {$A_3$};
    \node (6) at (1.5,-0.5) [] {$A_6$};
    \node (7) at (3,-0.5) [] {$A_7$};
    \draw [-implies, double, double distance=1pt] (1) -- node [above left] {$s_i$} (4);
    \draw [-implies, double, double distance=1pt] (4) -- node [above] {$s_j$} (5);
    \draw [-implies, double, double distance=1pt] (5) -- node [above right] {$s_i$} (3);
    \draw [-implies, double, double distance=1pt] (1) -- node [below left] {$s_j$} (6);
    \draw [-implies, double, double distance=1pt] (6) -- node [below] {$s_i$} (7);
    \draw [-implies, double, double distance=1pt] (7) -- node [below right] {$s_j$} (3);
\end{tikzpicture}
\]
Then the following are equivalent.
\begin{enumerate}
    \item $u(B_1,A^2,B_3)=y_i(1)y_j(-f)y_i(-1)y_j(f)$ for some $f\neq 0$.
    \item $\Delta_{\Lambda_i}(A^2,A_4)=\Delta_{\Lambda_i}(A^2,A_5)=\Delta_{\Lambda_j}(A^2,A_6)=\Delta_{\Lambda_j}(A^2,A_7)=0$.
    \item At least one of $\Delta_{\Lambda_i}(A^2,A_4)$ and $\Delta_{\Lambda_i}(A^2,A_5)$ vanishes, and at least one of $\Delta_{\Lambda_j}(A^2,A_6)$ and $\Delta_{\Lambda_j}(A^2,A_7)$ vanishes.
\end{enumerate}
\end{proposition}
\begin{proof} (1)$\implies$(2). By Lemma \ref{lem: possible Tits codistance}, the Tits codistances from $A^2$ to $A_4$ and $A_5$, respectively, can only be both $e$ or both $s_i$. We claim that they cannot be $e$; otherwise, by Lemma \ref{lem: Lusztig variables in terms of A variables}, there must exist non-zero scalars $f_1$, $f_2$, and $f_3$ such that $u(B_1,A^2,B_4)=y_i(f_1)$, $u(B_4,A^2,B_5)=y_j(f_2)$, and $u(B_5,A^2,B_3)=y_i(f_3)$ (here $B_4$ and $B_5$ are the underlying undecorated flags of $A_4$ and $A_5$). This implies that
\[
y_i(1)y_j(-f)y_i(-1)y_j(f)=u(B_1, A^2,B_3)=y_i(f_1)y_j(f_2)y_i(f_3),
\]
which is impossible by Lemma \ref{lem: change of word in special case} (2).

Now given that $A^2\overset{s_i}{\mbox{ --- }}A_4$, we can act diagonally by some element of $G$ and assume that $(A_1,A^2,A_4)=(tN, N^-,t\overline{s}_i^{-1}N)$. This implies that
\[
\Delta_{\Lambda_i}(A^2,A_4)=\Delta_{\Lambda_i}\left(t\overline{s}_i^{-1}\right)=\Lambda_i(t)\la \overline{s}_i\cdot \xi_{\Lambda_i}, v_{\Lambda_i}\ra=\Lambda_i(t)\la \xi_{s_i\Lambda_i},v_{\Lambda_i}\ra=0.
\]
The vanishing of the other three functions can be proved similarly as well. 

(2)$\implies$(3) is obvious.

(3)$\implies$(1). By the same argument as in the proof of (1)$\implies$(2) part, we know that when (3) is true, we must have the following Tits codistances as shown in the left picture of Figure \ref{fig: special unipotent computation}.
\begin{figure}[H]
\begin{tikzpicture}[scale=1.3]
    \node (1) at (0,0) [] {$B_1$};
    \node (4) at (1.5,0.8) [] {$B_4$};
    \node (5) at (3,0.8) [] {$B_5$};
    \node (3) at (4.5,0) [] {$B_3$};
    \node (6) at (1.5,-0.8) [] {$B_6$};
    \node (7) at (3,-0.8) [] {$B_7$};
    \draw [->] (1) -- node [above] {$s_i$} (4);
    \draw [->] (4) -- node [below] {$s_j$} (5);
    \draw [->] (5) -- node [above] {$s_i$} (3);
    \draw [->] (1) -- node [below left] {$s_j$} (6);
    \draw [->] (6) -- node [below] {$s_i$} (7);
    \draw [->] (7) -- node [below right] {$s_i$} (3);
    \node (2) at (2.25,3.5) [] {$A^2$};
    \draw (1) -- (2) -- node [below left] {$s_i$} (4);
    \draw (2) -- node [below] {\quad$s_j$}(6);
    \draw (2) -- node [below] {$s_j$\quad\quad} (7);
    \draw (5) -- node [below right] {$s_i$} (2) -- (3);
\end{tikzpicture}\hspace{1cm}
\begin{tikzpicture}[scale=1.3]
    \node (1) at (0,0) [] {$B_1$};
    \node (4) at (1.5,0.8) [] {$B_4$};
    \node (5) at (3,0.8) [] {$B_5$};
    \node (3) at (4.5,0) [] {$B_3$};
    \node (6) at (1.5,-0.8) [] {$B_6$};
    \node (7) at (3,-0.8) [] {$B_7$};
    \draw [->] (1) -- node [above] {$s_i$} (4);
    \draw [->] (4) -- node [below] {$s_j$} (5);
    \draw [->] (5) -- node [above] {$s_i$} (3);
    \draw [->] (1) -- node [below left] {$s_j$} (6);
    \draw [->] (6) -- node [below] {$s_i$} (7);
    \draw [->] (7) -- node [below right] {$s_i$} (3);
    \node (2) at (2.25,3.5) [] {$A^2$};
    \draw (1) -- (2) -- node [below left] {$s_i$} (4);
    \draw (2) -- node [below] {\quad$s_j$}(6);
    \draw (2) -- node [below] {$s_j$\quad\quad} (7);
    \draw (5) -- node [below right] {$s_i$} (2) -- (3);
    \node [blue] (8) at (1,0) [] {$B_8$};
    \draw [->,blue] (1) -- node [below right] {$s_i$} (8);
    \draw [blue] (8) -- (2);
    \draw [blue,->] (8) -- node [right] {$s_i$} (4);
    \node [blue] (9) at (2.25,0) [] {$B_9$};
    \node [blue] (10) at (3.5,0) [] {$B_{10}$};
    \draw [blue,->] (8) -- node [below] {$s_j$} (9);
    \draw [blue,->] (9) -- node [above] {$s_i$} (10);
    \draw [blue,->] (10) --node [above left] {$s_j$} (3);
    \draw [blue,->] (10) -- node [left] {$s_j$} (7);
    \draw [blue] (9) -- (2) -- (10);
\end{tikzpicture}
\caption{Illustration of the step (3)$\implies$(1).}\label{fig: special unipotent computation}
\end{figure}
Now let $B_8$ be the unique flag with $A^2\mbox{ --- }B_8$ and $u(B_1,A^2,B_8)=y_i(1)$. Note that $B_8\neq B_4$ because $A^2\overset{s_i}{\mbox{ --- }}B_4$; then by Lemma \ref{lem: possible Tits codistance} (3), we can deduce that $B_8\xrightarrow{s_i}B_4$. Then by Lemma \ref{lem: relative position} (1), we know that $B_8\xrightarrow{s_is_js_i}B_3$. Now let $B_9$ and $B_{10}$ be the unique flags that satisfy $B_8\xrightarrow{s_j}B_9\xrightarrow{s_i}B_{10}\xrightarrow{s_j}B_3$. 

We claim that the Tits codistance from $A^2$ to either of $B_9$ and $B_{10}$ is $e$. To see this, suppose $A^2\overset{s_j}{\mbox{ --- }}B_{10}$. Then by Lemma \ref{lem: relative position} (3), we have $B_{10}=B_7$; but then the Tits distance from $B_1$ to $B_{10}=B_7$ is $s_is_js_i$ if we go through $B_8$ and $B_9$, and is $s_js_i$ if we go through $B_6$, which is absurd. Thus, we must have $A^2\mbox{ --- }B_{10}$ and hence $A^2\mbox{ --- }B_9$ as well.

Since the chain $B_8\xrightarrow{s_j}B_9\xrightarrow{s_i} B_{10}\xrightarrow{s_j}B_3$ cannot be switched to a chain with reduced word $s_is_js_i$ while maintaining the general position between $A^2$ and the intermediate flags, we know that the angle unipotents $u(B_8,A^2,B_9)=y_j(-f)$ and $u(B_{10},A^2,B_3)=y_j(f)$ for some non-zero scalar $f$ (Proposition \ref{prop: Lusztig transform}).

Lastly, by the same argument as for the Tits distance from $B_8$ to $B_4$, we know that $B_7\xrightarrow{s_j}B_{10}$. But then we observe that the chain $B_1\xrightarrow{s_i}B_8\xrightarrow{s_j}B_9\xrightarrow{s_i}B_{10}$ cannot be switched to a chain with reduced word $s_js_is_j$ while maintaining the general position between $A^2$ and the intermediate flags. Thus, $u(B_9,A^2,B_{10})=u(B_1,A^2,B_8)^{-1}=y_i(-1)$. By multiplying the angle unipotents together, we can conclude that
\begin{align*}
u(B_1,A^2,B_3)=&u(B_1,A^2,B_8)\cdot u(B_8,A^2,B_9)\cdot u(B_9,A^2,B_{10})\cdot u(B_{10},A^2,B_3)\\
=&y_i(1)y_j(-f)y_i(-1)y_j(f)
\end{align*}
for some scalar $f\neq 0$.
\end{proof}

\begin{lemma}\label{lem: special case} Suppose $a_{ij}=a_{ji}=-1$ in the Cartan matrix and suppose $
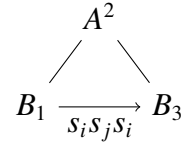
\begin{tikzpicture}[scale=0.6,baseline=5] \node (1) at (0,0) [] {$B_1$};
\node (2) at (1.5,2) [] {$A^2$};
\node (3) at (3,0) [] {$B_3$};
\draw (1) -- (2);
\draw (2) -- (3);
\draw [->] (1) -- node [below] {$s_is_js_i$} (3);
\end{tikzpicture}$ with
\[
u(B_1,A^2,B_3)=y_i(1)y_j(-f)y_i(-1)y_j(f)
\]
for some non-zero scalar $f$. Let $A_1\xRightarrow{s_is_js_i}A_3$ be compatible lifts of $B_1\xrightarrow{s_is_js_i}B_3$. Let $\Delta_j:=\Delta_{\Lambda_j}(A^2,A_1)$ and $\Delta'_j:=\Delta_{\Lambda_j}(A^2,A_3)$. Then
\begin{enumerate}
    \item $\Delta_k=\Delta'_k$ for $k\neq i,j$.
    \item $f=\dfrac{\prod_{k\neq i,j}\Delta_k^{-a_{ki}-a_{kj}}}{\Delta_j\Delta'_i}=-\dfrac{\prod_{k\neq i,j}\Delta_k^{-a_{ki}-a_{kj}}}{\Delta_i\Delta'_j}$.
\end{enumerate}
\end{lemma}
\begin{proof} Let $y(f)=y_i(1)y_j(-f)y_i(-1)y_j(f)$ and $x(f)=x_j(1)x_i(-f)x_j(-1)x_i(f)$. Then there exists some $n\in N$ such that
\[
x(1/f)\overline{s}_i\overline{s}_j\overline{s}_i=y(f)\cdot \check{\alpha}_i(1/f)\cdot\check{\alpha}_j(-1/f) \cdot n.
\]
Without loss of generality, we may assume that 
\[
(A_1,A^2,A_3)=(tN, N^-,x(1/f)\overline{s}_i\overline{s}_j\overline{s}_it'N),
\]
with $t'=s_is_js_i(t)$. Note that $\Delta_k=\Lambda_k(t)$ and $\Delta'_k=\Lambda_k\left(\check{\alpha}_i(1/f)\cdot\check{\alpha}_j(-1/f) \cdot t'\right)$

Now for $k\neq i,j$,
\[
\Delta_k=\Lambda_k(t)=\Lambda_k(s_is_js_i(t))=\Lambda_k(t')=\Lambda_k\left(\check{\alpha}_i(1/f)\cdot\check{\alpha}_j(-1/f) \cdot t'\right)=\Delta'_k,
\]
proving claim (1). 

For claim (2), note that
\[
\Delta'_i=\Lambda_i\left(\check{\alpha}_i(1/f)\cdot\check{\alpha}_j(-1/f) \cdot t'\right)=\frac{1}{f}\Lambda_i(s_is_js_i(t))=\frac{1}{f}\Lambda_i(t)\alpha_i(t^{-1})\alpha_j(t^{-1}).
\]
In total, we have
\[
\frac{\prod_{k\neq i,j}\Delta_k^{-a_{ki}-a_{kj}}}{\Delta_j\Delta'_i}=\frac{\prod_{k\neq i,j} \Lambda_k(t)^{-a_{ki}-a_{kj}}}{\frac{1}{f}\Lambda_j(t)\Lambda_i(t)\alpha_i(t^{-1})\alpha_j(t^{-1})}=f,
\]
where the last equality follows from the fact that 
\begin{align*}
\alpha_i+\alpha_j-\Lambda_i-\Lambda_j-\sum_{k\neq i,j}(a_{ki}+a_{kj})\Lambda_k=&\left(\sum_k(a_{ki}+a_{kj})\Lambda_k\right)-\Lambda_i-\Lambda_j-\sum_{k\neq i,j}(a_{ki}+a_{kj})\Lambda_k\\=&0.
\end{align*}

The identity $f=-\dfrac{\prod_{k\neq i,j}\Delta_k^{-a_{ki}-a_{kj}}}{\Delta_i\Delta'_j}$ in (2) then follows by first applying Lemma \ref{lem: change of word in special case} (1) to $u(B_1,A^2,B_3)$ and then doing the same proof with $i$ and $j$ switched.
\end{proof}


\subsection{Twist map and its properties} Next, we give an alternative description of the twist map of Berenstein and Zelevinsky in terms of configurations of flags, and provide new geometric proofs of a few identities (e.g. \cite[Theorem 4.4]{Kam} and \cite[Proposition 5.4]{BZ}) that will be useful to us. Let us start by introducing a modified version of transposition on $G$.

\begin{definition} We define an anti-involution $t:g\mapsto g^t$ on $G$ called the \emph{(modified) transposition}. It is defined on the generators of $G$ by
\begin{align*}
x_i(f)^t:=&\left\{\begin{array}{ll} y_i(f) & \text{if $\left(X_N^{(r)},m,i\right)\neq \left(A_{2l}^{(2)},4,l\right)$},\\
y_i(-f) & \text{if $\left(X_N^{(r)},m,i\right)= \left(A_{2l}^{(2)},4,l\right)$},\end{array}\right.\\
h^t:=&h \quad \text{for all $h\in H$}.
\end{align*}
\end{definition}

\begin{remark} Note that except for the case $A_{2l}^{(2)}$ with $m=4$, our modified transposition coincides with the usual transposition on $G$. Thus, we will drop the word ``modified'' and refer to $g^t$ simply as the ``transposition'' of $g$.
\end{remark}

\begin{lemma}\label{lem: transposition of Weyl group element lift} For each Weyl group element $w$, there exists some element $h_w\in H(k)$ such that $\overline{w}^t=h_w\overline{w}^{-1}$. Moreover, if any cases other than $A_{2l}^{(2)}$ with $m=4$, $h_w=e$ for all $w$.
\end{lemma}
\begin{proof} On the one hand, when $\left(X_N^{(r)},m,i\right)\neq \left(A_{2l}^{(2)},4,l\right)$, we have
\[
\overline{s}_i^t=\left(x_i(-1)y_i(1)x_i(-1)\right)^t=y_i(-1)x_i(1)y_i(-1)=\overline{s}_i^{-1}.
\]
On the other hand, when $\left(X_N^{(r)},m,i\right)=\left(A_{2l}^{(2)},4,l\right)$, we have
\[
\overline{s}_l^t=y_l(1)x_l(-1)y_l(1)=\overline{s}_l=\check{\alpha}_l(-1)\overline{s}_l^{-1}.
\]
Let $(i_1,i_2,\dots, i_k)$ be a reduced word for $w$. Then
\[
\overline{w}^t=\overline{s}_{i_l}^t\overline{s}_{i_{l-1}}^t\cdots \overline{s}_{i_1}^t=h_w\overline{s}_{i_l}^{-1}\overline{s}_{i_{l-1}}^{-1}\cdots \overline{s}_{i_1}^{-1}=h_w\overline{w}^{-1},
\]
where $h_w$ is obtained by pushing the possible factors of the form $\check{\alpha}_l(-1)$ all the way to the left.
\end{proof}

\begin{lemma} The Tits (co)distances satisfy the following properties under transposition:
\begin{enumerate}
    \item $B_1\xrightarrow{w}B_2$ if and only if $(B_2)^t\xrightarrow{w^{-1}}(B_1)^t$;
    \item $B^1 \mbox{ --- } B_1$ if and only if $(B_1)^t  \mbox{ --- } (B^1)^t$;
    \item $A^1 \varequals{10pt} A_1$ if and only if $(A_1)^t\varequals{10pt}(A^1)^t$.
\end{enumerate}
\end{lemma}
\begin{proof} These follow from the fact that $\overline{w}^t$ is a lift of $w^{-1}$, and transposition swaps $B$ with $B^-$, and swaps $N$ with $N^-$. 
\end{proof}

\begin{lemma}\label{lem: transposition} The automorphism $\sigma$ commutes with the transposition $t$.
\end{lemma}
\begin{proof} This identity is obviously true on $H$. Thus, it suffices to verify it on the generators $x_i(f)$ and $y_i(f)$. When $m=2$, we have $\epsilon=-1$ and therefore
\begin{align*}
\sigma(x_i(f)^t)=\sigma(y_i(f))=&y_{\tau(i)}\left((-1)^{-\alpha_i(h)}\gamma(f)\right)=y_{\tau(i)}\left((-1)^{\alpha_i(h)}\gamma(f)\right)\\
=&x_{\tau(i)}\left((-1)^{\alpha_i(h)}\gamma(f)\right)^t=\left(\sigma(x_i(f))\right)^t.
\end{align*}

When $m=3$, $h=0$ by Eq. \eqref{eq_h}, and therefore $\alpha_i(h)=0$ for all $i$. It then follows that
\[
\sigma(x_i(f)^t)=\sigma(y_i(f))=y_{\tau(i)}(\gamma(f))=x_{\tau(i)}(\gamma(f))^t=(\sigma(x_i(f)))^t.
\]

When $m=4$, $\alpha_i(h)=0$ for all $i$ except $i=l,l+1$. Thus, for $i\neq l,l+1$, we can show $\sigma(x_i(f))^t=(\sigma(x_i(f)))^t$ in the same way as the $m=3$ case above. For $i=l$, we have
\[
\sigma(x_l(f)^t)=\sigma(y_l(-f))=y_{l+1}(-\epsilon^{-1}\gamma(f))=x_{l+1}(\epsilon\gamma(f))^t=(\sigma(x_l(f)))^t,
\]
and for $i=l+1$, we have
\[
\sigma(x_{l+1}(f)^t)=\sigma(y_{l+1}(f))=y_{l}(\epsilon^{-1}\gamma(f))=x_l(\epsilon\gamma(f))^t=(\sigma(x_{l+1}(f)))^t.
\]

Similar computations prove that $\sigma(y_i(f)^t)=\left(\sigma(y_i(f))\right)^t$ in all cases as well.
\end{proof}

\begin{proposition}\label{prop: rotation} Suppose we have the following configuration of a quadruple of flags associated with $G$.
\[
\begin{tikzpicture}[scale=1.2]
    \node (1) at (2,2) [] {$A^1$};
    \node (2) at (2,0) [] {$A_2$};
    \node (3) at (0,0) [] {$B_3$};
    \node (4) at (0,2) [] {$B^4$};
\draw [double, double distance=1pt] (2) -- (1);
    \draw [->] (3) -- node [below] {$w$} (2);
    \draw [->] (4) -- node [above] {$w^{-1}$} (1);
    \draw (3) -- node [left] {$w^{-1}$} (4);
    \draw (1) -- (3);
    \draw (2) -- (4);
\end{tikzpicture}
\]
Let $g:=u(B_3,A^1,B_2)$ and $y:=u\left((B^1)^t, (A_2)^t, (B^4)^t\right)$. Then 
\begin{enumerate}
  \item $B^4$ is uniquely determined by the pair $(B^1,B_3)$.
  \item $B_3$ is uniquely determined by the pair $(B^4, B_2)$.
    \item both $y$ and $g$ are in $N^-(w)$ and they are related by
\[
y=[g^tw]_- \quad \text{and} \quad g=\left([wy^{-1}]_+^t\right)^{-1}.
\]
\end{enumerate}
\end{proposition}
\begin{proof} (1) and (2) are proved in \cite[Lemma 2.5]{ShenWeng}. For (3), to prove $y=[g^tw]_-$, without loss of generality, let us start with the configuration 
\[
(A^1,A_2,B_3,B^4) = (N^-, gN, B, B^-w^{-1}).
\]
It is clear that $u(B_3,A^1,B_2)=g$. After taking transposition, we see that $((B^1)^t, (A_2)^t, (B^4)^t)=(B, N^-g^t, wB)$; then we apply the action by $g^t$ and get
\[
g^t\cdot ((B^1)^t, (A_2)^t, (B^4)^t)=(B,N^-,g^twB),
\]
which yields $u((B^1)^t, (A_2)^t, (B^4)^t)=[g^tw]_-$.

Conversely, to prove $g=\left([wy^{-1}]^t_+\right)^{-1}$, we start with the configuration
\[
(A^1,A_2,B_3,B^4)=\left(N^-(y^t)^{-1},N,w^{-1}B,B^-\right).
\]
It is clear that $u\left((B^1)^t, (A_2)^t,(B^4)^t\right)=y$. Now by acting on $(B_3,A^1,B_2)=(w^{-1}B,N^-(y^t)^{-1},B)$ by $(y^t)^{-1}$, we get
\[
(y^t)^{-1}\cdot (B_3,A^1,B_2)=((y^t)^{-1}w^{-1}B,N^-,B).
\]
Thus,
\[
g:=u(B_3,A^1,B_2)=\left[(y^t)^{-1}w^{-1}\right]_-^{-1}=\left[(wy^{-1})^t\right]_-^{-1}=\left([wy^{-1}]_+^t\right)^{-1}. \qedhere
\]
\end{proof}

\begin{remark} The expressions $\left[g^tw\right]_-$ and $[wy^{-1}]_+$ are well-defined because different lifts of $w$ differ only by multiplication by an element of $H$ on the left or right, which does not affect the corresponding Gaussian decomposition.
\end{remark}

\begin{definition} The \emph{twist map} on $N^-(w)$ is defined by
\begin{align*}
    \eta_w:N^-(w)&\longrightarrow N^-(w)\\
    g & \longmapsto y=[g^tw]_-.
\end{align*}
Its \emph{inverse twist map} is then defined by
\begin{align*}
    \eta_w^{-1}:N^-(w)&\longrightarrow N^-(w)\\
    y & \longmapsto g=\left([wy^{-1}]_+^t\right)^{-1}.
\end{align*}
\end{definition}

\begin{lemma}\label{lem: sigma invariant} If $g,y\in N^-(w)(\mathcal{K})$ satisfy $y=\eta_w(g)$, then $\sigma(g)=g$ if and only if $\sigma(y)=y$.
\end{lemma}
\begin{proof} This follows from the fact that the automorphism $\sigma$ commutes with group inverse (since it is a group automorphism), transposition (Lemma \ref{lem: transposition}), and Gaussian decomposition (Lemma \ref{lem: Gaussian decomposition}).
\end{proof}

\begin{proposition}\label{prop: minor recursive formula} Let $i_\bullet=(i_1,i_2,\dots, i_m)$ be a reduced word of $w$. Let $a_1,a_2,\dots, a_m$ be elements in $\mc{K}$ and set
\[
y=y_{i_1}(a_1) y_{i_2}(a_2)\cdots y_{i_m}(a_m)\in N^-(w).
\]
Let $g:=\eta_w^{-1}(y)$. For $0\leq k\leq m$, define $w^{i_\bullet}_k:=s_{i_1}s_{i_2}\cdots s_{i_k}$. Then
\begin{enumerate}
    \item 
    \[a_k=\lambda^{i_\bullet}_k\dfrac{\prod_{j\neq i_k}\Delta_{w^{i_\bullet}_k\Lambda_j}(g)^{-a_{ji_k}}}{\Delta_{w^{i_\bullet}_{k-1}\Lambda_{i_k}}(g)\Delta_{w^{i_\bullet}_k\Lambda_{i_k}}(g)}
    \]
    for some non-zero constant $\lambda^{i_\bullet}_k\in k^\times$ that only depends on the reduced word $i_\bullet$ and the position $k$.
    \item If $(i_{s-1},i_s,i_{s+1})=(i,j,i)$ for $i$ and $j$ with $a_{ij}=a_{ji}=-1$, and we replace the product $y_{i_{s-1}}(a_{s-1})y_{i_s}(a_s)y_{i_{s+1}}(a_{s+1})$ in the definition of $y$ with $y_i(1)y_j(-a)y_i(-1)y_j(a)$, then (1) will still hold for all $k\neq s-1,s,s+1$ and 
    \[
    a=\lambda \dfrac{\prod_{r\neq i,j}\Delta_{w^{i_\bullet}_s\Lambda_r}(g)^{-a_{ri}-a_{rj}}}{\Delta_{w^{i_\bullet}_{s-2}\Lambda_i}(g)\Delta_{w^{i_\bullet}_{s+1}\Lambda_j}(g)}
    \]
    for some non-zero constant $\lambda\in k^\times$ that only depends on the reduced word $i_\bullet$ and the position $s$.
\end{enumerate}

\end{proposition}
\begin{proof} By following the proof of Proposition \ref{prop: rotation}, if we start with the configuration $(B_3,A^1,A_2)=(B, N^-, gN)$, then $((B^1)^t,(A_2)^t,(B^4)^t)=(B, N^-g^t, wB)$. We lift $B\xrightarrow{w}wB$ to the compatible pair $N\xRightarrow{\overline{w}}\overline{w}N$, and divide the triangle spanned by $(A^1)^t$, $(A_2)^t$, and $(A^4)^t$ into smaller triangles as below. Each small triangle now has an angle unipotent $y_{i_k}(a_k)$ for $k=1,2,\dots, m$. 
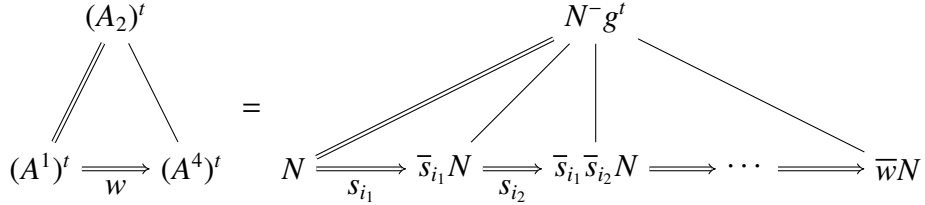
\[
\begin{tikzpicture}[baseline=20]
    \node (1) at (0,0) [] {$(A^1)^t$};
    \node (2) at (1,2) [] {$(A_2)^t$};
    \node (4) at (2,0) [] {$(A^4)^t$};
    \draw [-implies, double, double distance=1pt] (1) -- node [below] {$w$} (4);
    \draw [double, double distance=1pt](1) -- (2);
    \draw (2) -- (4);
\end{tikzpicture} =\begin{tikzpicture}[baseline=20]
\node (t) at (4,2) [] {$N^-g^t$};
\node (0) at (0,0) [] {$N$};
\node (1) at (2,0) [] {$\overline{s}_{i_1}N$};
\node (2) at (4,0) [] {$\overline{s}_{i_1}\overline{s}_{i_2}N$};
\node (3) at (6,0) [] {$\cdots$};
\node (4) at (8,0) [] {$\overline{w}N$};
\draw [-implies, double, double distance=1pt] (0) -- node [below] {$s_{i_1}$} (1);
\draw [-implies, double, double distance=1pt] (1) -- node [below] {$s_{i_2}$} (2);
\draw [-implies, double, double distance=1pt] (2) --  (3);
\draw [-implies, double, double distance=1pt] (3) --  (4);
\draw [double, double distance=1pt] (t) -- (0);
\foreach \i in {1,2,4}
{
\draw (t) -- (\i);
}
\end{tikzpicture}
\]
By applying Lemma \ref{lem: Lusztig variables in terms of A variables} to each small triangle, we get that
\[
a_k=\frac{\prod_{j\neq i_k}\Delta_{\Lambda_j}\left(g^t\overline{w^{i_\bullet}_k}\right)^{-a_{ji_k}}}{\Delta_{\Lambda_{i_k}}\left(g^t\overline{w^{i_\bullet}_{k-1}}\right)\Delta_{\Lambda_{i_k}}\left(g^t\overline{w^{i_\bullet}_k}\right)}.
\]
Now (1) follows from the fact that for any $u\in W$ and fundamental weight $\Lambda_i$, 
\[
\Delta_{\Lambda_i}(g^t\overline{u})=\Delta_{\Lambda_i}(\overline{u}^tg)=\Delta_{\Lambda_i}(h_u\overline{u}^{-1}g)=\la \xi_{\Lambda_i}, h_u\overline{u}^{-1}g\cdot v_{\Lambda_i}\ra=\Lambda_i(h_u)\la \xi_{u\Lambda_i},g\cdot v_{\Lambda_i}\ra=\Lambda_i(h_u)\Delta_{u\Lambda_i}(g),
\]
where the maximal torus element $h_u\in T(k)$ exists by Lemma \ref{lem: transposition of Weyl group element lift}.

The formula for $a$ in claim (2) can be obtained in a similar way by replacing Lemma \ref{lem: Lusztig variables in terms of A variables} with Lemma \ref{lem: special case}. 
\end{proof}

\begin{proposition}\label{prop: y=y2y1} Suppose $w=w_1w_2$ with $l(w)=l(w_1)+l(w_2)$. Let $y_1\in N^-(w_1)$ and $y_2\in N^-(w_2)$. Then $y:=y_1y_2\in N^-(w)$. Let $g=\eta_{w}^{-1}(y)$. Then $g=g_2g_1$ with $g_1=\eta_{w_1}^{-1}(y_1)$ and $g_2\in N^-\cap \overline{w}_1N^-\overline{w}_1^{-1}$.
\end{proposition}
\begin{proof} Let us consider the starting configuration in the left picture below. We apply the transposition to obtain the picture on the right.
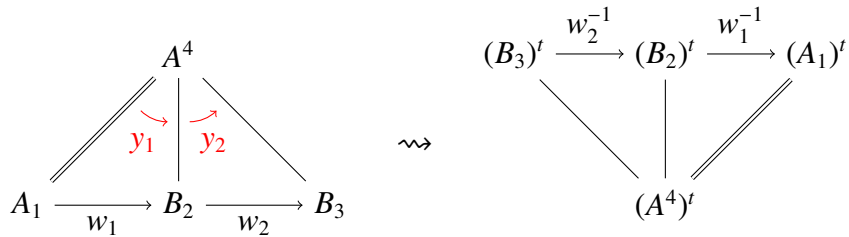
\[
\begin{tikzpicture}[baseline=20]
    \node (1) at (0,0) [] {$A_1$};
    \node (2) at (2,0) [] {$B_2$};
    \node (3) at (4,0) [] {$B_3$};
    \node (4) at (2,2) [] {$A^4$};
    \draw [double, double distance=1pt] (4) -- (1);
    \draw [->] (1) -- node [below] {$w_1$} (2);
    \draw 
    [->] (2) -- node [below] {$w_2$} (3);
    \draw (4) -- (2);
    \draw (4) -- (3);
    \draw [red, ->] (1.5,1.25) arc (-135:-90:0.5) node [below left] {$y_1$};
    \draw [red, <-] (2.5,1.25) arc (-45:-90:0.5) node [below right] {$y_2$};
\end{tikzpicture} \quad \rightsquigarrow \quad
\begin{tikzpicture}[baseline=20]
    \node (1) at (4,2) [] {$(A_1)^t$};
    \node (2) at (2,2) [] {$(B_2)^t$};
    \node (3) at (0,2) [] {$(B_3)^t$};
    \node (4) at (2,0) [] {$(A^4)^t$};
    \draw [double, double distance=1pt] (4) -- (1);
    \draw [<-] (1) -- node [above] {$w_1^{-1}$} (2);
    \draw [<-] (2) -- node [above] {$w_2^{-1}$} (3);
    \draw (4) -- (2);
    \draw (4) -- (3);
\end{tikzpicture}
\]
We then find the unique $B_5$ between $(B_3)^t$ and $(A^4)^t$ and the unique $B_6$ between $(B_2)^t$ and $B_5$ that satisfy the Tits (co)distance conditions below. Then $g=\eta_{w}^{-1}(y)=u(B_6, (A_1)^t, (B^4)^t)$.
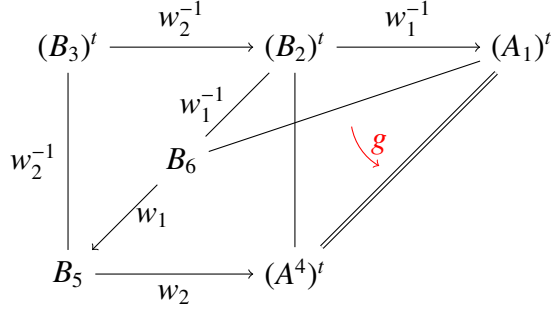
\[
\begin{tikzpicture}[baseline=20,scale=1.5]
    \node (1) at (4,2) [] {$(A_1)^t$};
    \node (2) at (2,2) [] {$(B_2)^t$};
    \node (3) at (0,2) [] {$(B_3)^t$};
    \node (4) at (2,0) [] {$(A^4)^t$};
    \node (5) at (0,0) [] {$B_5$};
    \node (6) at (1,1) [] {$B_6$};
    \draw [double, double distance=1pt] (4) -- (1);
    \draw [<-] (1) -- node [above] {$w_1^{-1}$} (2);
    \draw [<-] (2) -- node [above] {$w_2^{-1}$} (3);
    \draw (4) -- (2);
    \draw [->] (5) -- node [below] {$w_2$} (4);
    \draw (3) -- node [left] {$w_2^{-1}$} (5);
    \draw [->] (6) -- node [right] {$w_1$} (5);
    \draw (2) -- node [left] {$w_1^{-1}$} (6);
    \draw (1) -- (6);
    \draw [red, ->] (2.5,1.3) arc (-170:-120:0.5) node [above] {$g$};
\end{tikzpicture}
\]
On the other hand, we can break the edge $(B_2)^t\mbox{ --- } (A^4)^t$ in half to compute $g_1=\eta_{w_1}^{-1}(y_1)=u(B_7,(A_1)^t,(A^4)^t)$.
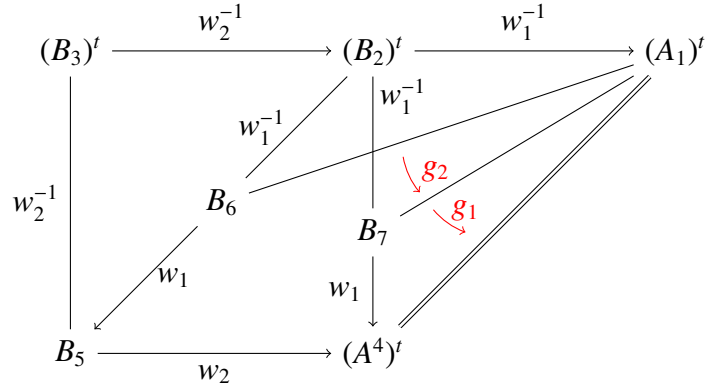
\[
\begin{tikzpicture}[baseline=20,scale=2]
    \node (1) at (4,2) [] {$(A_1)^t$};
    \node (2) at (2,2) [] {$(B_2)^t$};
    \node (3) at (0,2) [] {$(B_3)^t$};
    \node (4) at (2,0) [] {$(A^4)^t$};
    \node (5) at (0,0) [] {$B_5$};
    \node (6) at (1,1) [] {$B_6$};
    \node (7) at (2,0.8) [] {$B_7$};
    \draw [double, double distance=1pt] (4) -- (1);
    \draw [<-] (1) -- node [above] {$w_1^{-1}$} (2);
    \draw [<-] (2) -- node [above] {$w_2^{-1}$} (3);
    \draw [->] (7) -- node [left] {$w_1$} (4);
    \draw (2) -- (7);
    \node at (2.2,1.7) [] {$w_1^{-1}$};
    \draw [->] (5) -- node [below] {$w_2$} (4);
    \draw (3) -- node [left] {$w_2^{-1}$} (5);
    \draw [->] (6) -- node [right] {$w_1$} (5);
    \draw (2) -- node [left] {$w_1^{-1}$} (6);
    \draw (1) -- (6);
    \draw (1) -- (7);
    \draw [red, ->] (2.4,0.95) arc (-140:-110:0.5) node [above] {$g_1$};
    \draw [red, ->] (2.2,1.3) arc (-170:-140:0.5) node [above] {$\quad g_2$};
\end{tikzpicture}
\]
By comparing the pictures we see that $g=g_2g_1$ with $g_2=u(B_6,(A_1)^t,B_7)\in N^-$. Now it remains to show that $g_2\in \overline{w}_1N^-\overline{w}_1^{-1}$. To see this, consider the following subgraph on the left.
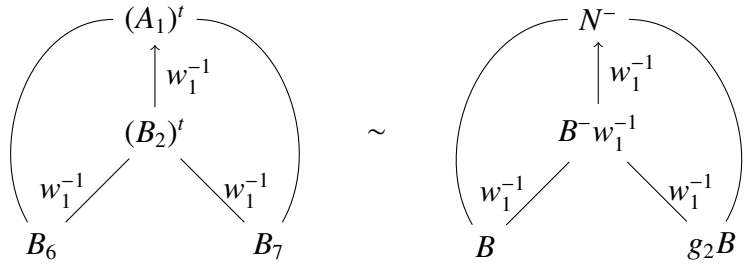
\[
\begin{tikzpicture}[scale=1.5, baseline=40]
    \node (7) at (0,0) [] {$B_6$};
    \node (6) at (2,0) [] {$B_7$};
    \node (2) at (1,1) [] {$(B_2)^t$};
    \node (3) at (1,2) [] {$(A_1)^t$};
    \draw (7) -- node [left] {$w_1^{-1}$} (2);
    \draw (6) -- node [right] {$w_1^{-1}$} (2);
    \draw [->] (2) -- node [right] {$w_1^{-1}$} (3);
    \draw (7) to [out=120,in=-180] (3); 
    \draw (6) to [out=60,in=0] (3);
\end{tikzpicture}\quad \sim \quad 
\begin{tikzpicture}[scale=1.5, baseline=40]
    \node (7) at (0,0) [] {$B$};
    \node (6) at (2,0) [] {$g_2B$};
    \node (2) at (1,1) [] {$B^-w_1^{-1}$};
    \node (3) at (1,2) [] {$N^-$};
    \draw (7) -- node [left] {$w_1^{-1}$} (2);
    \draw (6) -- node [right] {$w_1^{-1}$} (2);
    \draw [->] (2) -- node [right] {$w_1^{-1}$} (3);
    \draw (7) to [out=120,in=-180] (3); 
    \draw (6) to [out=60,in=0] (3);
\end{tikzpicture}
\]
Up to the action of a suitable element $G$, we can move the configuration so that $(A_1)^t$ becomes $N^-$ and $B_6$ becomes $B$. This forces $(B_2)^t$ to become $B^-w_1^{-1}$ and $B_7$ to become $g_2B$, as shown in the picture on the right. Since both $B_6\mbox{ --- } (A_1)^t$ and $B_7 \mbox{ --- } (A_1)^t$ share the same intermediate flag $(B_2)^t=B^-w_1^{-1}$, it follows that $g_2\cdot (B^-w_1^{-1})=B^-w_1^{-1}g_2^{-1}=B^-w_1^{-1}$, which is equivalent to $g_2\in w_1B^-w_1^{-1}$. But since $g_2$ is unipotent, it follows that $g_2\in N^-\cap \overline{w}_1N^-\overline{w}_1^{-1}$.
\end{proof}

\section{Parametrization of twisted MV cycles}\label{Section_para_MV}
In this section, we explicitly construct twisted MV cycles via the twist map on configuration spaces of decorated flags (see \cite{GS} for more details), following the approach of Kamnitzer \cite{Kam}. 

Throughout this section, $G$ is a simply-connected simple algebraic group.

\subsection{Parametrization}\label{ss:parametrization} In this subsection, we 
fix a reduced word $\imath_\bullet=(\imath_1,\imath_2,\dots, \imath_N)$ of $w_0\in W^\tau$ and give a surjective parametrization of $A^{\imath_\bullet}(n_\bullet)$ via the twist map. We  prove that $A^{\imath_\bullet}(n_\bullet)$ is irreducible and consequently, its closure gives an MV cycle.

Recall that, for each $\imath\in I_\tau$, we have  $m_\imath:=\frac{m}{|\imath|}$, $\mc{K}_\imath:=k\dbp{t^{m_\imath}}$, and $\mc{O}_\imath:=k\dbb{t^{m_\imath}}$. Denote by $\mc{O}^\times_\imath$ the multiplicative group of invertible elements inside $\mc{O}_\imath$, consisting precisely of those elements with a nonzero constant term.

\begin{definition}\label{defn: A} Given $\imath\in I_\tau$ and a non-negative integer $n$, we define the scheme $\mc{A}^\imath(n)$ as follows.
\begin{itemize}
    \item If $\imath$ is standard, then $\mc{A}^\imath(n):=t^{nm_\imath}\mc{O}_\imath^\times$.
    \item If $\imath$ is exceptional with $m=2$, then depending on the parity of $n$, we have two subcases:
    \begin{itemize}
        \item if $n=2k$, then $\mc{A}^\imath(2k):=t^{k}\mc{O}^\times\times t^{2k}\mc{O}^\gamma$;
        \item if $n=2k+1$, then $\mc{A}^\imath(2k+1):=t^{(k+1)}\mc{O}\times t^{2k}(\mc{O}^\times)^\gamma$.
    \end{itemize}
    \item If $\imath$ is exceptional with $m=4$, then depending on the parity of $n$, we have two subcases:
    \begin{itemize}
        \item if $n=2k$, then $\mc{A}^\imath(2k):=t^{2k}\mc{O}_\imath\times t^{4k}(\mc{O}_\imath^\times)^\gamma$;
        \item if $n=2k+1$, then $\mc{A}^\imath(2k+1):=t^{2k}\mc{O}_\imath^\times \times t^{4(k+1)}\mc{O}_\imath^\gamma$.
    \end{itemize}
\end{itemize}
For an $N$-tuple $n_\bullet=(n_1,\dots, n_N)$ of non-negative integers, we define 
\[
\mc{A}^{\imath_\bullet}(n_\bullet)=\prod_{k=1}^N \mc{A}^{\imath_k}(n_k).
\]
\end{definition}

\begin{proposition}\label{prop: A2n valuation} For the  $a$ defined in \eqref{eq: a} for an exceptional $\imath$, $\val(a)=nm_\imath$.
\end{proposition}
\begin{proof} Let us first consider the case with $m=2$. Since $a=(b\gamma(b)+tc)/2$, $\val(b\gamma(b))$ is always even, and $\val(tc)$ is always odd, it follows that 
\[
\val(a)=\min\{2\val(b),\val(c)+1\}=\left\{\begin{array}{ll}2\val(b)=2k & \text{if $n=2k$,}\\
\val(c)+1=2k+1 & \text{if $n=2k+1$}.
\end{array}\right.
\]
For the case with $m=4$,  since $a=(-t^2b\gamma(b)+c)/2$, $\val(t^2b\gamma(b))\equiv 2 \mod 4$, and $\val(c)\equiv 0 \mod 4$, we get
\[
\val(a)=\min\{2\val(b)+2,\val(c)\}=\left\{\begin{array}{ll}\val(c)=4k & \text{if $n=2k$},\\
2\val(b)+2=4k+2 & \text{if $n=2k+1$}.\end{array}\right. \qedhere
\]
\end{proof}

\begin{definition}\label{def: twist image} Let  $a_\bullet\in \mc{A}_{\imath_\bullet}(n_\bullet)$. We define
\[
y_{\imath_\bullet}=y_{\imath_1}y_{\imath_2}\cdots y_{\imath_N},
\]
where $y_{\imath_k}:=y_{\imath_k}(a_k)$ if $\imath_k$ is standard, and $y_{\imath_k}:=y_{\imath_k}(b_k,c_k)$ if $\imath_k$ is exceptional. By the fact that $y_i(f)\in N^-(s_i)(\mc{K})$ and $y_\imath(0,c)\in N^-(s_\imath)(\mc{K})$ for exceptional $\imath$ (Lemma \ref{lem: change of word in special case} (3)), we know that $y_{\imath_\bullet}\in N^-(w_0)(\mc{K})$. We define $g_{\imath_\bullet}:=\eta_{w_0}^{-1}(y_{\imath_\bullet})$. Since $y_{\imath_k}$ is $\sigma$-invariant for each $k$ (Proposition \ref{prop: sigma invariant unipotent}), by Lemma \ref{lem: sigma invariant}, we know that $g_{\imath_\bullet}$ is also $\sigma$-invariant and hence $[g_{\imath_\bullet}]$ is an element of $\mGr$. We denote the resulting map by 
\[
\Phi:\mc{A}^{\imath_\bullet}(n_\bullet)\rightarrow \mGr.
\]
Lemma \ref{lem: limit} implies that, despite there being two ways of defining $y_\imath$ for an exceptional $\imath$ depending on whether $b$ vanishes or not, the resulting map $\Phi$ is still algebraic.
\end{definition}

\begin{lemma}\label{lem: ratio}  Let $\imath_\bullet=(\imath_1, \dots, \imath_N)$ be a reduced word of $w_0\in W^\tau$, and set $w^{\imath_\bullet}_k:=s_{\imath_1}\cdots s_{\imath_k}$.  Let $y_{\imath_\bullet}$ and $g_{\imath_\bullet}$ be as in Definition \ref{def: twist image}. For any $\Lambda_i$ and $v\in V(\Lambda_i)$, we  have, for $1\leq k\leq N$,
\[
\val(\la \xi_{w^{\imath_\bullet}_k\Lambda_i}, g_{\imath_\bullet}\cdot v\ra)\geq \val(\Delta_{w^{\imath_\bullet}_k\Lambda_i}(g_{\imath_\bullet})).
\]
\end{lemma}
\begin{proof} Let $w_1:=w^{\imath_\bullet}_k$ and $w_2:=w_1^{-1}w$. Let $y_1:=y_{\imath_1}\cdots y_{\imath_k}$ and let $y_2:=y_1^{-1}y$. Then by Proposition \ref{prop: y=y2y1}, we can write $g=g_2g_1$ with $g_1=\eta_{w_1}^{-1}(y_1)$ and $g_2\in N^-\cap \overline{w}_1N^-\overline{w}_1^{-1}$. On the one hand, since $y_1=\eta_{w_1}(g_1)$, it follows that $y_1b=g_1^t\overline{w}_1$ for some $b\in B$, which is equivalent to $g_1=\overline{w}_1lty_1^t$ for some $l\in N^-$ and $t\in T$. On the other hand, $g_2\in N^-\cap \overline{w}_1N^-\overline{w}_1^{-1}$ implies that $g_2^{-1}\cdot \xi_{w_1\Lambda_i}=\xi_{w_1\Lambda_i}$. By combining these two facts, we get
\[
g^{-1}\cdot \xi_{w_1\Lambda_i}=g_1^{-1}g_2^{-1}\cdot \xi_{w_1\Lambda_i}=g_1^{-1}\cdot \xi_{w_1\Lambda_i}=(y_1^t)^{-1}t^{-1}l^{-1}\overline{w}_1^{-1}\xi_{w_1\Lambda_i}=\Lambda_i(t)(y_1^t)^{-1}\cdot \xi_{\Lambda_i}.
\]
Thus, we have
\[
\Delta_{w_1\Lambda_i}(g)=\la \xi_{w_1\Lambda_i},g\cdot v_{\Lambda_i}\ra=\la g^{-1}\cdot \xi_{w_1\Lambda_i},v_{\Lambda_i}\ra=\Lambda_i(t)\la \xi_{\Lambda_i}, y_1^t\cdot v_{\Lambda_i}\ra=\Lambda_i(t),
\]
and
\[
\la\xi_{w_1\Lambda_i}, g\cdot v\ra=\la g^{-1}\cdot \xi_{w_1\Lambda_i}, v\ra=\Lambda_i(t)\la \xi_{\Lambda_i}, y_1^t\cdot v\ra=\la \xi_{\Lambda_i}, y_1^t\cdot v\ra\Delta_{w_1\Lambda_i}(g).
\]
It remains to argue that $\val(\la \xi_{\Lambda_i}, y_1^t\cdot v\ra)\geq 0$. Let $s$ be an index with $1\leq s\leq k$. When $\imath_s$ is standard, $\la \xi_{\Lambda_i}, y_{\imath_s}^t\cdot v\ra$ is a polynomial (over $\mc{O}$) in $a_s$, which has a non-negative valuation. When $\imath_s$ is exceptional,  $\la \xi_{\Lambda_i}, y_{\imath_s}^t\cdot v\ra$ is a polynomial (over $\mc{O}$) in $a_s$, $b_s$, and $\gamma(b_s)$ if $b_s\neq 0$, and is a polynomial (over $\mc{O}$) in $a_s$ if $b_s=0$ (c.f. Proof of Lemma \ref{lem: limit}); in either case, $\la \xi_{\Lambda_i}, y_{\imath_s}^t\cdot v\ra$ has a non-negative valuation as well. Thus, we can conclude that $\val(\la \xi_{\Lambda_i}, y_1^t\cdot v\ra)\geq 0$ and hence the lemma is true.
\end{proof}

\begin{proposition} The image of the map $\Phi: \mc{A}^{\imath_\bullet}(n_\bullet)\rightarrow \mGr$ defined in Definition \ref{def: twist image} lies in $A^{\imath_\bullet}(n_\bullet)$. 
\end{proposition}
\begin{proof} For convenience, we write $x\doteq y$ if $x, y \in \mc{K}^\times$ and $\val(x)=\val(y)$.
Note that multiplying by a nonzero element in $k$, or changing a factor under the action of a power of $\gamma$, does not change the valuation. 


Let $g:= g_{\imath_\bullet}$. For each $\jmath\in I_\tau$, let $\Lambda_j$ be a lift of $\Lambda_\jmath$; since $g\in N^-(w_0)(\mc{K})^\sigma\subset N^-(\mc{K})^\sigma$, we have
\begin{align*}
D_{\Lambda_\jmath}([g])=&\min_{v\in \overline{V}_{\Lambda_\jmath}}\val(\la \xi_{\Lambda_j},g\cdot v\ra)=\min_{v\in \overline{V}_{\Lambda_\jmath}} \val(\langle g^{-1}\cdot\xi_{\Lambda_j}, v\rangle )\\
=&\min_{v\in \overline{V}_{\Lambda_\jmath}}\val(\langle  \xi_{\Lambda_j},v\rangle)=\val(\langle  \xi_{\Lambda_j},v_{\Lambda_j}\rangle)=0=\val(\Delta_{\Lambda_j}(g)).
\end{align*}

For the rest of the chamber weights of the form $w^{\imath_\bullet}_{k}\Lambda_{\imath_k}$ with standard $\imath_k$, Lemma \ref{lem: minor under gamma action} and Proposition \ref{prop: minor recursive formula} imply that, for the chosen representative $i_k$ of $\imath_k$, 
\[
a_k=\lambda^{\imath_\bullet}_k\frac{\prod_{j\neq i_k}\Delta_{w^{\imath_\bullet}_{k}\Lambda_j}(g)^{-a_{ji_k}}}{\Delta_{w^{\imath_\bullet}_{k-1}\Lambda_{i_k}}(g)\Delta_{w^{\imath_\bullet}_{k}\Lambda_{i_k}}(g)}\doteq \frac{\prod_{\jmath\neq \imath_k}\Delta_{w^{\imath_\bullet}_{k-1}\Lambda_j}(g)^{-a_{\jmath\imath_k}}}{\Delta_{w^{\imath_\bullet}_{k-1}\Lambda_{i_k}}(g)\Delta_{w^{\imath_\bullet}_{k}\Lambda_{i_k}}(g)},
\]
where the $j$'s in the last expression are the chosen representatives of $\jmath$.

For the chamber weights of the form $w^{\imath_\bullet}_{ k}\Lambda_{\imath_k}$ with exceptional $\imath_k=\{l,l+1\}$ in the case of $A_{2l}^{(2)}$, let us consider the following.
\begin{figure}[H]
    \centering
    \begin{tikzpicture}
        \foreach \i in {1,2,3}
        {
        \draw [gray] (0,\i) -- (6,\i);
        }
        \draw [gray](1.5,1) -- node [black] {$s_l$} (1.5,2);
        \draw [gray](3,2) -- node [black] {$s_{l+1}$} (3,3);
        \draw [gray](4.5,1) -- node [black] {$s_l$} (4.5,2);
        \node at (0,1.5) [] {$\Delta_{w^{\imath_\bullet}_{k-1}\Lambda_l}$};
        \node at (3,1.5) [] {$\Delta_{w^{\imath_\bullet}_{k-1}s_i\Lambda_l}$};
        \node at (6,1.5) [] {$\Delta_{w^{\imath_\bullet}_{k}\Lambda_l}$};
        \node at (1,2.5) [] {$\Delta_{w^{\imath_\bullet}_{k-1}\Lambda_{l+1}}$};
        \node at (5,2.5) [] {$\Delta_{w^{\imath_\bullet}_{k}\Lambda_{l+1}}$};
        \node at (3,0.5) [] {$\Delta_{w^{\imath_\bullet}_{k}\Lambda_{l-1}}$};
        \node at (3,3.5) [] {$\Delta_{w^{\imath_\bullet}_{k}\Lambda_{l+2}}$};
    \end{tikzpicture}
    \caption{Local picture of generalized minors near an exceptional $\imath_k$ in $A_{2l}^{(2)}$.}
\end{figure}
If $b_k\neq 0$, by Proposition \ref{prop: minor recursive formula} (1) and Eq. \eqref{eq: root subgroup element for A2n}, we see that there are non-zero constants $\lambda,\lambda'\in k^\times$ such that,
\[
\lambda\frac{\Delta_{w^{\imath_\bullet}_{k}\Lambda_{l+2}}(g)\Delta_{w^{\imath_\bullet}_{k-1}s_l\Lambda_l}(g)}{\Delta_{w^{\imath_\bullet}_{k-1}\Lambda_{l+1}}(g)\Delta_{w^{\imath_\bullet}_{k}\Lambda_{l+1}}(g)}=\left\{\begin{array}{ll}
    b_k & \text{if $m=2$},  \\
    tb_k & \text{if $m=4$},
\end{array}\right.
\]
and
\[
\lambda'\frac{\Delta_{w^{\imath_\bullet}_{k}\Lambda_{l-1}}(g)\Delta_{w^{\imath_\bullet}_{k}\Lambda_{l+1}}(g)}{\Delta_{w^{\imath_\bullet}_{k-1}s_l\Lambda_l}(g)\Delta_{w^{\imath_\bullet}_{k}\Lambda_l}(g)}=\left\{\begin{array}{ll} a_k/b_k & \text{if $m=2$},\\
a_k/(tb_k) & \text{if $m=4$}. \end{array}\right.
\]
Thus, regardless of whether $m=2$ or $m=4$, we have
\[
a_k=\lambda\lambda'\frac{\Delta_{w^{\imath_\bullet}_{k}\Lambda_{l-1}}(g)\Delta_{w^{\imath_\bullet}_{k}\Lambda_{l+2}}(g)}{\Delta_{w^{\imath_\bullet}_{k-1}\Lambda_{l+1}}(g)\Delta_{w^{\imath_\bullet}_{k}\Lambda_{l}}(g)}\doteq\frac{\prod_{\jmath\neq \imath_k}\Delta_{w^{\imath_\bullet}_{k-1}\Lambda_j}(g)^{-a_{\jmath\imath_k}}}{\Delta_{w^{\imath_\bullet}_{k-1}\Lambda_l}(g)\Delta_{w^{\imath_\bullet}_{k}\Lambda_{l}}(g)}.
\]

If $b_k=0$, by Proposition \ref{prop: minor recursive formula} (2) and Eq. \eqref{eq: root subgroup element for A2n}, we can also conclude that for some non-zero constant $\lambda\in k^\times$,
\[
a_k=\lambda\frac{\Delta_{w^{\imath_\bullet}_{k}\Lambda_{l-1}}(g)\Delta_{w^{\imath_\bullet}_{k}\Lambda_{l+2}}(g)}{\Delta_{w^{\imath_\bullet}_{k-1}\Lambda_l}(g)\Delta_{w^{\imath_\bullet}_{k}\Lambda_{l+1}}(g)}\doteq\frac{\prod_{\jmath\neq \imath_k}\Delta_{w^{\imath_\bullet}_{k-1}\Lambda_j}(g)^{-a_{\jmath\imath_k}}}{\Delta_{w^{\imath_\bullet}_{k-1}\Lambda_l}(g)\Delta_{w^{\imath_\bullet}_{k}\Lambda_{l}}(g)}.
\]

These relations between $a_k$'s and the generalized minors of $g$ allow us to compute $\val(\Delta_{w^{\imath_\bullet}_{k}\Lambda_{i_k}}(g))$ recursively as:
\begin{align*}
\val\left(\Delta_{w^{\imath_\bullet}_{k}\Lambda_{i_k}}(g)\right)=&-\left(\sum_{\jmath\neq \imath_k}a_{\jmath\imath_k}\val\left(\Delta_{w^{\imath_\bullet}_{k-1}\Lambda_j}(g)\right)\right)-\val\left(\Delta_{w^{\imath_\bullet}_{k-1}\Lambda_{i_k}}(g)\right)-\val\left(a_k\right)\\
=&-\left(\sum_{\jmath\neq \imath_k}a_{\jmath\imath_k}\val\left(\Delta_{w^{\imath_\bullet}_{k-1}\Lambda_j}(g)\right)\right)-\val\left(\Delta_{w^{\imath_\bullet}_{k-1}\Lambda_{i_k}}(g)\right)-n_km_{\imath_k}.
\end{align*}
By comparison with Proposition \ref{prop: A(n)}, we see that it suffices to prove $D_{w^{\imath_\bullet}_{k} \Lambda_{\imath_k}}([g])=\val\left(\Delta_{w^{\imath_\bullet}_{k}\Lambda_{i_k}}(g)\right)$ for all $1\leq k\leq N$. By definition of $D_\gamma$ (Equation \eqref{D_def}), we see that on the one hand,
\[
D_{w^{\imath_\bullet}_{k}\Lambda_{\imath_k}}([g])= \min_{v\in \overline{V}(\Lambda_{\imath_k})}\val\left(\left\la \xi_{w^{\imath_\bullet}_{k}\Lambda_{i_k}},g\cdot v\right\ra\right)\leq \val\left(\Delta_{w^{\imath_\bullet}_{k} \Lambda_{i_k}}(g)\right).
\]
On the other hand, Lemma \ref{lem: ratio} gives the reverse inequality. Thus, we can conclude that 
\[
D_{w^{\imath_\bullet}_{k} \Lambda_{\imath_k}}([g])=\val(\Delta_{w^{\imath_\bullet}_{ k}\Lambda_{i_k}}(g))
\]
for all $1\leq k\leq N$, which shows that the image $[g]$ lands inside $A^{\imath_\bullet}(n_\bullet)$.
\end{proof}

\begin{theorem}\label{prop: surjective} The map $\Phi:\mc{A}^{\imath_\bullet}(n_\bullet)\rightarrow A^{\imath_\bullet}(n_\bullet)$ in Definition \ref{def: twist image} is surjective.
\end{theorem}
\begin{proof} 
Suppose $L=[g]$ for some $g\in N^-(\mathcal{K})^\sigma$. By Lemma \ref{lem: representative} we may assume without loss of generality that $D_{w^{\imath_\bullet}_k\Lambda_{\imath_k}}([g])=\val(\Delta_{w^{\imath_\bullet}_{k}\Lambda_{i_k}}(g))$ for all $1\leq k\leq N$. 
Note that except for the generalized minor $\Delta_{w^{\imath_\bullet}_{k-1}s_l\Lambda_l}(g)$ when $\imath_k=\{l,l+1\}$ is exceptional, all other generalized minor of $g$ appear in Proposition \ref{prop: minor recursive formula} are $\doteq\Delta_{w^{\imath_\bullet}_s\Lambda_{i_s}}(g)$ for some $s$, and hence are non-zero because they have finite valuations. In particular, $\Delta_{w_0\Lambda_i}(g)\neq 0$ for all $\Lambda_i$ and therefore $g\in N^-(w_0)(\mc{K})$ (Lemma \ref{lem: in N^-(w_0)}). 

Let $y:=\eta_{w_0}(g)$. We claim that $y$ admits a factorization $y=y_{\imath_1}y_{\imath_2}\cdots y_{\imath_N}$, where each $y_{\imath_k}$ carries a parameter $a_k$ if $\imath_k$ is standard, and two parameters $b_k$ and $c_k$ if $\imath_k$ is exceptional. To see this, note that by Lemma \ref{lem: relative position} (1), there exists intermediate $B_1,B_2,\dots, B_{N-1}$ that fit compatibly in between $B\xrightarrow{w_0}yB$ as 
\[
B=B_0\xrightarrow{s_{\imath_1}}B_1\xrightarrow{s_{\imath_2}}B_2\xrightarrow{s_{\imath_3}} \cdots \xrightarrow{s_{\imath_{N-1}}} B_{N-1}\xrightarrow{s_{\imath_N}}yB=B_N,
\]
and even more decorated flags that sit within each $B_{k-1}\xrightarrow{s_{\imath_k}}B_k$ depending on the factorization of $s_{\imath_k}$ (Eq. \eqref{eq: s_i}). Moreover, up to a diagonal $G$-action, we can move (and lift) these flags into a compatible configuration of decorated flags as in the proof of Proposition \ref{prop: minor recursive formula}. Now the non-vanishing of $\Delta_{w^{\imath_\bullet}_s\Lambda_{i_s}}(g)$ for all $1\leq s\leq N$ implies that $N^-\mbox{ --- }B_k$ for every $k$, and therefore the angle unipotents $u(B_{k-1},N^-,B_k)$ are well-defined. Moreover, if $\imath_k$ is standard, the angle unipotent $u(B_{k-1},N^-,B_k)$ must be of the form $y_{\imath_k}(a_k)$ for some $a_k$ (Lemmas \ref{lem: minor under gamma action} and \ref{lem: Lusztig variables in terms of A variables}); if $\imath_k$ is exceptional, the angle unipotent $u(B_{k-1},N^-B_k)$ must be of the form $y_{\imath_k}(b_k,c_k)$ for some $b_k$ and $c_k$ (Lemmas \ref{lem: minor under gamma action} and \ref{lem: Lusztig variables in terms of A variables} when $\Delta_{w^{\imath_\bullet}_{k-1}s_l\Lambda_l}(g)\neq 0$ and Proposition \ref{prop: vanishing minor implies special angle unipotent} and Lemma \ref{lem: special case} when $\Delta_{w^{\imath_\bullet}_{k-1}s_l\Lambda_l}(g)= 0$). This proves the claim of the factorization $y=y_{\imath_1}y_{\imath_2}\cdots y_{\imath_N}$.

Next, we need to argue that the parameters $a_k$, $b_k$, and $c_k$ are from the correct $\mc{A}^{\imath_k}(n_k)$. Note that if $\imath_k$ is standard, we know that $a_k\in \mc{K}_{\imath_k}$ by Corollary \ref{prop: quotient in K}, and Proposition \ref{prop: minor recursive formula} dictates that $\val(a_k)=n_km_{\imath_k}$, and therefore $a_k\in t^{n_km_{\imath_k}}\mc{O}_{\imath_k}^\times$.


It remains to consider the case of $A_{2l}^{(2)}$ with exceptional $\imath_k=\{l,l+1\}$. 
Let us first consider the case with $m=2$, for which $m_{\imath_k}=1$ and hence $\mc{K}_{\imath_k}=\mc{K}$. 
On the one hand, if $\Delta_{w^{\imath_\bullet}_{k-1}s_l\Lambda_l}(g)\neq 0$, then both 
\[
b_k=\lambda\dfrac{\Delta_{w^{\imath_\bullet}_{k}\Lambda_{l+2}}(g)\Delta_{w^{\imath_\bullet}_{k-1}s_l\Lambda_l}(g)}{\Delta_{w^{\imath_\bullet}_{k-1}\Lambda_{l+1}}(g)\Delta_{w^{\imath_\bullet}_{k}\Lambda_{l+1}}(g)} \quad \text{and} \quad
\frac{a_k}{b_k}=\lambda'\frac{\Delta_{w^{\imath_\bullet}_{k}\Lambda_{l+1}}(g)\Delta_{w^{\imath_\bullet}_{k}\Lambda_{l-1}}(g)}{\Delta_{w^{\imath_\bullet}_{k-1}s_l\Lambda_l}(g)\Delta_{w^{\imath_\bullet}_k\Lambda_l}(g)} 
\]
are non-zero, and hence we can compute $a_k$ by taking their product and then compute $c_k$ from Eq. \eqref{eq: a}. In particular, since $c_k=(2a_k-b_k\gamma(b_k))/t=(a_k-\gamma(a_k))/t$, $c_k$ must be an element of $\mc{K}^\gamma$. 

On the other hand, if $\Delta_{w^{\imath_\bullet}_{k-1}s_l\Lambda_l}(g)=0$, then $\gamma\left(\Delta_{w^{\imath_\bullet}_{k-1}s_l\Lambda_l}(g)\right)=0$. It then follows from Propositions \ref{prop: vanishing minor implies special angle unipotent} and \ref{prop: minor recursive formula} (2) that the angle unipotent element $y_{\imath_k}=y_l(1)y_{l+1}(-a_k)y_l(-1)y_{l+1}(a_k)$ for
\[
a_k=\lambda\frac{\Delta_{w^{\imath_\bullet}_{k}\Lambda_{l-1}}(g)\Delta_{w^{\imath_\bullet}_{k}\Lambda_{l+2}}(g)}{\Delta_{w^{\imath_\bullet}_{k-1}\Lambda_l}(g)\Delta_{w^{\imath_\bullet}_{k}\Lambda_{l+1}}(g)}.
\]
Note that the numerator is $\gamma$-invariant since $l+2=\tau(l-1)$. For the denominator, by Eq. \eqref{eq: minor identity}, we can deduce that 
\[
\Delta_{w^{\imath_\bullet}_{k-1}\Lambda_l}(g)\Delta_{w^{\imath_\bullet}_{k}\Lambda_{l+1}}(g)=-\Delta_{w^{\imath_\bullet}_{k-1}\Lambda_{l+1}}(g)\Delta_{w^{\imath_\bullet}_{k}\Lambda_l}(g)
\]
But then by Lemmas \ref{lem: minor under gamma action} and \ref{lem: pairing mod 4} (1), we know that
\begin{align*}
\Delta_{w^{\imath_\bullet}_{k-1}\Lambda_{l+1}}(g)\Delta_{w^{\imath_\bullet}_{k}\Lambda_l}(g)=&(-1)^{\la h, w^{\imath_\bullet}_{k-1}\Lambda_l-\Lambda_l+w^{\imath_\bullet}_{k}\Lambda_{l+1}-\Lambda_{l+1}\ra}\gamma\left(\Delta_{w^{\imath_\bullet}_{k-1}\Lambda_l}(g)\Delta_{w^{\imath_\bullet}_{k}\Lambda_{l+1}}(g)\right)\\
=&\gamma\left(\Delta_{w^{\imath_\bullet}_{k-1}\Lambda_l}(g)\Delta_{w^{\imath_\bullet}_{k}\Lambda_{l+1}}(g)\right).
\end{align*}
Thus, the denominator of $a_k$ is $\gamma$-anti-invariant and therefore $\gamma(a_k)=-a_k$. Comparing this with Eq. \eqref{eq: a}, we see that $b_k=0$ and $c_k=2a_k/t\in \mc{K}^\gamma$.

In either subcase of the $A_{2l}^{(2)}$ with $m=2$ case, we have $\val(a_k)=n_k$ by Proposition \ref{prop: minor recursive formula}. Since both $2\val(b_k)$ and $\val(c_k)$ are even numbers, $n_k=\val(a_k)=\min\{2\val(b_k),\val(c_k)+1\}$, which equals $2\val(b_k)$ if $n_k$ is even, and $\val(c_k)+1$ if $n_k$ is odd, yielding the correct $\mc{A}^{\imath_k}(n_k)$. 

For the case of $A_{2l}^{(2)}$ with $m=4$, we have $m_{\imath_k}=2$ and hence $\mc{K}_{\imath_k}=k\dbp{t^2}$. If $\Delta_{w^{\imath_\bullet}_{k-1}s_l\Lambda_l}(g)\neq 0$, then we can compute $a_k, b_k, c_k$ similar to the $m=2$ case above; in particular, by Corollaries \ref{cor: A2l with m=4} and \ref{cor: A2n m=4 middle minor}, we know that $a_k,b_k\in \mc{K}_{\imath_k}$, and hence $c_k=2a_k+t^2b_k\gamma(b_k)=a_k+\gamma(a_k)\in \mc{K}_{\imath_k}^\gamma$.

If $\Delta_{w^{\imath_\bullet}_{k-1}s_l\Lambda_l}(g)=0$, then we again set $a_k=\lambda\dfrac{\Delta_{w^{\imath_\bullet}_{k}\Lambda_{l-1}}(g)\Delta_{w^{\imath_\bullet}_{k}\Lambda_{l+2}}(g)}{\Delta_{w^{\imath_\bullet}_{k-1}\Lambda_l}(g)\Delta_{w^{\imath_\bullet}_{k}\Lambda_{l+1}}(g)}$ as in Proposition \ref{prop: minor recursive formula} (2). Note that the numerator is again $\gamma$-invariant, and by Eq. \eqref{eq: minor identity}, we again have $\Delta_{w^{\imath_\bullet}_{k-1}\Lambda_l}(g)\Delta_{w^{\imath_\bullet}_{k}\Lambda_{l+1}}(g)=-\Delta_{w^{\imath_\bullet}_{k-1}\Lambda_{l+1}}(g)\Delta_{w^{\imath_\bullet}_{k}\Lambda_l}(g)$.
But this time, Lemmas \ref{lem: minor under gamma action} and \ref{lem: pairing mod 4} (2) implies that
\begin{align*}
\Delta_{w^{\imath_\bullet}_{k-1}\Lambda_{l+1}}(g)\Delta_{w^{\imath_\bullet}_{k}\Lambda_l}(g)=&\epsilon^{\la h, w^{\imath_\bullet}_{k-1}\Lambda_l-\Lambda_l+w^{\imath_\bullet}_{k}\Lambda_{l+1}-\Lambda_{l+1}\ra}\gamma\left(\Delta_{w^{\imath_\bullet}_{k-1}\Lambda_l}(g)\Delta_{w^{\imath_\bullet}_{k}\Lambda_{l+1}}(g)\right)\\
=&-\gamma\left(\Delta_{w^{\imath_\bullet}_{k-1}\Lambda_l}(g)\Delta_{w^{\imath_\bullet}_{k}\Lambda_{l+1}}(g)\right).
\end{align*}
Thus, we can conclude that the denominator of $a_k$, and hence $a_k$ itself, is $\gamma$-invariant. Comparing this with Eq. \eqref{eq: a}, we see that $b_k=0$ again and $c_k=2a_k\in \mc{K}_{\imath_k}^\gamma$.

In either subcase of the $A_{2l}^{(2)}$ with $m=4$ case, we have $\val(a_k)=2n_k$ by Proposition \ref{prop: minor recursive formula}. Since both $2\val(b_k)$ and $\val(c_k)$ are multiples of $4$, $2n_k=\val(a_k)=\min\{2\val(b_k)+2,\val(c_k)\}$, which equals $\val(c_k)$ if $n_k$ is even, and $2\val(b_k)+2$ if $n_k$ is odd, yielding the correct $\mc{A}^{\imath_k}(n_k)$ again. 
\end{proof}

\begin{corollary} 
\label{cor:irred}
$A^{\imath_\bullet}(n_\bullet)$ is irreducible. 
\end{corollary}
\begin{proof} Since the image of an irreducible space is irreducible, the corollary follows from the fact that $\mc{A}^{\imath_\bullet}(n_\bullet)$ is irreducible and the map $\Phi:\mc{A}^{\imath_\bullet}(n_\bullet)\rightarrow A^{\imath_\bullet}(n_\bullet)$ is surjective (Theorem \ref{prop: surjective}).
\end{proof}

Given an element $\mu\in   \sum_k  \mathbb{Z}_{\geq 0} \gamma_k\subset X_*(H)_\tau$ and a fixed reduced expression $\imath_\bullet$ of $w_0\in W^\tau$, by Lemma \ref{lem:semiinfinite-int}, using  similar argument as in \cite[Proposition 4.1]{Kam},  we have the following decomposition
\begin{equation}
    \label{eq:ST-decomp}
 S_\mu\cap  T_0 =  \bigsqcup_{ n_\bullet }  A^{\imath_\bullet}(n_\bullet) ,  \end{equation} 
where $n_\bullet$ satisfies $\mu=\sum_s  n_s\beta^{\imath_\bullet}_s$ and $n_s\geq 0$. 

\begin{theorem}
\label{thm_MV_constr}
In the decomposition (\ref{eq:ST-decomp}), for each $A^{\imath_\bullet}(n_\bullet)$ 
the closure $\overline{A^{\imath_\bullet}(n_\bullet)}$ is an irreducible component of  $\overline{S_\mu\cap T_0}$ and each component of $\overline{S_\mu\cap T_0}$ arises exactly once this way. In particular, the variety $\overline{A^{\imath_\bullet}(n_\bullet)}$ has dimension $\la \mu, \rho \ra$.  
\end{theorem}
\begin{proof}
Let $k(\mu)$ be the Kostant partition function for the Satake group $(\check{G})^\tau$.

By Proposition \ref{T_S_prop}, 
we can choose a sufficient dominant $\eta\in X_*(H)^+_\tau$, such that
\[ \overline{ S_{\mu+\eta} \cap T_{ \eta}  }= \overline{ \mGr_{\mu+\eta} \cap T_{\eta } }.\]
By Theorem \ref{thm_ramified_Satake}, the number of irreducible components of $\overline{ \mGr_{\mu+\eta} \cap T_{\eta } }$ equals $\dim V_{\eta+\mu}(\eta)$. It is well-known that when $\eta$ is sufficiently dominant, $\dim V_{\eta+\mu}(\eta)=k(\mu)$. 
Note that, $\overline{S_\mu\cap T_0}$ is a translation of $ \overline{ S_{\mu+\eta} \cap T_{ \eta}  }$. 
Thus, $\overline{S_\mu\cap T_0}$ has $k(\mu)$ many irreducible components. 

By \eqref{eq:ST-decomp}, $S_\mu\cap  T_0$ is a disjoint union of $k(\mu)$ many strata $A^{\imath_\bullet}(n_\bullet)$. 
By Corollary \ref{cor:irred}, $A^{\imath_\bullet}(n_\bullet)$ is irreducible.
  By \cite[Lemma 4.3]{Kam},  all $\overline{A^{\imath_\bullet}(n_\bullet)}$ are distinct irreducible components of   $\overline{S_\mu\cap T_0}$ of dimension  $\la \mu, \rho \ra$.
\end{proof}

\subsection{Braid moves and Lusztig transformations} Recall that if $\imath_\bullet$ and $\imath'_\bullet$ are two reduced words of $w\in W^\tau$ spelled with $s_{\imath}$'s, then there exists a sequence of braid moves involving $s_\imath$'s that turns one reduced word to the other. Thus, in order to relate $A^{\imath_\bullet}(n_\bullet)$ and $A^{\imath'_\bullet}(n'_\bullet)$, we need to investigate how the Lusztig data $n_\bullet$ and $n'_\bullet$ for the same MV cycle are related under a braid move.

On the other hand, following \cite[Section 2.3]{Hong-Shen-2015}, we can define
\[
\check{\gamma}_\imath=\left\{\begin{array}{ll}
    \sum_{i\in \imath}\alpha_i & \text{if $\imath$ is standard}, \\
    2(\alpha_i+\alpha_j) & \text{if $\imath$ is exceptional},
\end{array}\right.
\]
and $\gamma_\imath=\overline{\check{\alpha}_i}$. By setting 
\begin{equation}\label{eq: cartan}
c_{\imath\jmath}:=\la \check{\gamma}_\imath,\gamma_\jmath\ra,
\end{equation}
we see that $c_{\imath\jmath}=a_{\imath\jmath}$ except when one of $\imath$ and $\jmath$ is exceptional, for which $c_{\imath\jmath}=a_{\jmath\imath}$.

We divide our discussion on the relation between $n_\bullet$ and $n'_\bullet$ based on the values of $c_{\imath\jmath}$ and $c_{\jmath\imath}$.

\medskip

\textbf{(o).} $c_{\imath\jmath}=c_{\jmath\imath}=0$, which corresponds to the braid move $s_\imath s_\jmath=s_\jmath s_\imath$. In this case, it is necessary that $a_{ij}=0$ for all $i\in \imath$ and $j\in \jmath$. As a result, we have
\begin{proposition}\label{prop: A1xA1} If $y_\imath(a'_1)y_\jmath(a'_2)=y_\jmath(a_1)y_\imath(a_2)$, then
    \[
    a'_1=a_2 \quad \text{and} \quad a'_2=a_1,
    \]
    and therefore
    \[
    n'_1=n_2 \quad \text{and} \quad n'_2=n_1.
    \]
    
\end{proposition}
    
    \medskip
    
    \textbf{(i).} $c_{\imath\jmath}=c_{\jmath\imath}=-1$, which corresponds to the braid move $s_\imath s_\jmath s_\imath = s_\jmath s_\imath s_\jmath$. In this case, it is necessary that $|\imath|=|\jmath|=1$ or $2$, and $m_\imath=m_\jmath=1$ or $2$. 
    \begin{proposition}\label{prop: A2} If $y_\imath(a'_1)y_\jmath(a'_2)y_\imath(a'_3)=y_\jmath(a_1)y_\imath(a_2)y_\jmath(a_3)$, then there is a non-empty open subset in $\mathcal{A}^{\imath_\bullet}(n_\bullet)$, over which
    \[
    n'_1=n_2+n_3-\min\{n_1,n_3\}, \quad n'_2=\min\{n_1,n_3\}, \quad n'_3=n_1+n_2-\min\{n_1,n_3\}.
    \]
    \end{proposition}
    \begin{proof} If $|\imath|=|\jmath|=1$, then by Eq. \eqref{eq: Lusztig transform}
    \[
    y_j(a_1)y_i(a_2)y_j(a_3)=y_i\left(\frac{a_2a_3}{a_1+a_3}\right)y_j(a_1+a_3)y_i\left(\frac{a_1a_2}{a_1+a_3}\right),
    \]
    we see that, over the open subset where the leading coefficients of $a_1$ and $a_3$ do not cancel, we have the desired identity.
    
    Suppose $|\imath|=|\jmath|=2$. Let us fix $i\in \imath$ and $j\in \jmath$ such that $a_{ij}=-1$. Let $i'=\tau(i)$ and $j'=\tau(j)$. Then it is necessary that $a_{ii'}=a_{jj'}=a_{ij'}=a_{i'j}=0$. This implies that
    \begin{align*}
    y_\jmath(a_1)y_\imath(a_2)y_\jmath(a_3)=& y_j(a_1)y_{j'}(\epsilon^{-\alpha_j(h)}\gamma(a_1))y_i(a_2)y_{i'}(\epsilon^{-\alpha_i(h)}\gamma(a_2))y_j(a_3)y_{j'}(\epsilon^{-\alpha_j(h)}\gamma(a_3))\\
    =& y_j(a_1)y_i(a_2)y_j(a_3)y_{j'}(\epsilon^{-\alpha_j(h)}\gamma(a_1))y_{i'}(\epsilon^{-\alpha_i(h)}\gamma(a_2))y_{j'}(\epsilon^{-\alpha_j(h)}\gamma(a_3)).
    \end{align*}
    A similar identity can also be derived for $y_\imath(a'_1)y_\jmath(a'_2)y_\imath(a'_3)$. By applying Eq. \eqref{eq: Lusztig transform}, we can deduce the same desired identity over the open subset where the leading coefficients of $a_1$ and $a_3$ do not cancel each other.
    \end{proof}

    \medskip

     Before considering the next few cases, we first establish a lemma on the valuation of polynomials with twisted variables.

    \begin{definition}\label{def: polynomial} 
    Let $\gamma$ be the field automorphism  of $\mc{K}=\rk\dbp{t}$ that fixes $\rk$ pointwise and sends $t\mapsto \epsilon^{-1}t$, where $\epsilon$ is a primitive $m$th root of unity. Let $F$ be a polynomial over $\rk$ in the variables $\gamma^j(x_i)$ with  $1\leq i \leq n$ and $0\leq j < m$. 
    When computing the degree vectors of monomials,  we identify the variables $x_i, \gamma(x_i), \ldots, \gamma^{m-1}(x_i)$ for each $i$. The \emph{Newton polytope} of $F$ is  the convex hull of the degree vectors of the monomials appearing in $F$. A degree vector is  \emph{extremal} if it is a vertex of the Newton polytope of $F$. 
    \end{definition}

    \begin{example}\label{exmp: polynomial} Suppose $m=3$. Let  $\dot{x}:=\gamma(x)$ and $\ddot{x}:=\gamma^2(x)$. Set $N(x):=x\dot{x}\ddot{x}$ and $C(x):=x+\dot{x}+\ddot{x}$. Consider the polynomial \begin{align*}F&=N(x_1+\dot{x}_3)x_2+C(x_2^3\dot{x}_3)\\
&=N(x_1)x_2+C(x_1\dot{x}_1x_3)x_2+C(\dot{x}_1x_3\dot{x}_3)x_2+x_2N(x_3)+C(x_2^3\dot{x}_3)
    \end{align*}The degree vectors of the monomials appearing in the expansion of $F$ are 
    \[
(3,1,0),\quad (2,1,1),\quad (1,1,2),\quad (0,1,3),\quad (0,3,1).
    \]
    Each of $(3,1,0)$ and $(0,1,3)$ occurs for a single monomial, whereas each of $(2,1,1)$, $(1,1,2)$, and $(0,3,1)$ occurs for three monomials.
 The extremal degree vectors of $F$ are $(3,1,0)$, $(0,1,3)$, and $(0,3,1)$. Hence, the Newton polytope of $F$ is a triangle. 
    \end{example}

    \begin{lemma}\label{lem: non-zero coefficient} Let $F$ be a polynomial as in Definition \ref{def: polynomial}. If each extremal degree vector $d$ of $F$ corresponds to one and only one term in $F$, then for generic $f_1,\dots, f_n\in \mc{K}$,
    \[
    \val(F(f_1,\dots, f_n))=\min_{\text{extremal $d$}}\{d\cdot \nu\},
    \]
    where $\nu:=(\val(f_1),\dots, \val(f_n))$.
    \end{lemma}
    \begin{proof} Suppose for each index $i$, $f_i=u_it^{\nu_i}+\text{higher terms in $t$}$ with $u_i\neq 0$. Then $\gamma^n(f_i)=\epsilon^{-n\nu_i} u_it^{\nu_i}+\text{higher terms in $t$}$. 
    Let $T$ be a monomial appearing in $F$ with coefficient $c_T\neq 0\in \rk$ and degree vector $d$. 
    After plugging all the $f_i$'s, we get
    \[
    T(f_1,\dots, f_n)=c_T\epsilon^{n_T}u^dt^{d\cdot \nu}+\text{higher terms in $t$},
    \]
    where $n_T$ is an integer determined by which $\gamma$-powered variables occur in $T$ and their respective exponents, $u^d:=\prod_i u_i^{d_i}$, and $\nu=(\nu_1,\dots, \nu_n)$. 
    
    Now suppose that each extremal degree vector $d$ of $F$ corresponds to exactly one term $T_d$ in $F$. For any given tuple of valuations $\nu$, define $\omega:=\min_d\{d\cdot \nu\}$ where the minimum is taken over all degree vectors of $F$. Then the terms in $F(f_1,\dots, f_n)$ that possibly achieve this minimum valuation are of the form
    \[
    \left(\sum_{\substack{\text{extremal $d$} \\ d\cdot \nu=\omega}}c_{T_d}\epsilon^{n_{T_d}}u^d+\sum_{\substack{\text{non-extremal $d$} \\ d\cdot\nu=\omega}}
    \cdots \right)t^\omega.
    \]
    In particular, the first sum over extremal degrees $d$ is non-empty. Thus, for any given tuple of valuations $\nu$, $\sum_{\substack{\text{extremal $d$} \\ d\cdot \nu=\omega}}c_{T_d}\epsilon^{n_{T_d}}u^d+\sum_{\substack{\text{non-extremal $d$} \\ d\cdot\nu=\omega}}
    \cdots$ is a non-zero polynomial, and hence we can always find a generic $u_i$ for each $f_i$ such that $\val(F(f_1,\dots, f_n))=\omega$.
    \end{proof}

\medskip

    \textbf{(ii).} $c_{\imath\jmath}=-2$ and $c_{\jmath\imath}=-1$, which corresponds to the braid move $s_\imath s_\jmath s_\imath s_\jmath=s_\jmath s_\imath s_\jmath s_\imath$. 

    \begin{proposition}\label{prop: B2} If $y_\imath(a'_1)y_\jmath(a'_2)y_\imath(a'_3)y_\jmath(a'_4)=y_\jmath(a_1)y_\imath(a_2)y_\jmath(a_3)y_\imath(a_4)$, then there is a non-empty open subset in $\mathcal{A}^{\imath_\bullet}(n_\bullet)$, over which
    \[
    n'_1=n_2+2n_3+n_4-p, \quad n'_2=p-q, \quad n'_3=2q-p, \quad n'_4=n_1+n_2+n_3-q
    \]
    where
    \[
    p:=\min\{2n_1+n_2,2n_1+n_4,2n_3+n_4\} \quad \text{and} \quad q:=\min\{n_1+n_2,n_1+n_4,n_3+n_4\}.
    \]
    \end{proposition}
    \begin{proof} There are two possibilities: either $\imath=\{i,i'\}$ and $\jmath=\{j\}$ with $a_{ij}=a_{i'j}=-1$ in types $A_{2l-1}^{(2)}$, $D_{l+1}^{(2)}$, or $E_6^{(2)}$, or $\imath=\{l,l+1\}$ and $\jmath=\{l-1,l+2\}$ in type $A_{2l}^{(2)}$. 

    In the first case where $\imath=\{i,i'\}$ and $\jmath=\{j\}$ with $a_{ij}=a_{i'j}=-1$ in types $A_{2l-1}^{(2)}$, $D_{l+1}^{(2)}$, or $E_6^{(2)}$, it is necessary that $m=2$ and $\epsilon=-1$. For notation convenience, let us define 
    for $f\in \mc{K}$,
    \[
    \overline{f}:=\epsilon^{-\alpha_i(h)}\gamma(f) \quad \text{and} \quad N(f):=f\overline{f}\in \mc{K}^\gamma.
    \]
    Note that by construction, 
    \[
    y_\jmath(a_1)y_\imath(a_2)y_\jmath(a_3)y_\imath(a_4)=y_j(a_1)y_i(a_2)y_{i'}(\epsilon^{-\alpha_i(h)}\gamma(a_2))y_j(a_3)y_i(a_4)y_{i'}(\epsilon^{-\alpha_i(h)}\gamma(a_4)),
    \]
    where $a_1,a_3\in \mc{K}^\gamma=\rk\dbp{t^2}$. In order to get to $y_\imath(a'_1)y_\jmath(a'_2)y_\imath(a'_3)y_\jmath(a'_4)$, we can apply the following braid moves in $W$:
    \begin{align*}
    &s_js_is_{i'}s_js_is_{i'}=s_js_{i'}s_is_js_is_{i'}=s_js_{i'}s_js_is_js_{i'}=s_{i'}s_js_{i'}s_is_js_{i'}\\=&s_{i'}s_js_is_{i'}s_js_{i'}=s_{i'}s_js_is_js_{i'}s_j=s_{i'}s_is_js_is_{i'}s_j=s_is_{i'}s_js_is_{i'}s_j.
    \end{align*}
    By applying Eq. \ref{eq: Lusztig transform} repetitively, we get that
    \[
    a'_1=\frac{a_2a_3\overline{a}_4}{\overline{P}}, \quad a'_2=\frac{N(P)}{Q}, \quad a'_3=\frac{Q}{\overline{P}}, \quad a'_4=\frac{a_1a_3N(a_2)}{Q},
    \]
    where
    \[        P=a_1(a_2+a_4)+a_3a_4 \quad \text{and} \quad        Q=a_1N(a_2+a_4)+a_3N(a_4).
    \]
    In particular, $Q\in \mc{K}^\gamma$. Note that $\val(a_1)=2n_1$, $\val(a_2)=n_2$, $\val(a_3)=2n_3$, and $\val(a_4)=n_4$; over the open subset where the leading coefficients of the terms of $P$ and $Q$ do not cancel, we have
    \[
    \val(P)=\min\{2n_1+n_2,2n_1+n_4,2n_3+n_4\}=p,
    \]
    \[
    \val(Q)=2\min\{n_1+n_2,n_1+n_4,n_3+n_4\}=2q.
    \]
    Thus, we can conclude that
    \[
    n'_1=\val(a'_1)=n_2+2n_3+n_4-p, \quad \quad n'_2=\frac{1}{2}\val(a'_2)=p-q,
    \]
    \[
    n'_3=\val(a'_3)=2q-p, \quad \quad n'_4=\frac{1}{2}\val(a'_4)=n_1+n_2+n_3-q.
    \]


    In the second case where $\imath=\{l,l+1\}$ and $\jmath=\{l-1,l+2\}$ in type $A_{2l}^{(2)}$, the element $y_\imath$ is defined by Eq. \eqref{eq: root subgroup element for A2n}, and $\alpha_{l-1}(h)=\alpha_{l+2}(h)=0$. Note that regardless of whether $m=2$ or $4$, the actions of $\gamma$ on the $a_i$'s and $b_i$'s are all of order $2$; we therefore adopt the shorthand notation $\overline{x}:=\gamma(x)$ and define $N(x):=x\overline{x}$ for all $x\in \mc{K}^{\gamma^2}$; also, for notation simplicity, let us set $l=2$. 
    Now, the product $y_\jmath(a_1)y_\imath(b_2,c_2)y_\jmath(a_3)y_\imath(b_4,c_4)$ is equal to
    \[
    \left\{\begin{array}{ll}
       y_1(a_1)y_4(\overline{a}_1)y_2\left(\frac{\overline{a}_2}{b_2}\right)y_3(b_2)y_2\left(\frac{a_2}{b_2}\right)y_1(a_3)y_4(\overline{a}_3)y_2\left(\frac{\overline{a}_4}{b_4}\right)y_3(b_4)y_2\left(\frac{a_4}{b_4}\right)  & \text{if $m=2$}, \\
       y_1(a_1)y_4(\overline{a}_1)y_2\left(\frac{-\overline{a}_2}{tb_2}\right)y_3(tb_2)y_2\left(\frac{a_2}{tb_2}\right)y_1(a_3)y_4(\overline{a}_3)y_2\left(\frac{-\overline{a}_4}{tb_4}\right)y_3(tb_4)y_2\left(\frac{a_4}{tb_4}\right) & \text{if $m=4$}.
    \end{array}\right.
    \]
    In order to get to $y_\imath(b'_1,c'_1)y_\jmath(a'_2)y_\imath(b'
    _3,c'_3)y_\jmath(a'_4)$, we can apply the following braid moves in $W$:
    \begin{align*}
        &\sx\sz\so\sp\so\sx\sz\so\sp\so=\sx\so\sz\sp\sz\so\sx\so\sp\so=\sx\so\sp\sz\sp\sx\so\sx\sp\so\\
        =&\sx\so\sp\sz\sx\sp\so\sp\sx\so=\sx\so\sp\sz\sx\so\sp\so\sx\so=\sx\so\sx\sp\sz\so\sp\sx\so\sx\\
        =&\so\sx\so\sp\so\sz\sp\sx\so\sx=\so\sx\sp\so\sp\sz\sp\sx\so\sx=\so\sp\sx\so\sz\sp\sz\sx\so\sx\\
        =&\so\sp\sx\so\sx\sz\sp\so\sx\sz=\so\sp\so\sx\so\sz\sp\so\sx\sz=\so\sp\so\sx\sz\so\sp\so\sx\sz.
        \end{align*}
    By applying Eq. \eqref{eq: Lusztig transform} repetitively, we get that when $m=2$:
    \[
    a'_1=\frac{a_2N(a_3)a_4}{\overline{P}}, \quad b'_1=\frac{(a_1\overline{a}_4b_2+a_1a_2b_4+a_3\overline{a}_4b_2)\overline{a}_3}{P}, \quad a'_2=\frac{\overline{P}}{\overline{Q}},
    \]
    \[
    a'_3=\frac{N(Q)}{P}, \quad b'_3=\frac{(\overline{a}_1b_2+\overline{a}_1b_4+\overline{a}_3b_4)Q}{P}, \quad a'_4=\frac{\overline{a}_1a_2a_3}{\overline{Q}},
    \]
    where
    \[
    P=N(a_1)\overline{a}_2+N(a_1+a_3)\overline{a}_4{\color{blue}+a_1(\overline{a}_1+\overline{a}_3)\overline{b}_2b_4}\quad \text{and} \quad   Q=a_1\overline{a}_2+(a_1+a_3)\overline{a}_4{\color{blue}+a_1\overline{b}_2b_4}.
    \]
    Note that $\val(\overline{b}_2)\geq \dfrac{1}{2}\val(\overline{a}_2)$ and $\val(b_4)\geq \dfrac{1}{2}\val(\overline{a}_4)$. Thus, we have
    \begin{align*}    \val(a_1\overline{a}_1\overline{b}_2b_4)\geq &\val(N(a_1))+\frac{1}{2}\val(\overline{a}_2)+\frac{1}{2}\val(\overline{a}_4)\geq
    \min\{\val(N(a_1)\overline{a}_2), \val(N(a_1)\overline{a}_4)\},\\
    \val(a_1\overline{a}_3\overline{b}_2b_4)\geq &\frac{1}{2}\val(N(a_1)\overline{a}_2)+\frac{1}{2}\val(N(a_3)\overline{a}_4)
    \geq  \min\{\val(N(a_1)\overline{a}_2), \val(N(a_3)\overline{a}_4)\},\\
    %
    \val(a_1\overline{b}_2b_4)\geq&\val(a_1)+\frac{1}{2}\val(\overline{a}_2)+\frac{1}{2}\val(\overline{a}_4) \geq  \min\{\val(a_1\overline{a}_2),\val(a_1\overline{a}_4)\}.
    \end{align*}
    In other words, the part highlighted in blue in each of $P$ and $Q$ does not affect the minimum of the valuations of the terms in each of them. 

    When at least one of $n_2$ and $n_4$ is odd, the above inequalities are strict, and we can simply ignore the blue terms. We then observe that among the remaining terms, their Newton polytopes are triangles (2-simplices), and there is exactly one term for each extremal degree vector; therefore, by Lemma \ref{lem: non-zero coefficient}, we know that the minima of the valuations of the terms in $P$ and $Q$, respectively, are achieved generically.

    When both $n_2$ and $n_4$ are even, we know that $\val(a_2)=2\val(b_2)$ and $\val(a_4)=2\val(b_4)$. Thus, we may replace $\overline{a}_2$ by $b_2\overline{b}_2/2$ and replace $\overline{a}_4$ with $b_4\overline{b}_4/2$ in $P$ and $Q$ when computing their valuations. Now with respect to the variables $a_1$, $b_2$, $a_3$, and $b_4$, the Newton polytopes of these modified $P$ and $Q$ are still triangles, and there is still exactly one term for each extremal degree vector. In particular, these terms are in the black part of the expansion. Thus, by Lemma \ref{lem: non-zero coefficient}, we know that the minima of the valuations of the terms in $P$ and $Q$, respectively, are achieved generically. In conclusion, there is a non-empty open subset of $\mc{A}^{\imath_\bullet}(n_\bullet)$ over which
    
    %
    \begin{align*}
    \val(P)=&\min\{2n_1+n_2,2n_1+n_4, 2n_3+n_4\}=p,
    \\
    \val(Q)=&\min\{n_1+n_2,n_1+n_4,n_3+n_4\}=q.
    \end{align*}
    Plugging them back into the valuations of $a'_i$'s yields the desired formulas.

    When $m=4$, we get that
    \[
    a'_1=\frac{a_2N(a_3)a_4}{\overline{P}}, \quad b'_1=\frac{(a_1\overline{a}_4b_2-a_1a_2b_4+a_3\overline{a}_4b_2)\overline{a}_3}{P}, \quad a'_2=\frac{\overline{P}}{\overline{Q}}
    \]
    \[
    a'_3=\frac{N(Q)}{P}, \quad b'_3=\frac{(\overline{a}_1b_2+\overline{a}_1b_4+\overline{a}_3b_4)Q}{P}, \quad a'_4=\frac{\overline{a}_1a_2a_3}{\overline{Q}}
    \]
    where
    \[
    P=N(a_1)\overline{a}_2+N(a_1+a_3)\overline{a}_4{\color{blue}+a_1(\overline{a}_1+\overline{a}_3)\overline{b}_2b_4t^2} \quad \text{and} \quad Q=a_1\overline{a}_2+(a_1+a_3)\overline{a}_4{\color{blue}+a_1\overline{b}_2b_4t^2}.
    \]
    An argument similar to the $m=2$ case shows that there is a non-empty open subset of $\mathcal{A}^{\imath_\bullet}(n_\bullet)$ such that $\val(P)=2p$ and $\val(Q)=2q$. Note that they carry an extra factor $2$ because $\val(a_i)=2n_i$ for each $i$; but then we also need to divide $\val(a'_i)$ by $2$ to get to $n'_i$. Thus, we conclude the desired formulas again.
    \end{proof}

    A similar computation can be employed to prove the following reverse transformation.

    \begin{proposition}\label{prop:B2'} If $y_\jmath(a'_1)y_\imath(a'_2)y_\jmath(a'_3)y_\imath(a'_4)=y_\imath(a_1)y_\jmath(a_2)y_\imath(a_3)y_\jmath(a_4)$, then there is a non-empty open subset in $\mathcal{A}^{\imath_\bullet}(n_\bullet)$, over which
    \[
    n'_1=n_2+n_3+n_4-p, \quad n'_2=2p-q, \quad n'_3=q-p, \quad n'_4=n_1+2n_2+n_3-q,
    \]
    where 
    \[
    p:=\min\{n_1+n_2,n_1+n_4,n_3+n_4\} \quad \text{and} \quad q:=\min\{n_1+2n_2,n_1+2n_4,n_3+2n_4\}.
    \]
    \end{proposition}

    \medskip

    \textbf{(iii).} $c_{\imath\jmath}=-3$ and $c_{\jmath\imath}=-1$, which corresponds to the braid move $s_\imath s_\jmath s_\imath s_\jmath s_\imath s_\jmath=s_\jmath s_\imath s_\jmath s_\imath s_\jmath s_\imath$. 

    \begin{proposition}\label{prop: G2} If $y_\imath(a'_1)y_\jmath(a'_2)y_\imath(a'_3)y_\jmath(a'_4)y_\imath(a'_5)y_\jmath(a'_6)=y_\jmath(a_1)y_\imath(a_2)y_\jmath(a_3)y_\imath(a_4)y_\jmath(a_5)y_\imath(a_6)$,
    then there is a non-empty open subset of $\mathcal{A}^{\imath_\bullet}(n_\bullet)$, over which
    \[
    n'_1=n_2+3n_3+2n_4+3n_5+n_6-p, \quad n'_2=p-q, \quad n'_3=3q-p-r,
    \]
    \[
    n'_4=r-q-s, \quad n'_5=3s-r, \quad n'_6=n_1+n_2+2n_3+n_4+n_5-s,
    \]
    where
    \begin{align*}
    p:=&\min\left\{\begin{array}{c}3n_1+2n_2+3n_5+n_6,\quad 3n_1+2n_4+3n_5+n_6, \quad 3n_3+2n_4+3n_5+n_6,\\
    3n_1+2n_2+3n_3+n_4, \quad 3n_1+2n_2+3n_3+n_6
    \end{array}\right\},\\
    q:=&\min\left\{\begin{array}{c}2n_1+2n_2+3n_3+n_4, \quad 2n_1+2n_2+3n_3+
    n_6, \quad 2n_1+2n_2+3n_5+n_6, \\
    2n_1+2n_4+3n_5+n_6, \quad 2n_3+2n_4+3n_5+n_6\end{array}\right\},\\
    r:=&\min\left\{\begin{array}{c}3n_1+3n_2+6n_3+2n_4, \quad 3n_1+3n_2+6n_3+2n_6, \quad 3n_1+3n_2+6n_5+2n_6,\\
    3n_1+3n_4+6n_5+2n_6, \quad 3n_3+3n_4+6n_5+2n_6\end{array}\right\},\\
    s:=&\min\left\{\begin{array}{c}n_1+n_2+2n_3+n_4, \quad n_1+n_2+2n_3+n_6, \quad n_1+n_2+2n_5+n_6,\\
    n_1+n_4+2n_5+n_6,\quad n_3+n_4+2n_5+n_6
    \end{array}\right\}.
    \end{align*}
    \end{proposition}
    \begin{proof} There is only one case to check: $D_4^{(3)}$. In this case, the field automorphism $\gamma$ on $\mc{K}$ is of order $3$. Let $\imath=\{1,3,4\}$ be the $\tau$-orbit of size $3$ and let $\jmath=\{2\}$ be the $\tau$-fixed vertex. Let us adopt the notations $\dot{x}$, $\ddot{x}$, $N(x)$, and $C(x)$ in Example \ref{exmp: polynomial}. Then
    \begin{align*}
        &y_\jmath(a_1)y_\imath(a_2)y_\jmath(a_3)y_\imath(a_4)y_\jmath(a_5)y_\imath(a_6)\\
        =&y_2(a_1)y_1(a_2)y_3(\dot{a}_2)y_4(\ddot{a}_2)y_2(a_3)y_1(a_4)y_3(\dot{a}_4)y_4(\ddot{a}_4)y_2(a_5)y_1(a_6)y_3(\dot{a}_6)y_4(\ddot{a}_6),
    \end{align*}
    with $a_1,a_3,a_5\in \mc{K}^\gamma=k\dbp{t^3}$ and $a_2,a_4,a_6\in \mc{K}$. In order to get to $y_\imath(a'_1)y_\jmath(a'_2)y_\imath(a'_3)y_\jmath(a'_4)y_\imath(a'_5)y_\jmath(a'_6)$, we can apply the following braid moves in $W$:
    \begin{align*}
        &\so\sx\sy\sz\so\sx\sy\sz\so\sx\sy\sz=\so\sx\sy\sz\so\sz\sy\sx\so\sx\sy\sz=\so\sx\sy\so\sz\so\sy\so\sx\so\sy\sz\\
        =&\so\sx\sy\so\sz\sy\so\sy\sx\so\sy\sz=\so\sx\sy\so\sy\sz\so\sx\sy\so\sy\sz=\so\sx\so\sy\so\sz\so\sx\so\sy\so\sz\\
        =&\sx\so\sx\sy\sz\so\sz\sx\so\sy\so\sz=\sx\so\sy\sz\sx\so\sx\sz\so\sy\so\sz=\sx\so\sy\sz\so\sx\so\sz\so\sy\so\sz\\
        =&\sx\so\sy\sz\so\sx\sz\so\sz\sy\so\sz=\sx\so\sy\sz\so\sz\sx\so\sy\sz\so\sz=\sx\so\sy\so\sz\so\sx\so\sy\so\sz\so\\
        =&\sx\sy\so\sy\sz\so\sx\sy\so\sy\sz\so=\sx\sy\so\sz\sy\so\sy\sx\so\sy\sz\so=\sx\sy\so\sz\so\sy\so\sx\so\sy\sz\so\\
        =&\sx\sy\sz\so\sz\sy\sx\so\sx\sy\sz\so=\sx\sy\sz\so\sx\sy\sz\so\sx\sy\sz\so.
    \end{align*}
    By applying Eq. \eqref{eq: Lusztig transform} repetitively, we get
    \[
    a'_1=\frac{a_2a_3\dot{a}_4\ddot{a}_4a_5a_6}{P}, \quad a'_2=\frac{N(P)}{Q},\quad a'_3=\frac{Q}{PR}, \quad a'_4=\frac{N(R)}{QS} \quad a'_5=\frac{S}{R}, \quad a'_6=\frac{a_1a_3^2a_5N(a_2a_4)}{S}
    \]
    where
    \begin{align*}
    P=&a_3\dot{a}_4\ddot{a}_4a_5a_6+a_1(\dot{a}_2+\dot{a}_4)(\ddot{a}_2+\ddot{a}_4)a_5a_6+a_1\dot{a}_2\ddot{a}_2a_3(a_4+a_6),\\
    Q=&a_1^2N\left(\dot{a}_2\ddot{a}_2a_3(a_4+a_6)+(\dot{a}_2+\dot{a}_4)(\ddot{a}_2+\ddot{a}_4)a_5a_6\right)+a_3^2a_5^3N(a_4^2a_6){\color{blue}+2a_1a_3a_5^3N(a_2+a_4)N(a_4a_6)}\\
    &{\color{blue}+a_1a_3^2a_5^2N(a_4)\left(C(a_2\dot{a}_2\ddot{a}_4^2a_6\dot{a}_6)+N(a_2)C(a_4\dot{a}_6\ddot{a}_6)+C(a_2\dot{a}_2\ddot{a}_4)N(a_6)+2N(a_2a_6)\right)},\\
    R=&a_1a_3^2N(a_2)(\dot{a}_4+\dot{a}_6)(\ddot{a}_4+\ddot{a}_6)+a_1a_5^2\dot{a}_6\ddot{a}_6N(a_2+a_4)+a_3a_5^2\dot{a}_6\ddot{a}_6N(a_4)\\
    &{\color{blue}+a_1a_3a_5\left(N(a_2)(\dot{a}_4\ddot{a}_6+\dot{a}_6\ddot{a}_4+2\dot{a}_6\ddot{a}_6)+\dot{a}_6\ddot{a}_6C(a_2\ddot{a}_2\dot{a}_4)+a_2(\ddot{a}_2\dot{a}_4^2\ddot{a}_6+\dot{a}_2\ddot{a}_4^2\dot{a}_6)\right)},\\
    S=&a_1a_3^2N(a_2)N(a_4+a_6)+a_1a_5^2N(a_2+a_4)N(a_6)+a_3a_5^2N(a_4a_6)\\
    &{\color{blue}+a_1a_3a_5\left(C(a_2\dot{a}_2\ddot{a}_4^2a_6\dot{a}_6)+N(a_2)C(a_4\dot{a}_6\ddot{a}_6)+C(a_2\dot{a}_2\ddot{a}_4)N(a_6)+2N(a_2a_6)\right)}.
    \end{align*}

    By using a computer program, we see that the Newton polytope of each of $P,Q,R,S$ is a $4$-dimensional simplex; we highlight the terms whose degree vectors are not extremal in blue above. It is not hard to see that for each of $P,Q,R,S$ and for each extremal degree vector $d$ within, there is exactly one term, and it has coefficient $1$. Thus, by Lemma \ref{lem: non-zero coefficient}, we can conclude that 
    \[
    \val(P)=p, \quad \val(Q)=3q, \quad \val(R)=r, \quad \val(S)=3s.
    \]
    (Note that $\val(a_i)=3n_i$ when $i$ is odd and $\val(a_i)=n_i$ when $i$ is even.) Plugging them into the rational expressions of the $a'_i$'s yields the desired formulas.
    \end{proof}

    A similar computation can be employed to prove the following reverse transformation.

    \begin{proposition}\label{prop:G2'} If $y_\jmath(a'_1)y_\imath(a'_2)y_\jmath(a'_3)y_\imath(a'_4)y_\jmath(a'_5)y_\imath(a'_6)=y_\imath(a_1)y_\jmath(a_2)y_\imath(a_3)y_\jmath(a_4)y_\imath(a_5)y_\jmath(a_6)$,
    then there is a non-empty open subset of $\mathcal{A}^{\imath_\bullet}(n_\bullet)$, over which
    \[
    n'_1=n_2+n_3+2n_4+n_5+n_6-p, \quad n'_2=3p-q, \quad n'_3=q-p-r,
    \]
    \[
    n'_4=3r-q-s, \quad n'_5=s-r, \quad n'_6=n_1+3n_2+2n_3+3n_4+n_5-s,
    \]
    where
    \begin{align*}
    p:=&\min\left\{\begin{array}{c}n_3+2n_4+n_5+n_6, \quad n_1+2n_4+n_5+n_6, \quad n_1+2n_2+n_5+n_6,\\
    n_1+2n_2+n_3+n_4,\quad n_1+2n_2+n_3+n_6
    \end{array}\right\},\\
    q:=&\min\left\{\begin{array}{c}2n_3+6n_4+3n_5+3n_6, \quad 2n_1+6n_2+3n_5+3n_6, \quad 2n_1+6n_4+3n_5+3n_6,\\
    2n_1+6n_2+3n_3+3n_4, \quad 2n_1+6n_2+3n_3+3n_6\end{array}\right\},\\
    r:=&\min\left\{\begin{array}{c}n_1+3n_2+2n_3+2n_4, \quad n_1+3n_2+2n_3+2n_6, \quad n_1+3n_2+2n_5+2n_6, \\
    n_1+3n_4+2n_5+2n_6, \quad n_3+2n_5+3n_4+2n_6\end{array}\right\},\\
    s:=&\min\left\{\begin{array}{c}n_1+3n_2+2n_3+3n_4, \quad n_1+3n_2+2n_3+3n_6, \quad n_1+3n_2+2n_5+3n_6,\\
    n_1+3n_4+2n_5+3n_6, \quad n_3+3n_4+2n_5+3n_6
    \end{array}\right\}.
    \end{align*}
    \end{proposition}

    \begin{theorem}
    \label{thm:main-MV}
    Suppose $\imath_\bullet$ and $\imath'_\bullet$ differ by a single braid move, and $n_\bullet$ and $n'_\bullet$ satisfy the respective Lusztig transformation, i.e. the formulas listed in Propositions \ref{prop: A1xA1}, \ref{prop: A2}, \ref{prop: B2}, \ref{prop:B2'}, \ref{prop: G2}, and \ref{prop:G2'}, then $A^{\imath_\bullet}(n_\bullet)$ and $A^{\imath'_\bullet}(n'_\bullet)$ have a non-empty intersection that is dense in each of them. 
    \end{theorem}
    \begin{proof} From the transformation formulas we see that as long as the expressions in the denominators of the arguments inside $y_{\imath'}$'s are non-zero, the product $y=y_{\imath_1}y_{\imath_2}\cdots y_{\imath_N}$ can be factored as $y_{\imath'_1}y_{\imath'_2}\cdots y_{\imath'_N}$ as well. Thus, there is a non-empty open subset of $\mc{A}^{\imath_\bullet}(n_\bullet)$ whose image $[g]=[\eta_{w_0}^{-1}(y)]$ lies inside $A^{\imath'_\bullet}(n'_\bullet)$. Since the map $\Phi:\mc{A}^{\imath_\bullet}(n_\bullet)\rightarrow A^{\imath_\bullet}(n_\bullet)$ is surjective (Theorem \ref{prop: surjective}), it follows that the intersection $A^{\imath_\bullet}(n_\bullet)\cap A^{\imath'_\bullet}(n'_\bullet)$ is dense inside $A^{\imath_\bullet}(n_\bullet)$. The same argument shows that the intersection is dense inside $A^{\imath'_\bullet}(n'_\bullet)$. 
    \end{proof}

For any two reduced expressions $\imath$ and $\imath'$, it is well-known that they are related by finite steps of braid moves.  If $(\imath_\bullet, n_\bullet)$ and $(\imath', n'_\bullet)$ are related by a sequence of Lusztig transforms, we write 
\[ n'_\bullet=R_{\imath_\bullet}^{ \imath'_\bullet}(n_\bullet)    .\]
We naturally have the following result.
    \begin{theorem}\label{cor_MV_Lusztig_transf}
    \begin{enumerate}
\item Fix a reduced expression $\imath_\bullet$ of $w_0\in W^\tau$,  there is a one-to-one correspondence between sequence $n_\bullet$ of non-negative integers such that $\sum_s n_s \beta_s^{\imath_\bullet}=\mu$,  and irreducible components of $\overline{S_\mu\cap T_0}$, sending $n_\bullet$ to $\overline{A^{\imath_\bullet}(n_\bullet)}$
 \item  Given any two Lusztig data $(\imath_\bullet, n_\bullet)$ and $(\imath'_\bullet, n'_\bullet)$, 
        if $n'_\bullet=R_{\imath_\bullet}^{ \imath'_\bullet}(n_\bullet) $, 
   then
        $A^{\imath_\bullet}(n_\bullet)\cap A^{\imath'_\bullet}(n'_\bullet)$ contains a dense subset in each of $A^{\imath_\bullet}(n_\bullet)$ and $A^{\imath'_\bullet}(n'_\bullet)$, and 
        $\overline{A^{\imath_\bullet}(n_\bullet)}=\overline{A^{\imath'_\bullet}(n'_\bullet)}$. 
     
   \end{enumerate}
   \end{theorem}
\begin{proof}
Part 1) follows from Theorem \ref{thm_MV_constr}, and part 2) follows from Theorem \ref{thm:main-MV}.    
\end{proof}
By Theorem \ref{cor_MV_Lusztig_transf}, the vertices of \(P(A^{\iota_\bullet}(n_\bullet))\) along the path determined by a reduced word \(\iota_\bullet\) are given by the Lusztig datum \(n_\bullet\). Under a change of reduced word, Theorem \ref{thm:main-MV} shows that these data transform by the usual Lusztig transformations for the reductive group $(G^\vee)^\tau$ arising from ramified geometric Satake. Hence the resulting collection of vertices satisfies the usual Berenstein–Zelevinsky relations and determines the ordinary MV polytope associated with \(n_\bullet\). Consequently,
$P(A^{\iota_\bullet}(n_\bullet))$
is precisely the usual MV polytope of weight \((\mu,0)\).


\bibliographystyle{alphaurl-a}

\bibliography{biblio}

\end{document}